\documentclass[oneside,english]{amsart}

\usepackage[top=2.7cm, bottom=3.7cm, right=2.2cm, left=2.2cm]{geometry}

\usepackage{color}

\usepackage{cite}

\usepackage[T1]{fontenc}
\usepackage[latin9]{inputenc}
\usepackage{amsbsy}
\usepackage{amstext}
\usepackage{amsthm}
\usepackage{amssymb,enumerate}
\usepackage[colorlinks=true]{hyperref}
\hypersetup{
	colorlinks=blue,%
	citecolor=red,%
	filecolor=black,%
	linkcolor=blue,%
	urlcolor=blue
}

\makeatletter
\newtheorem{thm}{Theorem}[section]
\newtheorem{prop}{Proposition}[section]
\newtheorem{lem}{Lemma}[section]
\newtheorem{cor}{Corollary}[section]

\newtheorem{remark}{\textbf{Remark}}[section]

\providecommand{\definitionname}{Definition}
\theoremstyle{plain}

\newcommand{\eq}[1]{\begin{equation}\allowdisplaybreaks\begin{alignedat}{2} #1 \end{alignedat}\end{equation}}

 \def\ve{\varepsilon}
 \def \Rn {\mathbb{R}^n}
 
  \newcommand{\R}{\mathbb{R}}

\newcommand{\Sn}{{\mathbb{S}^n}}

 \def\s{\sigma}
 \def\n{{\nabla}}
 
\def\ss{\mathbb{S}}

\def\a{\alpha}

\def\s{\sigma}

\def\C{\mathcal{C}}

\def\<{\langle}
\def\>{\rangle}
\def\div{{\rm div}}
\def\n{\nabla}

\def\R{{\mathbb R}}

\def \ds{\displaystyle}

\def \intsn {\int_{\mathbb{S}^n}}
\def \gSn{g_{\mathbb{S}^n}}

\def\triangle {\Delta}

\def\circledwedge{\setbox0=\hbox{$\bigcirc$}\relax \mathbin {\hbox
to0pt{\raise.5pt\hbox to\wd0{\hfil $\wedge$\hfil}\hss}\box0 }}
 
\makeatother

\usepackage{babel}

\allowdisplaybreaks

\numberwithin{equation} {section}

\begin{document}
\title[A new Yamabe Problem]{A Yamabe problem for the quotient between the $Q$ curvature and the scalar curvature}
\author{Yuxin Ge}
\address{Institut de Math\'ematiques de Toulouse,\\
Universit\'e Paul Sabatier,\\
118, route de Narbonne,\\
31062 Toulouse Cedex, France}
\email{yge@math.univ-toulouse.fr}

\author{Guofang Wang}
\address{Albert-Ludwigs-Universit\"at Freiburg,
Mathematisches Institut,
Eckerstr. 1,
D-79104 Freiburg, Germany}
\email{guofang.wang@math.uni-freiburg.de}

\author{Wei Wei}
\address{Nanjing University, School of Mathematics, Nanjing 210093, P.R. China}
\email{wei\_wei@nju.edu.cn}

\begin{abstract}
In this paper we introduce the following Yamabe problem for
 the quotient between the $Q$ curvature and the scalar curvature $R$: {\it Find a conformal metric $g$ in a given conformal class $[g_0]$ with}
 \[
 Q_g\slash R_g=const.
 \]
When the dimension $n\ge 5$, we first prove a new  Sobolev inequality between the total $Q$-curvature and the total scalar curvature on $\ss^n$ ($n\ge 5$), namely
		\[
		\frac {{\intsn} Q_g dv_g}{ ({\intsn} R_g dv_g)^{\frac {n-4}{n-2} }} \ge 
		\frac {{\intsn} Q_{g_{\mathbb{S}^n}}  dv(\gSn )} {({\intsn} R_{\gSn}  dv(\gSn) )^{\frac {n-4}{n-2} }},
		\]
        for any $g$ in the conformal class of the round metric $\gSn$ with positive scalar curvature, 
        with equality if and only if $g$ is also a metric with constant sectional curvature. With this inequality we introduce a new Yamabe constant $Y_{4,2}(M,[g_0])$ and prove the existence of the above problem provided that $Y_{4,2}(M,[g_0]) <Y_{4,2} (\ss^n, [g_{\ss^n}]).$ This strict inequality is proved if $(M,g)$ is not conformally equivalent to the round  sphere. This follows from  a crucial relation between $Y_{4,2}$ and the ordinary Yamabe constant $Y(M,[g_0])$, $Y_{4,2} (M, [g_0]) \le c(n) Y(M, [g_0])^{\frac  n{n-2}}$  with equality if and only if  $(M, g_0)$ is conformally equivalent to an Einstein manifold.
Finally, we prove that
on a closed $n$-dimensional  Riemannian manifold $(M,g_{0})$ with semi-positive $Q$-curvature
and non-negative scalar curvature, the above Yamabe problem is solvable, thanks to the maximum principle of Gursky-Malchiodi \cite{GM}. The proof for $n=3$ and $n=4$ follows closely the methods developed by Hang-Yang in \cite{HYCPAM3-dim},  
 Gursky-Malchiodi in \cite{GM}, and Chang-Yang in \cite{ChangYang1995}. 
\end{abstract}

\maketitle

\vspace{-3mm}
\tableofcontents

\section{Introduction}

Let $(M,g_{0})$ be an $n$-dimensional closed manifold, i.e, a compact manifold without boundary. The celebrated Yamabe problem asks if there is a metric $g$ in the conformal class $[g_0]$ of $g_0$  with constant scalar curvature, i.e., 
\eq{\label{Y} R_g=const.}
The Yamabe probem has been affirmatively solved through the work of Yamabe, Trudinger, Aubin and Schoen \cite{Yamabe,Trudinger, Aubin, Schoen}. See the nice survey of Lee-Parker \cite{Lee-Parker}. This problem plays a very  important role in differential geometry and geometric analysis.  More importantly,
analytical methods developed to attack this problem have been widely used in other problems.

There are several natural generalizations of the Yamabe problem in conformal geometry. In this paper we are interested in two classes of Yamabe problems: one is the $\sigma_k$-Yamabe problem, a fully nonlinear counterpart,  and the other is the Yamabe problem for the $Q$ curvature, a higher order generalization, which will also be called $Q$-Yamabe problem in the paper. The $\sigma_k$-Yamabe problem was first studied by Viaclovsky \cite{ViacDuke}. Since then there have been a lot of work 
by Chang-Gursky-Yang \cite{CGYAnn}, Guan-Wang \cite{GWCrelle,GWDuke}, Gursky-Viaclovsky \cite{Gursky-Viaclovsky}, Li-Li \cite{LLActa,LLCPAM}, Sheng-Wang-Trudinger\cite{STWJDG}, Wang-Trudinger \cite{Wang-Trudinger}, Ge-Wang \cite{GWAdv,GWCAG,GWEcole}, and Li-Nguyen \cite{LNJFA2014} and others. For the definition of $\sigma_k$ curvature and more on the $\sigma_k$-Yamabe problem, see the references mentioned above or Subsection \ref{subsection:s_k Yamabe problem} and \ref{s2 and s1} below. The $Q$ curvature was introduced by Branson \cite{Branson}, as a higher order generalization of the scalar curvature. The corresponding Yamabe problem has been studied by many mathematicians, just to mention a few, Van der Vorst \cite{Van}, Djadli-Hebey-Ledoux \cite{DHL}, Esposito-Robert \cite{ER}, Chang-Yang \cite{ChangYang1995}, Xu-Yang \cite{Xu-Yang}, Qing-Raske \cite{QR}, Gursky-Malchiodi \cite{GM},  Hang-Yang \cite{HY,HYCPAM3-dim, HYIMRN2015} and  Gursky-Hang-Lin \cite{GHL}. One of the breakthroughs in the study of the $Q$ curvature Yamabe problem is the maximum principle discovered by Gursky-Malchiodi \cite {GM}, which will also play a crucial role in this paper. See also a recent survey of Malchiodi
\cite{Malchiodi} and references therein.  For the definition of $Q$ curvature and its generalization, see Subsection \ref{Q Yamabe} below.

In this paper we consider the following Yamabe problem.

\

\noindent{\bf New Yamabe Problem}: {\it Let $(M, g_0)$ be a closed manifold. Is there a metric $g$ in $[g_0]$ such that its $Q$ curvature $Q_g$ is proportional to its scalar curvature $R_g$, i.e.,
\eq{\label{eq_QR}
Q_g\slash R_g=c,}
for a constant $c$?}

\

We give an affirmative answer if $g_0$ has positive scalar curvature and semi-positive $Q$ curvature.

\begin{thm}\label{mainthm} Let $n\ge 3$ and $(M, g_0)$ be a closed n-dimensional manifold such that $Q_{g_{0}}$ is semi-positive, namely $Q_{g_0}$ is nonnegative and somewhere positive,    and $R_{g_{0}}$ is non-negative.
Then there exists a conformal metric 
$g\in [g_0]$
with
a constant quotient $Q_{g}/R_{g}$.
\end{thm}

The problem is  variational: a conformal metric with a constant quotient between the $Q$ curvature and the scalar curvature $R$. i.e., $Q\slash R=const$, is a critical point of
\[I(g) :=\frac{\int _M Q_g dv_g}{(\int_M R_g dv_g)^{\frac {n-4}{n-2}}}.
\]
As in the ordinary Yamabe problem, we first define
a Yamabe type constant for $n\ge 5$
\[
Y_{4,2} (M, [g_0]) =\frac{n-4}{2}\inf_{R_g>0, Q_g\ge 0}  \frac {\int_M Q_g dv_g} {(\int _M R_g dv_g)^{\frac {n-4}{n-2}} }
\]
where as a convention in this paper we use $Q\ge 0$ to denote $Q$ semi-positive, namely $Q$ is nonnegative and somewhere positive.

We will show a new optimal Sobolev inequality in Theorem \ref{thm1.1} in Section \ref{section:sobolev inequality} below, which gives
\eq{\label{eq2.22}
Y_{4,2}(\Sn) =  \frac{(n-4)(n+2)}{n(n-2)^{\frac{2}{n-2}}(4(n-1))^{\frac{n-4}{n-2}}} (Y(\Sn))^{\frac{n}{n-2}}=\frac{n-4}{16}(n^2-4)n^{\frac{2}{n-2}}(n-1)^{\frac{4-n}{n-2}}\omega_n^{\frac{2}{n-2}}
.
}
All minimizers are classified in Theorem \ref{thm1.1} below, with which one can show that
\[
Y_{4,2} (M, [g_0])\le Y_{4,2}(\Sn)
\]
for any conformal class $[g_0]$ with a positive Yamabe constant $Y(M,[g_0])$ using a gluing argument. However, we need not this result and hence omit its proof. Instead we will show in Theorem \ref{thm_Yamabe_Constants}
an inequality between $Y_{4,2}$ and $Y(M,[g_0])$ 
\eq{ \label{eq2.23}
  Y_{4,2}(M,[g_{0}])\le \frac{(n-4)(n+2)}{n(n-2)^{\frac{2}{n-2}}(4(n-1))^{\frac{n-4}{n-2}}}   (Y(M,[g_0]))^{\frac{n}{n-2}},
}
provided that $(M,g_{0})$  is an $n$-dimensional  closed
manifold with nonnegative $R_{g_0}$ and semi-positive $Q_{g_{0}}$. We remark that in the proof of \eqref{eq2.23} the maximum principle of Gursky-Malchiodi plays a role. Formulas \eqref{eq2.22} and  \eqref{eq2.23}, together with \eqref{Y1_ineq}, imply clearly
 \eq{\label{eq_2.24}
 Y_{4,2}([g_0]) < Y_{4,2}(\Sn),
 }
 provided that $(M,g_0)$ has non-negative scalar curvature $R_{g_0}$ and semi-positive $Q_{g_0}$ curvature and $(M, g_0)$ is not conformally equivalent to $\Sn$. This is the second step in the usual approach to Yamabe type problems.
Now we remain to show that minimizers of \eqref{eq_functional} exist if  \eqref{eq_2.24} holds. This will be done in Section \ref{section:5+dim} below using a nonlocal flow inspired by the work of Gursky-Malchiodi \cite{GM}
and the blow-up analysis

The proof of Theorem \ref{mainthm} relies on a lot of previous work of many mathematicians in  the $Q$-Yamabe problem, as well as on the work in the $\s_k$-Yamabe problem, which we will review in Section \ref{section:various yamabe problems}.

\

\noindent{\it The  rest of the paper is organized as follows.} In Section \ref{section:various yamabe problems}, we review the  Yamabe problem and its generalizations: the $\sigma_k$-Yamabe problem, the Yamabe problem for the $Q$ curvature, and their corresponding equations, conformal operators and their properties. The new Yamabe problem is introduced in Subsection \ref{new problem}, together with  a sketch of ideas of the proof. 
In Section \ref{section:sobolev inequality}  we establish a new  optimal Sobolev inequality in $\ss^n$ between the total $Q$ curvature and the total scalar curvature, together with its optimizers,  which gives the precise value of 
$Y_{4,2} (\ss^n) $, the $Q\slash R$ Yamabe constant for the standard sphere $\ss^n$.  In fact, it is this optimal inequality that 
leads us to consider the new Yamabe problem. 
Moreover we obtain an Obata type result for Equation \eqref{eq_QR}. In Section \ref{Various Yamabe constants} we establish relationship between various
Yamabe type constants and prove that $Y_{4,2} (M,[g_0])<Y_{4,2}(\Sn,[g_{\Sn}])$, if $(M, g_0)$ is not conformally equivalent to $\Sn$.
Especially we also prove optimal inequalities between  all previous Yamabe constants for the $Q$-Yamabe problem in \cite{GM} and in \cite{HY} and the ordinary Yamabe constant $Y$. Hence, the desired strict inequality $Y_{4,2}(M,[g]) < Y_{4,2}(\mathbb{S}^n)$ 
follows from the resolution of the celebrated Yamabe problem \eqref{Y}. Section \ref{section:3-dim} and Section \ref{section:dim4} devote to the proof of Theorem \ref{mainthm} in 3 dimensions and 4 dimensions respectively. In Section \ref{section:5+dim}, we prove Theorem \ref{mainthm} for $n\ge 5.$ In Appendix we provide a proof of the regularity of $W^{2,2}$ weak solutions.

\

\section{Various Yamabe problems}\label{section:various yamabe problems}
Let $n\ge 3$ and
 $(M, g_0)$ be a closed $n$-dimensional manifold and $[g_0]$ the conformal class of $g_0$. For any Riemannian metric $g$, let $R_g$ be the scalar curvature of $g$. The relation between two scalar curvature $R_g$ and $R_{g_0}$ with $g=u^{\frac 4 {n-2}}g_0$ is given by
 \eq{\label{eq_scalar}
 \frac {n-2}{4(n-1)} R_g u^{\frac {n+2}{n-2}}=-{\Delta_{g_0}} u+ \frac {n-2}{4(n-1)} R_{g_0} u.
 }
 The Yamabe problem is to ask if there is a conformal metric $g=u^{\frac 4{n-2}}g_0\in[g_0]$ with a constant scalar curvature $R_g$,
 which is equivalent to solving the following equation
\begin{equation}
    \label{eq1.1}
    L_{g_0} u=cu^{\frac{n+2}{n-2}}, \quad u>0,
\end{equation}
and
$L_{g_0}$ is the so-called conformal Laplacian (or the Yamabe operator) w.r.t. $g_0$
\begin{equation}
\label{eq1.2}
    L_{g_0} u=-{\Delta_{g_0}} u+\frac{n-2}{4(n-1)}R_{g_{0}}u.
\end{equation}
Equation \eqref{eq1.1} is a critical semilinear equation. The  conformal Laplacian is conformally invariant in the following sense.
For any given $\varphi\in C^{\infty}(M),$ we have 
\begin{equation}\label{conformal laplacian transform}
L_{u^{\frac{4}{n-2}}g_0}\varphi=u^{-\frac{n+2}{n-2}}L_{g_0}(u\varphi).
\end{equation}
The following Yamabe constant
\begin{equation}
    \label{Y_2}
    Y(M, [g_0]) :=\frac{n-2}{4(n-1)}\inf_{g\in [g_0]} \frac {\int_MR_g dv_g}{vol(g)^{\frac {n-2}n}}= \inf_{u>0}\frac{\int_M uL_{g_0}u  dv_{g_0}
    }{(\int_M u^{\frac {2n}{n-2}} dv_{g_0})^{\frac {n-2}n}} 
\end{equation}
and the optimal Sobolev inequality
\begin{equation}
    \label{Sobolev1}
    Y(\mathbb{S}^n)=  \inf_{u>0}\frac{\int_\Sn  |\n u|^2 + \frac 14  n(n-2) u^2
    }{(\int_\Sn u^{\frac {2n}{n-2}})^{\frac {n-2}n}} =\frac{n(n-2)}{4}\omega^{\frac 2n}
\end{equation}
play an important role in the resolution of the Yamabe problem. Here $\omega_n$ is the area of the unit sphere $\Sn$. See the survey \cite{Lee-Parker}.
The Yamabe problem was completely solved by  \cite{Yamabe, Trudinger, Aubin,Schoen}
through
two steps: 
\begin{enumerate}
    \item  $Y(M, [g_0])\le Y(\Sn)$ for any given $[g_0]$, and if
    $Y(M,[g_0])< Y(\Sn)$, then $Y(M, [g_0])$ is achieved.
    \item   
    if $(M, g)$ is not conformally equivalent to $\Sn$, then we have  \begin{equation}\label{Y1_ineq}
        Y(M, [g_0])< Y(\Sn).
        \end{equation}
\end{enumerate}
The second step is more delicate and proved by Aubin \cite{Aubin} and Schoen \cite{Schoen}.

\subsection{\texorpdfstring{$\s_k$}{sk}-Yamabe problem}\label{subsection:s_k Yamabe problem}

 For any Riemannian metric $g$, let
\[
A_g=\frac 1{n-2} \left( Ric_g-\frac {R_g}{2(n-1)} g\right)
\]
be the Schouten tensor of $g$. Here $R_g$ and $Ric_g$
are the scalar curvature and the Ricci curvature of $g$ respectively. The $\sigma_k$ scalar curvature, introduced by Viaclovsky in \cite{ViacDuke}, is defined
by
\[\sigma_k(g) =\sigma_k(g^{-1}\cdot A_g),
\]
where $\sigma_k$ is the $k$-th elementary symmetric function. It is easy to see that $$\s_1(g)=\frac 1{2(n-1)} R_g. $$ Therefore,  $\sigma_k$ scalar curvature is a generalization of the scalar curvature $R_g$. The $\s_k$-Yamabe problem is to ask if there is a conformal metric $g\in[g_0]$ with a constant $\s_k(g)$. If we set $g=e^{-2u} g_0$, then the corresponding equation is
\begin{equation}\label{eq1.4}
    \s_k\left(\n^2 u +\n u \circledwedge \n u -\frac {|\n u|^2} 2 g_0+A_{g_0}\right) = ce^{-2ku}.
\end{equation}
When $k=1$, it is clear that it is equivalent to the ordinary Yamabe equation, which is a semilinear critical equation.
When $k\ge 2$, \eqref{eq1.4} becomes  a fully nonlinear equation. Hence for the $\s_k$-Yamabe problem with $k\ge 2$ one needs certain assumptions to guarantee the ellipticity of \eqref{eq1.4}. A common one in the field of fully nonlinear equations is   an assumption on the given conformal class that
\begin{equation} \label{elliptic}
\mathcal{C}_k[g_0]:=\{g\in [g_0] | g^{-1}\cdot A_g \in \Gamma_k^+\} \not = \emptyset,
\end{equation}
where 
\[
\Gamma_k^+ : =\{\Lambda \in \Rn | \s_j(\Lambda)>0, \forall j\le k\} 
\]
is the famous Garding cone. Condition \eqref{elliptic} guarantees that \eqref{eq1.4} is elliptic.
This condition has been assumed in all previous work on the $\s_k$-Yamabe problem mentioned in the Introduction. In our recent work \cite{GWW} we studied \eqref {eq1.4} in a larger cone 
$\mathcal{C}_{k-1} [g_0] \not =\emptyset$. See the end of this subsection.

When $k=2$ or $M$ is locally conformally flat, like the ordinary Yamabe problem, the $\s_k$-Yamabe problem is a variational problem of the following functional
\[
\mathcal{F}_k(g)=\int_M \s_k(g) dv_g
\]
by \cite{ViacDuke}.
This paper is related to the case $k=2$.
In this case, the following $\s_2$-Yamabe constant for $n\ge 5$ plays an important role in solving the $\s_2$-Yamabe problem in \cite{GWEcole} and \cite{STWJDG}:
\begin{equation}
    \label{Yamabe_s}
    Y_{\s_2}(M, [g_0]) :=\inf _{g\in \mathcal{C}_2[g_0]} \frac{\int_M \s_2(g) dv_g}{vol(g)^{\frac{n-4}n}}.
\end{equation} 

The corresponding Sobolev inequality for $n\ge 5$ is
\[
Y_{\s_2}(\Sn) = \inf _{g\in \mathcal{C}_2[\Sn]} \frac{\int_\Sn \s_2(g) dv_g}{vol(g)^{\frac{n-4}n}}= \frac {n(n-1)} 8 \omega_n^{\frac{4}n},
\]
which was proved by Guan-Viaclovsky-Wang \cite{GVW03}.
 By the local estimates proved  by Guan-Wang \cite{GWIMRN} and the classification of the entire solutions by Li-Li \cite{LLActa}, it was proved in  \cite{GWEcole, STWJDG} that the infimum of $Y_{\s_2}(M, [g_0])$ is achieved 
 if 
\begin{equation} \label{Y2_ineq}
 Y_{\s_2}(M, [g_0])<Y_{\s_2}(\Sn).\end{equation}
 Hence to solve the $\s_2$-Yamabe problem, as the classical Yamabe problem one only needs to show that 
 \eqref{Y2_ineq} holds for any manifold, which is not conformally equivalent to the standard sphere. This was proved by Ge-Wang in \cite{GWEcole} for non locally conformally flat manifolds  of dimension $n>8$ and by Sheng-Trudinger-Wang in \cite{STWJDG} for a general manifold of dimensional $n\ge 5$, and hence the $\s_2$-Yamabe problem was solved. The locally conformally flat case was solved previously by Guan-Wang \cite{GWCrelle}   and Li-Li \cite{LLCPAM}. 
 {For $n=3$ and $n=4$, the $\s_2$-Yamabe problem was solved by Gursky-Viaclovsky \cite{Gursky-Viaclovsky,Gursky-ViaclovskyAdv} and  Chang-Gursky-Yang \cite{CGYAnn, CGYAnalysis} respectively. Also see \cite{Gursky-ViaclovskyJDG, Wang-Trudinger,LNJFA2014}}. In \cite{STWJDG} Sheng-Trudinger-Wang used a clever trick to  derive   \eqref{Y2_ineq} from \eqref{Y1_ineq} by proving
 \begin{equation}
 Y_{\s_2}(M, [g_0]) \le \frac{2(n-1)} {n(n-2)^2}  (Y(M, [g_0]))^2.
 \end{equation} 
In fact,  inequality \eqref{Y2_ineq} follows from the above inequality and \eqref{Y1_ineq} clearly
 \[
 Y_{\s_2}(M, [g_0]) \le \frac{2(n-1)} {n(n-2)^2} (Y(M, [g_0]))^2 <  \frac{2(n-1)} {n(n-2)^2} 
 (Y(\mathbb{S}^n))^2 = Y_{\s_2} (\mathbb{S}^n),
 \]
 if $(M, g_0)$ is not conformally equivalent to the standard sphere. We will show  that this trick works for the $Q$-Yamabe problem studied in \cite{GM}, which is new to our best knowledge. 
 Actually it works also for our new Yamabe problem. See Section \ref{Various Yamabe constants} below.

 In our recent work \cite{GWW}, we  revisited the $\s_2$-Yamabe problem and proved that 
 the infimum
 \begin{equation}
  \inf _{g\in \mathcal{C}_1[g_0]} \frac{\int_M \s_2(g) dv_g}{Vol(g)^{\frac{n-4}n}}
\end{equation}
is achieved in a larger cone ${\mathcal C}_1[g_0]:=\{g\in [g_0]\,|\, R_g >0\},$
 provided that the infimum is positive. Namely the $\s_2$-Yamabe problem is solvable  in the larger cone
 $\mathcal{C}_1[g_0]$. This answers a recent question asked by J. Case.
 
\subsection{Yamabe problem for the quotient between \texorpdfstring{$\s_2$}{} and \texorpdfstring{$\s_1$}{} scalar curvatures}\label{s2 and s1}

 Besides $\sigma_2$-Yamabe problem, the  quotient  Yamabe 
 between $\s_2$ and $\s_1$
 also plays an essential role in 
 the study of the $\s_k$ scalar curvature, especially in establishing optimal geometric inequalities, see for example \cite{GWDuke}.
For $n\ge 5$, its corresponding Yamabe constant is defined  in \cite{GWIMRN}
 $$Y_{\sigma_2/\sigma_1}(M, [g_0]):=\inf _{g\in \mathcal{C}_2[g_0]} \frac{\int \sigma_2 dv_g}{(\int \sigma_1 dv_g)^{\frac{n-4}{n-2}}}\le Y_{\sigma_2/\sigma_1}(\Sn).$$
 As the  ordinary Yamabe constant, strict inequality holds if and  only if $(M,g_0)$ is not conformal to sphere.
In \cite{GWDuke}, the corresponding Sobolev inequality was proved
$$Y_{\sigma_2/\sigma_1}(\Sn)=\frac{n(n-1)}{8}\frac{\omega_n^{\frac{2}{n-2}}}{(\frac 1 2 n)^{\frac{n-4}{n-2}}}.$$
It holds  even in a larger cone $\C_1[g_0]$ by \cite{GWCAG}, i.e.,
\begin{equation}
    \label{Sobolev_21}
    \int_\Sn \s_2 (g) dv_g \ge  \frac{n(n-1)}{8}\frac{\omega_n^{\frac{2}{n-2}}}{(\frac 1 2 n)^{\frac{n-4}{n-2}}} 
    \big(\int_\Sn \s_1(g) dv_g\big)^{\frac{n-4}{n-2}}, \quad \forall g\in \C_1[g_0].\end{equation}
    Inequality \eqref{Sobolev_21} is crucial for a new Sobolev inequality in Section 3 below.
    
\subsection{\texorpdfstring{$Q$}{}-Yamabe problem}\label{Q Yamabe}
The $Q$ curvature of Branson \cite{Branson} is defined by
\eq{\label{Definition of Q}
Q_g &=& -\Delta \s_1 +4 \s_2(g) +\frac {n-4} 2 \s_1(g)^2\\ 
}
and the associated operator is called Paneitz operator, which was introduced by Paneitz in 1983 \cite{Paneitz}, 
\[
P_{g_{0}}u=\Delta^{2}u+ {\rm div}_{g_{0}}\{(4A_{g_{0}}-(n-2)\sigma_{1}(g_{0})g_{0})(\nabla u,\cdot)\}+\frac{n-4}{2}Q_{g_{0}}u.
\]
The Paneitz operator changes also conformally:  given $\varphi\in C^{\infty}(M),$ 
\begin{equation}
P_{u^{\frac{4}{n-4}}g_0}\varphi=u^{-\frac{n+4}{n-4}}P_{g_0}(u\varphi).\label{eq:conformal P transformation}
\end{equation}

The  Yamabe problem for $Q$ curvature (or the $Q$-Yamabe problem) is to ask if there is a conformal metric $g\in[g_0]$ with a constant $Q_g$. If we let $g_{u}=u^{\frac{4}{n-4}}g_{0}$ for $n\neq 4$, then 
\begin{equation}
\frac{n-4}{2}Q_{g_{u}}=u^{-\frac{n+4}{n-4}}P_{g_{0}}u.\label{eq:Q-transformation}
\end{equation}

The main new challenge of the $Q$-Yamabe problem is the lack of the maximum principle, since it is a 4-$th$ order equation.   There is a lot of work to deal with this problem. See the references mentioned in the Introduction. Recently
 Gursky-Malchiodi \cite{GM} made a breakthrough to show that it does have the maximum principle, if the underlying manifold has non-negative scalar curvature and semi-positive $Q$ curvature. Since it is also important for the paper, we will recall the maximum principle in Subsection \ref{subsec2.5} below. Then they introduced
 a conformal flow for $n\ge 5$
 \begin{equation}
     \label{non-local-flow}
u_{t}=-u+\frac{n-4}{2}\mu P_{g_{0}}^{-1}\big (u^{\frac{n+4}{n-4}}\big), \qquad \mu =\frac {\int  uP_{g_0} u }{\int  u^{\frac {2n}{n-4}} dv_{g_0}} \end{equation}
 which preserves the positivity of the conformal factor $u$
and satisfies suitable monotonicity properties. The last step to solve the $Q$-Yamabe problem in \cite{GM} is to show that there exists a suitable test function, which is equivalent to show 
\begin{eqnarray}
    \label{Y_Q}
    Y_Q (M, [g_0]) < Y_Q (\Sn),
\end{eqnarray}
if $(M, g_0)$ is not conformally equivalent to $\Sn$.
Here the Yamabe type constant $Y_Q$ for $n\ge 5$ is defined by
$$Y_{Q}(M,[g_0])=\frac{n-4}{2}\inf_{\mathcal C_{Q}[g_0]} \frac{\int_M Q_{g} dv_{g}}{Vol(g)^{\frac{n-4}{n}}}$$
where $$\mathcal C_{Q}[g_0]=\{g\in [g_0]|R_g>0\,   \& \, Q_{g} \hbox{ is semi-positive} \}.$$ 
In Section \ref{Various Yamabe constants} below we will provide a simpler proof of \eqref{Y_Q}.

 After the work of Gursky-Malchiodi \cite{GM},  conformally invariant
extensions of the above results have been obtained by 
 Hang-Yang  in \cite{HYIMRN2015, HYCPAM3-dim, HYLecture} and Gursky-Hang-Lin in \cite{GHL}. See also \cite{HYLecture} and \cite{Malchiodi}. See also a recent work of Gong-Lee-Wei \cite{GongLiWei} on the global convergence of flow \eqref{non-local-flow}.

In the above we mainly focus on dimension $n\ge 5$. There is a lot
of work for the 4-dimensional  case, in which the corresponding 
equation is a higher-order Liouville-type equation. See \cite{ChangYang1995, ChangGurskyYangAJM1999,BCY, BairdFR,Brendle-Annofmath2003,Chen-Xu2011,Chen-Xu2012, DM,BrendleAdv, MalchiodiStruwe2006JDG}.
The
 compactness of solutions to the $Q$ -Yamabe problem has been also intensively studied. See \cite{HebeyRobert2004,QingRaskeCV2006,LiXiong2019, LiPacific2019, WeiZhao, GongKimWei2025}. {The phenomenon of the 3-dimensional case is different from the higher-dimensional case, and it is expected that the scalar curvature and the $Q$ curvature play a more dominant role for the geometry of the conformal class in dimension 3.  See \cite{HYCPAM3-dim,HYLecture,HYIMRN2015,Hang-Yang2004,Xu-Adv2005,Xu-YangDCDS, Xu-Yang,Yang-Zhu}.}

\subsection{A new Yamabe problem}\label{new problem} 

Motivated by previous Yamabe problems and a new optimal Sobolev inequality on $\Sn$ proved in Section \ref{section:sobolev inequality} below, we are interested in the following new Yamabe  problem:

\

\noindent{\bf Problem.} {\it Find a conformal metric $g$ in a given conformal class $[g_0]$ with}
 \eq{\label{eq_newY}
 Q_g\slash R_g=const.
}

\

\noindent This problem is also a variational one: the critical points of the following functional
 \eq{\label{eq_functional}
  \frac{n-4}{2}\frac {\int_M Q_g dv_g} {(\int _M R_g dv_g) ^{\frac {n-4}{n-2}} }
}
are solutions of \eqref{eq_newY}.
Like the previous Yamabe problems it is natural to define  a Yamabe type constant by 
\[Y_{4,2} (M, [g_0]):= \inf_{g\in {\mathcal C}_Q[g_0]} \frac{n-4}{2}\frac {\int_M Q_g dv_g} {(\int _M R_g dv_g)^{\frac {n-4}{n-2}} }.\]
We will prove in Sections \ref{section:3-dim}, \ref{section:dim4} and \ref{section:5+dim} that the minimum is achieved.
 

\subsection{Maximum Principle and related constants on \texorpdfstring{$\mathbb{S}^n$}{}} \label{subsec2.5}
For the convenience of the reader,
we recall the maximum principle theorem for $P_{g_0}$ in \cite{GM}, which plays a crucial role in the whole paper.
\begin{thm}[Lemma 2.1 and Theorem 2.2 in \cite{GM}]
\label{thm:maximum principle}Let $\left(M^{n},g_0\right)$
be a closed Riemannian manifold of dimension $n\geq5$. Assume (i)
$Q_{g_0}$ is semi-positive, (ii) $R_{g_0}\geq0$. We have the following 
\begin{itemize}
    \item 
$R_{g_0}>0$. 

\item If $u\in C^{4}$ satisfies
$P_{g_0}u\geq0,$ then either $u>0$ or $u\equiv0$ on $M^{n}$. Moreover,
if $u>0$, then $u^{4/(n-4)}g_0$ is a metric with non-negative Q-curvature
and positive scalar curvature.
\end{itemize}
\end{thm}

We also need the result for $n=3$ in \cite{Hang-Yang2017} and for $n=4$ in \cite{Vetois}.
\begin{thm}[\cite{HYCPAM3-dim,Vetois}]
\label{thm:positive scalar curvature}Let $(M,g_{0})$ be a smooth,
closed Riemannian manifold of dimension $n\geq 3$ with positive scalar
curvature and non-negative $Q$-curvature. Let $g$ be a conformal
metric to $g_{0}$ with non-negative $Q$-curvature. Then the scalar
curvature of $g$ is positive.
\end{thm}

Since  there are several scalar curvatures and various constants
on Sobolev inequalities, we list them for the convenience of the reader.
Let $\mathbb{S}^n$ be the standard metric $g_0$ of constant sectional curvature $1$. Hence its Schouten tensor is
\[
A_0=\frac 12 g_0.
\]
Hence its $\s_1$ and $\s_2$-scalar curvature are $\frac n2$ and $\frac 1 8 n(n-1)$ resp., while its scalar curvature is $n(n-1)$.
Its  Yamabe constant and $\s_2$ Yamabe constant are
{\[
Y(\mathbb{S}^n, [g_{\mathbb S^n}])=\frac{n(n-2)}{4}\omega _n^ {2\slash n}, \qquad 
Y_{\sigma_2}(\mathbb{S}^n, [g_{\mathbb S^n}])= \frac 18 n(n-1) \omega_n^ {4\slash n}.
\]
Its 
$Q$ curvature  is $\frac 1 8 n(n^2-4)$ 
and its quotient between $Q$ and $R$ is
\[
Q\slash R=\frac{n^2-4}{8(n-1)}. 
\]

\newpage
\section{Sobolev inequality and Rigidity}\label{section:sobolev inequality}

In this section we introduce a new type Sobolev inequality and  prove an Obata-type result for Equation \eqref{eq_QR}.

\subsection{A new Sobolev inequality}
We prove in this section the following Sobolev inequality

\begin{thm} \label{thm1.1} Let $n\ge 5$. 
We have
\eq{\label{new_ineq}
\int _\Sn Q_g dv_g \ge  \frac{1}{8}(n^2-4)n^{\frac{2}{n-4}}(n-1)^{\frac{4-n}{n-2}}\omega_n^{\frac{2}{n-2}}\left( \int_\Sn R_g dv_g\right)^{\frac {n-4}{n-2}}, \qquad \forall g\in \mathcal{C}_1[\gSn].
}
Equivalently, 
for  any smooth positive function $u$ with a property that $g= u^{\frac 4{n-4}}g_{\mathbb{S}^n}$ has a positive scalar curvature, i.e., $L_{g_{\Sn}} u^{\frac{n-2}{n-4}}>0$,
		it holds
			\begin{equation} \label{inq1.11}
			\int_{\Sn} u P_{g_{\Sn}} u \ge  
		(\frac n 2-2 )(\frac n2 +1) \{(\frac n2 )(\frac n 2-1)\}^{\frac 2 {n-2}} \omega_n^{\frac 2 {n-2}}
			\left( \int_{\Sn} u ^{\frac{n-2}{n-4}} L_{g_{\Sn}}  u^{\frac{n-2}{n-4}} \right) ^{\frac {n-4} {n-2} },			\end{equation}
			with equality if and only if
			\begin{equation}\label{eq:3.4}
			u(\xi) =a (1+ b\cdot \xi ) ^{-\frac {n-4}2},  \quad \hbox{ for } {a\in (0,+\infty)},   b\in \mathbb B^{n+1}.
			\end{equation}
	\end{thm}

\begin{proof} By the definition of $Q$ and the Newton inequality we have
\begin{equation}
Q_g= -\Delta \s_1 +4 \s_2(g) +\frac {n-4} 2 \s_1(g)^2 \ge -\Delta \s_1 +\frac {n^2 -4}{n-1} \s_2(g). 
\end{equation}
It follows from \eqref{Sobolev_21} that
\eq{
\int_\Sn Q & \ge \frac {n^2 -4}{n-1}  \frac{n(n-1)}{8}\frac{\omega_n^{\frac{2}{n-2}}}{(\frac 1 2 n)^{\frac{n-4}{n-2}}}  \big(\int_\Sn \s_1(g) dv_g\big)^{\frac{n-4}{n-2}}\\
&=
\frac{n(n^2-4)} 8 \left( \frac 1{n(n-1)}\right)^{\frac{n-4}{n-2}} \omega_n^{\frac 2{n-2} }
\left(\int _\Sn R dv_g \right)^{\frac {n-4}{n-2}}.
}

It is clear that equality holds if and only if $g$ is Einstein and hence has constant sectional curvature. Then it is easy to check that $u$ has form  \eqref{eq:3.4}.
\end{proof}

\begin{cor} \label{Y4,2Sphere}
We have
    \[ Y_{4,2}(\mathbb{S}^n)=Y_{4,2}(\mathbb{S}^n, [g_{\mathbb{S}^n}])=\frac{n-4}{2}\frac {\int_{\mathbb{S}^n} Q_{g_{\mathbb{S}^n}} dv_{g_{\mathbb{S}^n}}} {(\int _{\mathbb{S}^n} R_{g_{\mathbb{S}^n}}dv_{g_{\mathbb{S}^n}})^{\frac {n-4}{n-2}} }=\frac{n-4}{16}(n^2-4)n^{\frac{2}{n-2}}(n-1)^{\frac{4-n}{n-2}}\omega_n^{\frac{2}{n-2}}. \]
\end{cor}

For $n=3$ and $n=4$, we have the corresponding reverse Sobolev inequality
and Moser-Trudinger inequality. Since we will not use these inequalities for the existence in these two cases, we will leave them in a forthcoming  paper  \cite{GWW26b}.

Before we end this section, we propose the following conjecture related with the $k$-th GJMS operator  on  $\mathbb{S}^n$,
which is defined by
\begin{equation}
L_{2k} := \Pi_{j=1}^k \left(-\Delta +\frac {(n-2j)(n+2j-2)} 4\right).
\end{equation}
Actually, $L_2$ is the Yamabe operator and $L_{4}$ is the Paneitz operator for the unit $n$-sphere.

\

\noindent {\bf Conjecture 1.}  {\it  Let $0<l<k<n/2$. 
The following inequality 
	 \begin{equation}
	 \intsn u L_{2k} u \ge   
	 \frac {\Gamma (n/2 +k)}  {\Gamma (n/2 -k)}  \left( \frac{\Gamma (n/2+l)}{\Gamma (n/2-l)}\right)^{-\frac {n-2k}{n-2l}}\omega_n^{\frac {2(k-l)} {n-2l}}
	 \left( \intsn u ^{\frac{n-2l}{n-2k}} L_{2l}  u^{\frac{n-2l}{n-2k}} \right) ^{\frac {n-2k} {n-2l} }
	 \end{equation}
     holds for any smooth positive function $u$
	 	 with 
	 	 a property that $g= u^{\frac {4}{n-2k}}g_{\mathbb{S}^n}$ has a positive $l$-th $Q$ curvature. Moreover,
	  equality holds if and only if
	 \begin{equation}
	 u(\xi) =a (1+ b\cdot \xi ) ^{-\frac {n-2k}2},  \quad \hbox{ for } a\in (0,+\infty),  b\in {\mathbb {B}}^{n+1}.
	 \end{equation}
 }

\

The integral of the left-hand side is the total $k$-th $Q$ curvature of $g=u^{\frac 4{n-2k}}\gSn$, while the integral in the right-hand side is the total $l$-th $Q$ curvature. It is an isoperimetric type inequality. We have also verified  the conjecture for $k=3$ and $l=1$ in \cite{GWW26b}. 
We also conjecture the  corresponding inequality in $\R^n$ holds, namely
\eq{\label{Sobolev_R}
\int_{\R^n} u (-\Delta)^k u 
\ge \frac {\Gamma (n/2 +k)}  {\Gamma (n/2 -k)}  \left( \frac{\Gamma (n/2+l)}{\Gamma (n/2-l)}\right)^{-\frac {n-2k}{n-2l}}\omega_n^{\frac {2(k-l)} {n-2l}}
	 \left( \int_{\R^n} u ^{\frac{n-2l}{n-2k}} (-\Delta)^{l}  u^{\frac{n-2l}{n-2k}} \right) ^{\frac {n-2k} {n-2l} },
} for $u$ with similar positive conditions as above. Due to the requirement that $R_g>0$, which we believe is necessary, 
inequality  \eqref{new_ineq} on $\Sn$ does not imply \eqref{Sobolev_R} on $\R^n$.
It would be also interesting to ask if certain positivity on curvatures is necessary. 

The classification of optimizers and critical points of \eqref{Sobolev_R} is also very interesting.
Theorem \ref{thm1.1} implies the classification of optimizers in $\Sn$. We classify all critical points in $\Sn$ by using  Obata's method in the next subsection,
but leave the classification of
positive  entire solutions of
\begin{eqnarray}
    \label{entire-sol}
    (-\Delta)^2 u = u^{\frac 2 {n-4}} (-\Delta) u^\frac{n-2}{n-4}, \qquad \hbox{ in } \Rn
\end{eqnarray}
still open. We will pursue these questions in a forthcoming paper.
\medskip

\subsection{Rigidity and  Obata Theorem}

As mentioned in the Introduction, we prove the rigidity by using the original idea of Obata \cite{Obata1971}, which will be used in the proof of Theorem \ref{thm_Yamabe_Constants}. Our
proof follows closely the recently developed methods in
\cite{Vetois}, \cite{Case2024} and especially in \cite{Li-Wei2025}.

Actually  we consider a slightly generalized equation
\begin{equation} \label{eq4.1}
Q_{g}/R_{g}^{\alpha}=const, 
\end{equation}
for  any given $\alpha\in (0,1].$

\begin{thm}\label{classification}
	Let $(M, g_0)$ be a compact Einstein manifold of dimension $n\ge 3$ with scalar curvature $R_{g_0}>0$ and $\alpha\in(0,1]$ be given. Assume that a conformal metric $g = u^2g_0$ $(u > 0)$ satisfies  \eqref{eq4.1} and $R_g>0$.  
	If $(M, g_0)$ is not conformally equivalent to the round sphere, then $u$ must be a
positive constant. If $(M, g_0)$ is the standard round sphere ${\mathbb S}^n
$, then
$u$ has a form
\begin{equation}
u(\xi)= a (1+ b\cdot \xi)^{-1},  \quad \hbox{ for } a\in (0,+\infty), b\in \mathbb B^{n+1}.
\end{equation}

\end{thm}

Recall that the 
$Q$-curvature is defined (equivalently) by 
\[
Q_{g}:=-\frac{1}{2(n-1)}\Delta_{g}R_{g}-\frac{2}{(n-2)^{2}}\left|E_{g}\right|_{g}^{2}+\frac{n^{2}-4}{8n(n-1)^{2}}R_{g}^{2},
\]
where $E_g$ is the trace-free Ricci tensor of $g$, i.e., $E_g=Ric_g-\frac 1 n R_g g$.
Hence \eqref{eq4.1} is written as 
\begin{equation}
-\frac{ \Delta_{g}R }{R_g^\alpha }-\alpha_{1}\frac {\left|E_{g}\right|_{g}^{2}}{R_g^\alpha}+\alpha_{2}R_{g}^{2-\alpha}=\frac{2(n-1)Q_g}{R_g^{\alpha }},\label{eq:equation}
\end{equation}
with
\begin{equation}\label{definition of alpha1 and alpha2}
\alpha_{1}=\frac{4(n-1)}{(n-2)^{2}},\quad\alpha_{2}=\frac{n^{2}-4}{4(n-1)n}.
\end{equation}
Since all the computations below are with respect to $g$, we omit the subscript
$g$ if there is no confusion.
For example we use $R$ and $E$ to denote $R_g$ and $E_g$. It is well-known  from the conformal transformation of  the Ricci tensor and the assumption that $g_0$ is Einstein, that 
\begin{equation}
u E_{ij}=-(n-2)\left(\nabla_{ij}^{2}u-\frac{\Delta u}{n}g_{ij}\right)\label{eq:expression of E}
\end{equation}
and 
\begin{equation} \label{eq4.4}
\Delta  u=\frac{n}{2}u^{-1}\left|\nabla u\right|^{2}-\frac{R u}{2(n-1)}+\frac{u^{-1}R_{g_0}}{2(n-1)}.
\end{equation}

To distinguish $R_{ij}$ and the second derivatives of $R$, from now on we use $S_{g}$ to denote the scalar curvature, and $S_i=\n_i S $, $S_{ij}=\nabla_{ij}^{2}S$ to avoid any confusion in this section. 
We first need a Lemma.

\begin{lem} For any $\alpha>0$ we have
\begin{align}
\frac{(n-2)(1-n)}{n}\int \frac{\Delta  u\Delta  S}{S^{\alpha}}dv_{g}+(n-2)\frac{n-1}{n}\alpha\int\frac{\Delta u|\nabla S|^{2}}{S^{\alpha+1}}\label{eq4.7}\\
+(n-1)\int\frac{E(\nabla u,\nabla S)}{S^{\alpha}}+\frac{n-2}{2n}\int_{S^{n}}\frac{u|\nabla S|^{2}}{S^{\alpha}}+\frac{n-2}{n}\int S^{1-\alpha}\<\nabla u,\nabla S\> & =0.\nonumber 
\end{align}
\end{lem}
\begin{proof}
Equation (\ref{eq:expression of E}) gives 
\begin{eqnarray}
\int u\<E,\frac{\nabla^{2}S}{S^{\alpha}}\>dv_{g} & = & \int-(n-2)\<\nabla_{ij}^{2}u-\frac{\Delta u}{n}g_{ij},\frac{S_{ij}}{S^{\alpha}}\>dv_{g}\nonumber \\
 & = & \int-(n-2)\<\nabla^{2}u,\frac{\nabla^{2}S}{S^{\alpha}}\>dv_{g}+\frac{n-2}{n}\int\frac{\Delta u\Delta S}{S^{\alpha}}dv_{g}.\label{eq:formula  1}
\end{eqnarray}
We compute the first term in the left-hand side of the above formula by using integration
of parts and $E_{ij,j}=\frac{n-2}{2n}S_{i}$ 
\begin{align}
\int u  \<E,\frac{\nabla^{2}S}{S^{\alpha}}\>dv_{g}\nonumber 
 & =-\int E  (\nabla  (\frac{u}{S^{\alpha}}),\nabla S)-\int E_{ij,j}S_{i}\frac{u}{S^{\alpha}}\nonumber \\
 & =-\int \frac{E  (\nabla  u,\nabla S)}{S^{\alpha}}+\alpha\int u\frac{E  (\nabla S,\nabla S)}{S^{\alpha+1}}-\frac{n-2}{2n}\int \frac{u|\nabla S|^{2}}{S^{\alpha}}.\label{eq:formula 2}
\end{align}
For the first term in the right-hand side of \eqref{eq:formula  1}, we have
\begin{align*}
  \int\langle \nabla^{2}u,\frac{\nabla^{2}S}{S^{\alpha}}\rangle dv_{g}
= & \int-\left(\frac{u_{ij}}{S^{\alpha}}\right)_{i}\cdot S_{j}\\
= & -\int\frac{(\nabla S,\nabla\Delta u)}{S^{\alpha}}-\int\frac{Ric(\nabla u,\nabla S)}{S^{\alpha}}+\alpha\int\frac{\nabla^{2}u(\nabla S,\nabla S)}{S^{\alpha+1}}\\
= & \int\frac{\Delta u\Delta S}{S^{\alpha}}-\alpha\int\frac{\Delta u|\nabla S|^{2}}{S^{\alpha+1}}-\int\frac{E(\nabla u,\nabla S)}{S^{\alpha}}-\int\frac{S\<\nabla u,\nabla S\>}{nS^{\alpha}}+\alpha\int\frac{\nabla^{2}u(\nabla S,\nabla S)}{S^{\alpha+1}},
\end{align*}
where we have used the Ricci identity
\[
u_{iji}=u_{iij}+R_{ij} u_i.
\]
Hence the right hand side of \eqref{eq:formula  1} is 
\begin{align*}
& \frac{n-2}{n}\int\frac{\Delta  u\Delta  S}{S^{\alpha}}dv_{g}\\
 & -(n-2)\big(\int\frac{\Delta u\Delta S}{S^{\alpha}}-\alpha\int\frac{\Delta u|\nabla S|^{2}}{S^{\alpha+1}}-\int\frac{E(\nabla u,\nabla S)}{S^{\alpha}}-\int\frac{S\<\nabla u,\nabla S\>}{nS^{\alpha}}+\alpha\int\frac{\nabla^{2}u(\nabla S,\nabla S)}{S^{\alpha+1}}\big)\\
= & \frac{(n-2)(1-n)}{n}\int\frac{\Delta  u\Delta  S}{S^{\alpha}}dv_{g}+(n-2)(\alpha\int\frac{\Delta u|\nabla S|^{2}}{S^{\alpha+1}}-\alpha\int\frac{\nabla^{2}u(\nabla S,\nabla S)}{S^{\alpha+1}})\\
 & +(n-2)\int\frac{E(\nabla u,\nabla S)}{S^{\alpha}}+\frac{n-2}{n}\int S^{1-\alpha}\<\nabla u,\nabla S\>.
\end{align*}

It follows, together with (\ref{eq:formula 2}) and \eqref{eq:formula  1}, 
\begin{align*}
 & \frac{(n-2)(1-n)}{n}\int \frac{\Delta  u\Delta  S}{S^{\alpha}}dv_{g}+(n-2)(\alpha\int\frac{\Delta u|\nabla S|^{2}}{S^{\alpha+1}}-\alpha\int\frac{\nabla^{2}u(\nabla S,\nabla S)}{S^{\alpha+1}})\\
 & +(n-1)\int\frac{E(\nabla u,\nabla S)}{S^{\alpha}}+\frac{n-2}{n}\int S^{1-\alpha}\<\nabla u,\nabla S\>-\alpha\int u\frac{E  (\nabla S,\nabla S)}{S^{\alpha+1}}+\frac{n-2}{2n}\int \frac{u|\nabla S|^{2}}{S^{\alpha}}=0.
\end{align*}
The summation of the 2nd, 3rd, and 5th terms in the previous equation can
be simplified by (\ref{eq:expression of E}) 
\begin{align*}
 & (n-2)\alpha(\int\frac{\Delta u|\nabla S|^{2}}{S^{\alpha+1}}-\int\frac{\nabla^{2}u(\nabla S,\nabla S)}{S^{\alpha+1}})-\alpha\int u\frac{E  (\nabla S,\nabla S)}{S^{\alpha+1}}\\
= & (n-2)\alpha(\int\frac{\Delta u|\nabla S|^{2}}{S^{\alpha+1}}-\int\frac{\nabla^{2}u(\nabla S,\nabla S)}{S^{\alpha+1}})+\int(n-2)\alpha\left(\nabla^{2}u-\frac{\Delta  u}{n}g\right)\frac{(\nabla S,\nabla S)}{S^{\alpha+1}}\\
= & (n-2)\frac{n-1}{n}\alpha\int\frac{\Delta u|\nabla S|^{2}}{S^{\alpha+1}}.
\end{align*}
Therefore we have 
\begin{align*}
\frac{(n-2)(1-n)}{n}\int\frac{\Delta  u\Delta  S}{S^{\alpha}}dv_{g}+(n-2)\frac{n-1}{n}\alpha\int\frac{\Delta u|\nabla S|^{2}}{S^{\alpha+1}}\\
+(n-1)\int\frac{E(\nabla u,\nabla S)}{S^{\alpha}}+\frac{n-2}{2n}\int \frac{u|\nabla S|^{2}}{S^{\alpha}}+\frac{n-2}{n}\int S^{1-\alpha}\<\nabla u,\nabla S\> & =0,\nonumber 
\end{align*}
the formula \eqref{eq4.7}.
\end{proof}

\begin{prop}
For any $\alpha>0$, 
\begin{align}
&\frac{(n-1)\alpha_{1}}{2}\int\frac{\left|E  \right|_{g}^{2}}{S^{\alpha}}u^{-1}\left|\nabla  u\right|^{2}dv_{g}+\frac{1}{2n}\int u^{-1}R_{g_{0}}\big(\frac{\alpha_{1}\left|E  \right|_{g}^{2}}{S^{\alpha}}+\alpha\frac{|\nabla S|^{2}}{S^{\alpha+1}}\big)\label{main equality}+\alpha\frac{(n-1)}{2}\int\frac{|\nabla S|^{2}}{S^{\alpha+1}}u^{-1}\left|\nabla  u\right|^{2} \\ &+\frac{1}{2n}(1-\alpha)\int \frac{u|\nabla S|^{2}}{S^{\alpha}}\nonumber+B_{\alpha}\int u|E  |^{2}S^{1-\alpha}+C_{\alpha}\int\frac{E(\nabla u,\nabla S)}{S^{\alpha}} =\frac{2(n-1)^2}{n}\int \langle \nabla u, \nabla \frac{Q_g}{S^{\alpha}}\rangle,\nonumber
\end{align}
where
\[
B_{\alpha}:=\frac{2n}{(n-2)^{2}}(\frac{1}{n}+\alpha_{2}(2-\alpha)\frac{(n-1)}{n})-\frac{\alpha_{1}}{2n}>0, {\mbox{ provided } \alpha\in (0,1]}
\]
and
\[
C_{\alpha} :=\frac{(n-1)}{n-2}-\frac{2(1-\alpha)}{(n-2)}\left(1+\alpha_{2}(2-\alpha)(n-1)\right).\]

\end{prop}

\begin{proof}
    
From our equation (\ref{eq:equation}), i.e.,
\[
-\frac{\Delta  S}{S^{\alpha}}=\frac{2(n-1)Q_g}{S^{\alpha}}+\frac{\alpha_{1}\left|E  \right|_{g}^{2}}{S^{\alpha}}-\alpha_{2}S^{2-\alpha},
\]
we have
\begin{align*}
-\int\frac{\Delta  u\Delta  S}{S^{\alpha}}dv_{g} & =\int\Delta  u \left(\frac{2(n-1)Q_g}{S^{\alpha}}+\frac{\alpha_{1}\left|E  \right|_{g}^{2}}{S^{\alpha}}-\alpha_{2}S^{2-\alpha}\right).
\end{align*}
Inserting it in (\ref{eq4.7}) gives
\begin{align*}
\frac{(n-1)}{n}\int\Delta  u(\frac{2(n-1)Q_g}{S^\alpha}+\frac{\alpha_{1}\left|E  \right|_{g}^{2}}{S^{\alpha}}-\alpha_{2}S^{2-\alpha})+\frac{n-1}{n}\alpha\int\frac{\Delta u|\nabla S|^{2}}{S^{\alpha+1}}\\
+\frac{(n-1)}{n-2}\int\frac{E(\nabla u,\nabla S)}{S^{\alpha}}+\frac{1}{2n}\int\frac{u|\nabla S|^{2}}{S^{\alpha}}+\frac{1}{n}\int S^{1-\alpha}\<\nabla u,\nabla S\> & =0.
\end{align*}
First integrating by parts  the integral involving with $\a_2$ and then using \eqref{eq4.4}, i.e., $\Delta u=\frac{n}{2}u^{-1}\left|\nabla  u\right|^{2}-\frac{Su}{2(n-1)}+\frac{u^{-1}S_{g_{0}}}{2(n-1)}$, to replace $\Delta u$,  we have
\begin{align*}
&\frac{(n-1)}{n}\int(\frac{\alpha_{1}\left|E  \right|_{g}^{2}}{S^{\alpha}}+\alpha\frac{|\nabla S|^{2}}{S^{\alpha+1}})(\frac{n}{2}u^{-1}\left|\nabla_{g}u\right|^{2}-\frac{Su}{2(n-1)}+\frac{u^{-1}S_{g_{0}}}{2(n-1)})+\frac{(n-1)}{n-2}\int\frac{E(\nabla u,\nabla S)}{S^{\alpha}}\\
&+\frac{1}{2n}\int\frac{u|\nabla S|^{2}}{S^{\alpha}}+(\frac{1}{n}+\alpha_{2}(2-\alpha)\frac{(n-1)}{n})\int S^{1-\alpha}\<\nabla u,\nabla S\> =\frac{2(n-1)^2}{n}\int \langle \nabla u, \nabla \frac{Q}{S^\alpha}\rangle,
\end{align*}
which implies
\begin{align*}
\frac{(n-1)}{n}\int\frac{\alpha_{1}\left|E  \right|_{g}^{2}}{S^{\alpha}}(\frac{n}{2}u^{-1}\left|\nabla  u\right|^{2}-\frac{Su}{2(n-1)}+\frac{u^{-1}S_{g_{0}}}{2(n-1)})\\
+\frac{1}{2n}(1-\alpha)\int\frac{u|\nabla S|^{2}}{S^{\alpha}}+\alpha\frac{(n-1)}{n}\int\frac{|\nabla S|^{2}}{S^{\alpha+1}}(\frac{n}{2}u^{-1}\left|\nabla_{g}u\right|^{2}+\frac{u^{-1}S_{g_{0}}}{2(n-1)})\\
+\frac{(n-1)}{n-2}\int\frac{E(\nabla u,\nabla S)}{S^{\alpha}}+(\frac{1}{n}+\alpha_{2}(2-\alpha)\frac{(n-1)}{n})\int S^{1-\alpha}\<\nabla u,\nabla S\> & =\frac{2(n-1)^2}{n}\int \langle \nabla u, \nabla \frac{Q}{S^\alpha}\rangle.
\end{align*}
Then 
\begin{align*}
\frac{(n-1)}{n}\int\frac{\alpha_{1}\left|E  \right|_{g}^{2}}{S^{\alpha}}(\frac{n}{2}u^{-1}\left|\nabla_{g}u\right|^{2}+\frac{u^{-1}S_{g_{0}}}{2(n-1)})-\int\frac{\alpha_{1}}{2n}S^{1-\alpha}\left|E  \right|_{g}^{2}u\\
+\frac{1}{2n}(1-\alpha)\int\frac{u|\nabla S|^{2}}{S^{\alpha}}+\alpha\frac{(n-1)}{n}\int\frac{|\nabla S|^{2}}{S^{\alpha+1}}(\frac{n}{2}u^{-1}\left|\nabla  u\right|^{2}+\frac{u^{-1}S_{g_{0}}}{2(n-1)})\\
+\frac{(n-1)}{n-2}\int\frac{E(\nabla u,\nabla S)}{S^{\alpha}}+(\frac{1}{n}+\alpha_{2}(2-\alpha)\frac{(n-1)}{n})\int S^{1-\alpha}\<\nabla u,\nabla S\> & =\frac{2(n-1)^2}{n}\int \langle \nabla u, \nabla \frac{Q}{S^\alpha}\rangle.
\end{align*}
Again from (\ref{eq:expression of E}) we have 
\begin{align*}
\int u|E  |^{2}S^{1-\alpha} & =-(n-2)\int S^{1-\alpha}\<E  ,\nabla^{2}u\>\\
 & =(n-2)(1-\alpha)\int S^{-\alpha}E(\nabla u,\nabla S)+\frac{(n-2)^{2}}{2n}\int S^{1-\alpha}\<\nabla u,\nabla S\>.
\end{align*}
Inserting it into the previous one, {we infer}
\begin{align*}
&\frac{(n-1)}{n}\int\frac{\alpha_{1}\left|E  \right|_{g}^{2}}{S^{\alpha}}(\frac{n}{2}u^{-1}\left|\nabla  u\right|^{2}+\frac{u^{-1}S_{g_{0}}}{2(n-1)})+\alpha\frac{(n-1)}{n}\int\frac{|\nabla S|^{2}}{S^{\alpha+1}}(\frac{n}{2}u^{-1}\left|\nabla  u\right|^{2}+\frac{u^{-1}S_{g_{0}}}{2(n-1)})\\
&+\frac{1}{2n}(1-\alpha)\int \frac{u|\nabla S|^{2}}{S^{\alpha}}+B_\alpha\int u|E  |^{2}S^{1-\alpha}
+C_\alpha\int\frac{E(\nabla u,\nabla S)}{S^{\alpha}}  =\frac{2(n-1)^2}{n}\int \langle \nabla u, \nabla \frac{Q}{S^\alpha}\rangle.
\end{align*}

Putting two terms with $S_{g_0}$ together we get \eqref{main equality}. \end{proof}

{It is easy to see that $\frac{n-3-4\alpha_{2}(n-1)}{n-2}=C_{0}\le C_{\alpha}\le C_{1}$
and $B_{0}\ge B_{\alpha}\ge B_{1}$, provided $\alpha\in (0,1]$. It is clear that we only need to handle the last term 
$C_\alpha\int\frac{E(\nabla u,\nabla S)}{S^{\alpha}},$}  when
$Q\slash S^\alpha$ is constant.


\medskip

\begin{proof}[Proof of Theorem \ref{classification}]
Since $\frac{Q_g}{S^\alpha}$ is constant, the right hand side of \eqref{main equality} is zero. All terms in the left hand side are non-negative except the last one.
We control the  last term ${C_\alpha}\int\frac{E(\nabla u,\nabla S)}{S^{\alpha}}$  by $\frac{(n-1)\alpha_{1}}{2}\int\frac{\left|E  \right|_{g}^{2}}{S^{\alpha}}u^{-1}\left|\nabla  u\right|^{2}dv_{g}$
and $\frac{1}{2n}(1-\alpha)\int\frac{u|\nabla S|^{2}}{S^{\alpha}}$
or by $B_{\alpha}\int u|E  |^{2}S^{1-\alpha}$ and $\alpha\frac{(n-1)}{2}\int\frac{|\nabla S|^{2}}{S^{\alpha+1}}u^{-1}\left|\nabla  u\right|^{2}$ as follows.
For any positive constant $A_{\alpha}$,  
\begin{align*}
A_{\alpha}\big|\int\frac{E(\nabla u,\nabla S)}{S^{\alpha}}\big| 
 \le \frac{(n-1)\alpha_{1}}{2}\int\frac{\left|E  \right|_{g}^{2}}{S^{\alpha}}u^{-1}\left|\nabla  u\right|^{2}dv_{g}+\frac{A_{\alpha}^{2}}{2(n-1)\alpha_{1}}\int\frac{u|\nabla S|^{2}}{S^{\alpha}},
\end{align*}
and  for any $A_{\alpha}\le|C_{\alpha}| $, 
\begin{align*}
(|C_{\alpha}|-A_{\alpha})\big|\int\frac{E(\nabla u,\nabla S)}{S^{\alpha}}\big| 
\le \alpha\frac{(n-1)}{2}\int\frac{|\nabla S|^{2}}{S^{1+\alpha}}\frac{\left|\nabla  u\right|^{2}}{u}+\frac{(|C_{\alpha}|-A_{\alpha})^{2}}{2\alpha(n-1)}\int u|E  |^{2}S^{1-\alpha}.
\end{align*}
 If 
\begin{equation}\label{condition 1}
   \frac{A_{\alpha}^{2}}{2(n-1)\alpha_{1}}\le\frac{1}{2n}(1-\alpha) 
\end{equation}
 and 
\begin{equation}\label{condition 2}
    \frac{(|C_{\alpha}|-A_{\alpha})^{2}}{2\alpha(n-1)}\le B_{\alpha},
\end{equation}
then together with \eqref{main equality}, we obtain $$\alpha\frac{(n-1)}{n}\int\frac{|\nabla S|^{2}}{S^{\alpha+1}}(\frac{n}{2}u^{-1}\left|\nabla  u\right|^{2}+\frac{u^{-1}R_{g_{0}}}{2(n-1)})=0$$ 
and 
it is easy to finish the proof of the theorem by Obata theorem for scalar curvature.

It remains to show \eqref{condition 1} and \eqref{condition 2}.

Once we prove $C_{\alpha}^{2}\le\frac{(1-\alpha)\alpha_{1}(n-1)}{n}+2\alpha(n-1)B_{\alpha}:=I_{\alpha}$, we can choose $A_{\alpha}\ge 0$ such that
\eqref{condition 1} and \eqref{condition 2} hold and  the theorem is proved.
Actually, 
for $\text{\ensuremath{\alpha_{1},\alpha_{2}}}$ defined as (\ref{definition of alpha1 and alpha2})
\[
C_{\alpha}=\frac{n-1}{n-2}-\frac{2(1-\alpha)(4n+(2-\alpha)n^{2}-4(2-\alpha))}{4n(n-2)}
\]
and 
\[
I_{\alpha}=\frac{4(n-1)^{2}}{n(n-2)^{2}}(1-2\alpha)+\frac{4\alpha(n-1)}{(n-2)^{2}}(1+\frac{n^{2}-4}{4n}(2-\alpha)).
\]

Note that $$I_{\alpha}-C_{\alpha}^{2}=I_{1}-C_{1}^{2}+\frac{(1-\alpha)E(\alpha)}{4n^2(n-2)^2}.$$
Since $\alpha^2<\alpha$ for $0<\alpha\le 1$, we have
\begin{align*}
E(\alpha)&=(n^2-4)^2 \alpha^3-(n-2)(n+2)\left(5 n^2+8 n-20\right)\alpha^2+8\left(n^2+n-4\right)\left(n^2+2 n-4\right)\alpha+4(n-2)^2(3 n-4)\\
&\ge (n^2-4)^2 \alpha^3+\left(n^2+4 n-4\right)\left(3 n^2+4 n-12\right)\alpha+4(n-2)^2(3 n-4)\ge 0.
\end{align*}
Hence $C_{\alpha}^{2}<I_{\alpha}$ for $0\le\alpha\le1$,
where $C_{0}=\frac{-3n+4}{n(n-2)}$ and $I_{0}=\frac{4(n-1)^{2}}{n(n-2)^{2}}$,
$C_{1}=\frac{n-1}{n-2}$ and $I_{1}=\frac{n(n-1)}{(n-2)^{2}}.$ 
\end{proof}

\begin{thm}\label{subcritical noe xistence on sphere} On $\Sn$, if a conformal metric $g = u^2g_{\mathbb{S}^n}$ $(u > 0)$ of positive scaler curvature satisfies  $\frac{Q_g}{R_g^{\alpha}}=u^{-\varepsilon}$ for  a fixed  $\varepsilon\in (0, \infty)$ and a constant $\alpha\in (0,1]$, then $u$ is constant and  hence $g=cg_{\mathbb{S}^n}$ for some positive constant $c>0$. 
\end{thm}
\begin{proof}
Since  the right hand side of \eqref{main equality} is $\int \langle \nabla u, \nabla \frac{Q_g}{R_g^{\alpha}}\rangle =-\varepsilon \int u^{-\varepsilon-1}|\nabla u|^2\le 0$ and the left side is non-negative, 
the proof of this theorem is similar to that of Theorem \ref{classification}. \end{proof}

For further developments of Obata's 
method on $\mathbb{R}^n$ we refer to the works of Xinan Ma et al. \cite{Ma-Wu-Wu2025,Ma-Wu-Zhou2025} and references therein.

\section{Various Yamabe constants}\label{Various Yamabe constants}
In this section we discuss relations between various  Yamabe constants and the ordinary Yamabe constant $Y$ and prove the important inequality
\begin{eqnarray}
    Y_{4,2} (M,[g_0]) < Y_{4,2}(\Sn),
\end{eqnarray}
if $(M, g)$ is not conformally equivalent to $\Sn$.

\subsection{\texorpdfstring{$Y_{4,2}$}{} vs 
\texorpdfstring
{$Y$}{}
}

We first  recall
$$
Y_{4,2}(M, [g_0])=\frac{n-4}{2}\inf_{g\in \mathcal C_Q[g_0]} \frac{\int_M Q_{g} dv_{g}}{\left(\int_M R_{g}  dv_{g}\right)^{\frac{n-4}{n-2}}}.
$$

In this section we prove that $Y_{4,2}(M,[g_{0}])<Y_{4,2}(\mathbb{S}^{n},[g_{\mathbb{S}^{n}}])$
when the manifolds are not locally conformally flat. The main result in this section is 
the following.
\begin{thm}\label{thm_Yamabe_Constants}
    Let $(M,g_{0})$ be an $n$-dimensional  closed 
manifold with $n\ge 5$. Assume that $Q_{g_{0}}$ is semi-positive
and $R_{g_{0}}$ is non-negative. We have
\begin{enumerate}
    \item [(i)]
the following inequality between $Y_{4,2}$ and $Y$
\eq{\label{Y42} Y_{4,2}(M,[g_{0}])\le \frac{(n-4)(n+2)}{n(n-2)^{\frac{2}{n-2}}(4(n-1))^{\frac{n-4}{n-2}}}   (Y(M,[g_0]))^{\frac{n}{n-2}}.
}
\item [(ii)]
Equality holds if and only if $(M,g_{0})$ is 
conformally equivalent to an Einstein metric with positive scalar curvature. 

\end{enumerate}
\end{thm}

\begin{proof} (i). 
Recall $Y(M,[g_0])$ and $Y(\mathbb{S}^n,[g_{\mathbb{S}^n}])$ the classical Yamabe invariant of $(M,[g_0])$ and $(\mathbb{S}^n, [g_{\mathbb{S}^{n}}])$ respectively.
By the assumption given in the Theorem and the maximum principle we have in fact that $R_{g_0}>0$ and hence $Y(M, [g_0])>0$.
 It is easy to see that the  $Q$ curvature has the following equivalent expressions
\begin{align}
Q&=-\frac{1}{2(n-1)}\Delta R+\frac{n^3-4n^2+16n-16}{8(n-1)^2(n-2)^2}R^2-\frac{2}{(n-2)^2}|Ric|^2\nonumber\\
&=-\frac{1}{2(n-1)}\Delta R+d(n)R^2-\frac{2}{(n-2)^2}|E|^2,\label{Einstein term}
\end{align}
where $E$ is the traceless  Ricci tensor and
\[d(n)=\frac{n^2-4}{8n(n-1)^2}. \]
Hence we have 
\begin{align*}
Q \le -\frac{1}{2(n-1)}\Delta R+d(n)R^2,
\end{align*}
with equality if and only if $E=0$.

Let $g_v=v^{4\slash {(n-4)}}g_0$ be a Yamabe metric in $[g_0]$. We may assume
$0< \frac{n-2}{4(n-1)}R_{g_v}=Y(M,[g_0])\le Y(\mathbb{S}^n,[g_{\mathbb{S}^{n}}])$ and $vol(g_v)=1$. Thus we have 
\begin{eqnarray} \label{eq_aa1}
\int_M v^{\frac{2n}{n-4}}dv_{g_0}&=&1, \\ 
R_{g_v}&=&\frac{4(n-1)}{n-2}v^{-\frac{n+2}{n-4}}L_{g_0}v^{\frac{n-2}{n-4}},\\
\frac{2}{n-4}v^{-\frac{n+4}{n-4}}P_{g_0}v&=&Q_{g_v} \le d(n)R^2_{g_v}. \label{eq_aa3}
\end{eqnarray}
We consider the following equation
\eq{ \label{eq_aa4}
P_{g_0}u=\frac {n-4}2 d(n) R^2_{g_v} v^{\frac{n+4}{n-4}},
}
which has a unique solution, since $P_{g_0}$ is invertible.
\eqref{eq_aa3} and \eqref{eq_aa4} imply
$$
P_{g_0}u \ge  P_{g_0}v.
$$
Using Theorem \ref{thm:maximum principle}, we have $u \ge v$. Moreover, since  
$P_{g_0}u>0$, 
\eqref{eq:Q-transformation} gives
$
  Q_{g_u}\ge 0.$ Hence $  R_{g_u}> 0$ from the maximum principle, Theorem \ref{thm:positive scalar curvature}.

Now we arrive at the following crucial estimate
\begin{align}\label{eq6.1}
\frac{2}{n-4}\int_M uP_{g_0}u&=d(n)R^2_{g_v}\int_M v^{\frac{n+4}{n-4}} u\\
&\le d(n)R^2_{g_v}(\left(\int_M v^{\frac{n+2}{n-4}} u^{\frac{n-2}{n-4}}\right)^{\frac{n-4}{n-2}}\nonumber\\
&= d(n)R^{\frac{n}{n-2}}_{g_v}\left(\int_M R_{g_v} v^{\frac{n+2}{n-4}} u^{\frac{n-2}{n-4}}\right)^{\frac{n-4}{n-2}}\nonumber\\
&= d(n)R^{\frac{n}{n-2}}_{g_v}\left(\int_M  \frac{4(n-1)}{n-2}
(L_{g_0}v^{\frac{n-2}{n-4}}) u^{\frac{n-2}{n-4}}\right)^{\frac{n-4}{n-2}}
\nonumber\\
&= d(n)R^{\frac{n}{n-2}}_{g_v}\left(\int_M \frac{4(n-1)}{n-2}(L_{g_0}u^{\frac{n-2}{n-4}} )v^{\frac{n-2}{n-4}}\right)^{\frac{n-4}{n-2}}
\nonumber\\
&= d(n)R^{\frac{n}{n-2}}_{g_v}\left(\int_M R_{g_u} u^{\frac{n+2}{n-4}} v^{\frac{n-2}{n-4}}\right)^{\frac{n-4}{n-2}}
\nonumber
\\
&\le  d(n)R^{\frac{n}{n-2}}_{g_v}\left(\int_M R_{g_u}  u^{\frac{2n}{n-4}}\right)^{\frac{n-4}{n-2}},\nonumber
\end{align}
where in the first equality we have used \eqref{eq_aa4}, in the second one the H\"older inequality and \eqref{eq_aa1}, in the fifth equality the self-adjointness of $L_{g_0}$ and in the last inequality $v\le u$.
It follows
\begin{align*}
\int_M Q_{g_u} dv_{g_u}&
\le  
d(n)R^{\frac{n}{n-2}}_{g_v}\left(\int_M R_{g_u}  dv_{g_u}\right)^{\frac{n-4}{n-2}}. 
\end{align*}
Hence we have
\begin{align}
\label{eq_aa5}
\frac{\int_M Q_{g_u} dv_{g_u}}{\left(\int_M R_{g_u}  dv_{g_u}\right)^{\frac{n-4}{n-2}}}&\le  d(n)R^{\frac{n}{n-2}}_{g_v}
 = d(n)  (\frac{4(n-1)}{n-2}Y(M,[g_0]))^{\frac{n}{n-2}},
\end{align}
because $g_v$ is a Yamabe metric.

It follows
\[
Y_{4,2}(M,[g_0])\le \frac{n-4}{2}d(n)\left(\frac{4(n-1)}{n-2}\right)^{\frac{n}{n-2}} (Y(M,[g_0]))^{\frac{n}{n-2}}.
\]

(ii). 
Now we discuss  equality in \eqref{Y42}. If equality in \eqref{Y42} holds, then equality in \eqref{eq_aa5} also holds, which implies that the Yamabe metric $g_v$ is Einstein. Conversely, assume $g_0$ is conformally  equivalent to an Einstein metric, wlog we may assume that $g_0$ is Einstein and $(M, g_0)$ is not conformally equivalent to $\Sn$. { Otherwise, the result is clear since $Y_{4,2}(\Sn,[g_{\Sn}])= \frac{(n-4)(n+2)}{n(n-2)^{\frac{2}{n-2}}(4(n-1))^{\frac{n-4}{n-2}}}   (Y(\Sn,[g_{\Sn}]))^{\frac{n}{n-2}}$}. By our main result, Theorem \ref{mainthm}, there exists a conformal metric $g\in [g_0]$ with constant quotient $Q\slash R$ and achieves the infimum $Y_{4,2}([g_0])$. By the Obata theorem, Theorem \ref{classification}, for metrics with constant quotient $Q\slash R$, we know that $g$ is Einstein and in fact $g=g_0$ up to a  multiple constant.  By the ordinary Obata theorem, $g$ is also a Yamabe metric. It is now clear that equality in \eqref{Y42} holds.
\end{proof}

We remark that the first statement in the previous Theorem is one of the main steps in the proof of  Theorem \ref{mainthm}, while the second one is its consequence.
As a consequence of (i) in Theorem \ref{thm_Yamabe_Constants} we have

\begin{cor}  \label{main_cor}  Let $(M,g_{0})$ be an $n$-dimensional closed manifolds  with $n\ge 5$. Assume that $Q_{g_{0}}$ is semi-positive
and $R_{g_{0}}$ is non-negative. Then $Y_{4,2}(M,[g_{0}])\le Y_{4,2}(\mathbb{S}^{n},[g_{\mathbb{S}^{n}}])$. Moreover,  equality holds if and only if $(M,g_{0})$ is 
conformally equivalent to the standard sphere. 
\end{cor}

\begin{proof}  
 The corollary follows from (i) in  Theorem \ref{thm_Yamabe_Constants}, the results of Aubin and Schoen for the Yamabe problem, as well as the new Yamabe inequality \eqref{new_ineq}  
 \begin{eqnarray*}
     Y_{4,2}(M,[g_{0}])&\le & \frac{(n-4)(n+2)}{n(n-2)^{\frac{2}{n-2}}(4(n-1))^{\frac{n-4}{n-2}}}  Y(M,[g_0])^{\frac{n}{n-2}}\\
 &<& \frac{(n-4)(n+2)}{n(n-2)^{\frac{2}{n-2}}(4(n-1))^{\frac{n-4}{n-2}}} Y(\Sn)^{\frac{n}{n-2}}= Y_{4,2} (\Sn),
 \end{eqnarray*}
 if $(M, g_0)$ is not conformally equivalent to $\Sn$.
\end{proof}

\begin{proof}[Proof of  Theorem \ref{mainthm} for $n\ge 5$]
Theorem \ref{mainthm} follows from  Theorem \ref{thm:existence under Yamabe restriction} below and Corollary \ref{main_cor}.
\end{proof}

Let $(M, g_0)$ be a closed manifold with positive Yamabe constant $Y(M) >0$.
As in \cite{GHL} we define {for $n\ge 5$} $$Y_4^*(M, [g_0])=\frac{n-4}{2} \inf _{\substack{{g} \in[g_0] \\ {R_g}>0}} \frac{\int_M {Q} dv_g}{Vol(g)^{\frac{n-4}{n}}}.$$
By \cite{GHL}, under the assumption of    $Y(M, [g_0])>0$, $Y_4^*(M, [g_0])>0$ for $n\ge 6$, there exists a metric $g\in [g_0]$ satisfying $R_g>0$ and $Q_g>0$, and hence we have the following corollary.
\begin{cor}\label{cor4.2}
    Let $(M,g_{0})$ be an $n$-dimensional closed manifolds  for $n\ge 6$. Assume $Y(M, [g_0])>0$,  $Y_4^*(M, [g_0])>0$. Then there exists a conformal metric $g=u^{\frac{4}{n-4}}g_{0}$ with
constant positive $Q_{g}/R_{g}$ curvature.
\end{cor}

In Corollary \ref{cor4.2} the assumptions are given on conformal invariants. It remains open if the result in \cite{GHL} mentioned above is true for $n=5$.

\subsection{\texorpdfstring{$Y^*$}{} and \texorpdfstring{$Y_4^{++}$}{} vs 
 \texorpdfstring{$ Y$}{}}

It is interesting to see that
the proof of Theorem \ref{thm_Yamabe_Constants} also works for the $Q$ curvature Yamabe constant and the $\Theta_4$ Yamabe constant. We will discuss $\Theta_4$ in Subsection 4.3 below.

 Let $(M,g_{0})$ be an $n$-dimensional closed manifolds  with $n\ge 5$. Assume that $Q_{g_{0}}$ is semi-positive
and $R_{g_{0}}$ is non-negative.
We define $$Y_4^{++}(M, [g_0])=\frac{n-4}{2} \inf _{\substack{{g} \in[g_0] \\ {R_g}>0, Q_g\ge 0}} \frac{\int_M {Q} dv_g}{Vol(g)^{\frac{n-4}{n}}}.$$
As above, here by $Q\ge 0$ we  means that $Q$ is semi-positive definite. {$Y^{++}_4$ is defined and used in \cite{GM}, where it is denoted by  $Y_Q$. It is trivial to see 
\[
Y^*_4 (M, [g_0])  \le  Y^{++}_4 (M, [g_0]).
\]

\begin{thm} \label{remark6.1}  Let $(M,g_{0})$ be an $n$-dimensional closed manifolds  with $n\ge 5$ {and positive Yamabe constant $Y(M,[g_0])>0$}. Assume that $Q_{g_{0}}$ is semi-positive and $R_{g_{0}}$ is non-negative. Then we have 
\begin{itemize}
    \item [(i)]
\eq{\label{eq:aa6}
    Y_4^* (M, [g_0])  \le Y^{++}_4(M, [g_0]) \le \frac{(n+2)(n-4)}{n(n-2)} Y(M, [g_0])^2.
}
Equality in the second inequality holds if and only if 
$(M,g_0)$ is conformal to an Einstein manifold.
\item[(ii)]  %
\begin{equation}\label{eq:aa7}
        Y_4^{++} (M, [g_0] )< Y_4^{++}(\mathbb{S}^{n},[g_{\mathbb{S}^{n}}])
    \end{equation}
    if and only if $(M,g_0)$ is not conformally equivalent to the standard sphere $\mathbb{S}^n$.
\end{itemize}
    
\end{thm}

\begin{proof}
    
Using the same proof of the previous theorem till \eqref{eq6.1}, a different H\"older inequality and $u\ge v$ again we have
\begin{align*}
\frac{2}{n-4}\int_M uP_{g_0}u&= d(n)R^2_{g_v}\int_M v^{\frac{n+4}{n-4}} u
\\
 & \le d(n)R^2_{g_v}(\int_M (vu)^{\frac {n}{n-4}})^{\frac {n-4}n}
 (\int _M v^{\frac {2n}{n-4}})^{\frac 4 n}\\
 &\le 
d(n)R^2_{g_v}(\int_M u^{\frac {2n}{n-4}})^{\frac {n-4}n}.
\end{align*}
It follows that
\[
\frac {\int_M Q_{g_u} dv_{g_u}}
{(\int_M u^{\frac {2n}{n-4}})^{\frac {n-4}n}}\le 
d(n)R^2_{g_v}.
\]
Notice that $g_u$ has positive scalar curvature and semi-positive $Q$ curvature. Hence \eqref{eq:aa6} follows. As in the proof of Corollary \ref{main_cor}, \eqref{eq:aa7} follows from \eqref{eq:aa6}, the resolution of the Yamabe problem and the optimal Sobolev inequality for the $Q$ curvature on $\Sn$.

As in the proof of Theorem \ref{thm_Yamabe_Constants}, it is easy to see that equality in the second inequality implies that $(M,g_v)$
is Einstein. To show that $(M,g_0)$ is conformal to an Einstein metric implies equality, we use the same proof but using the Obata  result for $Q$ curvature proved by Vetois \cite{Vetois}.
\end{proof}

\begin{remark} 
\begin{itemize}
\item[1.] Inequality \eqref{eq:aa7} was proved in \cite{GM} by choosing suitable test functions and using the expansion of the Green function for the Paneitz operator $P$. 
This is one of the crucial steps to solve the $Q$-Yamabe problem in \cite{GM}. 
  \item[2.] The inequality
    \[
    Y^*_4 (M,[g_0]) \le  \frac{(n+2)(n-4)}{n(n-2)} Y(M, [g_0])^2
    \]
    follows from a proof in \cite{WangFang} by Wang-Zhou, which implies that
    $Y^*_4(M,[g_0])< Y^*_4(\Sn)$, if $(M,g_0)$ is not conformally equivalent to $\Sn$.  However  unlike $Y^{++}_4$ it is unable (at least very difficult)  to show that this inequality does imply the achievement of $Y^*_4(M,[g_0])$, due to the lack of the maximum principle in the class of $R_g>0$.
  \end{itemize}  
\end{remark}

\subsection{\texorpdfstring{$\Theta_4$}{} vs 
 \texorpdfstring{$ Y$}{}}

It is interesting that we can also give an alternative proof for the result of Hang-Yang on $\Theta_4$.
In \cite{HY} Hang and Yang introduced
\begin{eqnarray}
    \Theta _4 (M, [g_0]) =\sup _{f\in L^{\frac {2n}{n+2}}(M)\slash \{0\} }\frac{ \int G_p f \cdot f}{\|f\|^2_{L^{\frac {2n}{n+4}}}},
\end{eqnarray}
where $G_pf(x)=\int G(x,y)f(y) dv_{g_0}$ and $G$ is the Green function of $P_{g_0}$, i.e. its inverse of the Paneitz operator $P_{g_0}$.
It was proved  in Lemma 2.1 in \cite{HY} that it is equivalent to 
\begin{equation}
    \Theta _4 (M, [g_0])=\sup \{\frac{ \int P_{g_0}u  \cdot u}{\|P_{g_0}u\|^2_{L^{\frac {2n}{n+4}}}} \,|\, u\in L^{\frac {2n}{n+2}}(M)\slash \{0\}, u>0, P_{g_0}u>0 \}.
\end{equation}

\begin{thm} 
\label{Thm4.3}
Let $(M,g_{0})$ be an $n$-dimensional  closed manifolds  with $n\ge 5$. Assume that  $Y(M, [g_0])>0$, $Q_{g_0} \geq 0$ and not identically zero. Then 
$$\Theta _4 (M, [g_0]) \ge \frac 2 {(n-4)d(n)}\left(\frac{n-2}{4(n-1)}\right)^2Y^{-2} (M, [g_0]). $$
Moreover, equality holds if and only if $(M,g_{0})$ is 
conformally equivalent to an Einstein metric with positive scalar curvature. As a consequence,
\eq{\label{eq:HY}
\Theta _4 (M, [g_0]) > \Theta _4 (\Sn), 
}
if $(M,g_0)$ is not conformally equivalent to $\Sn$.
\end{thm}

\begin{proof}
By Hang-Yang's result \cite{HY}, the assumption implies that ${\rm Ker}\, P_{g_0}= \{0\}$ and its  Green function $G$ is positive. 
 Let $g_v=v^{4\slash {(n-4)}}g_0$ be the Yamabe metric in $[g_0]$ such that 
$\frac{n-2}{4(n-1)}R_{g_v}=Y(M,[g_0])>0$ and $vol(g_v)=1$. Thus
\[
\int_M v^{\frac{2n}{n-4}}dv_{g_0}=1,
\]
\[
R_{g_v}=\frac{4(n-1)}{n-2}v^{-\frac{n+2}{n-4}}L_{g_0}v^{\frac{n-2}{n-4}},
\]
\[
\frac{2}{n-4}v^{-\frac{n+4}{n-4}}P_{g_0}v=Q_{g_v} \le d(n)R^2_{g_v}.
\]
The last one implies 
\[
\frac{2}{n-4} v=\int_M G(x,y) Q_{g_v} v^{\frac{n+4}{n-4} } \le \int_M G(x,y)   d(n)R^2_{g_v} v^{\frac{n+4}{n-4}}
.
\]
Now
define $u$ by
\[
\frac 2 {n-4} u:= \int_M G(x,y) d(n) R^2 _{g_v} v^ {\frac{n+4}{n-4}}.
\]
It is trivial to see that $u \ge v$ and $\frac 2 {n-4} P_{g_0}u=  d(n) R^2 _{g_v} v^ {\frac{n+4}{n-4}}$.
 In view of $vol(g_v)=1$, we have
\begin{align*}
    \left(\int |P_{g_0}u|^{\frac{2n}{n+4}}\right)^{\frac{n+4}{n}}
    = (\frac{n-4}{2})^2 d(n)^2 R_{g_v}^4
\end{align*}
and \begin{align*}
\frac{2}{n-4}\int_M uP_{g_0}u&= d(n)R^2_{g_v}\int_M v^{\frac{n+4}{n-4}} u
\ge d(n)R^2_{g_v}\int_M v^{\frac{2n}{n-4}} 
\\
&= d(n)R^2_{g_v},
\end{align*}
where we have used $u\ge v$.
Therefore, we have 
\begin{align*}
\Theta _4 (M, [g_0]) &\ge \frac{ \int P_{g_0}u  \cdot u}{\|Pu\|^2_{L^{\frac {2n}{n+4}}}}\ge \frac{2}{n-4}d(n)^{-1}R^{-2}_{g_v}\\
&=\frac{2}{n-4}d(n)^{-1} \left(\frac{n-2}{4(n-1)}\right)^2Y(M, [g_0])^{-2}. 
\end{align*}
Hence 
\begin{align*}
\Theta _4 (M, [g_0]) \ge \frac{2}{n-4}d(n)^{-1} \left(\frac{n-2}{4(n-1)}\right)^2Y(M, [g_0])^{-2}\ge\frac{2}{n-4}d(n)^{-1} \left(\frac{n-2}{4(n-1)}\right)^2Y(\mathbb{S}^n, [g_{\mathbb{S}^n}])^{-2}=\Theta_4(\Sn). 
\end{align*}

It is easy to see that equality implies that $(M, g_0)$ is conformal to an Einstein metric. If $g_0$ is conformal to an Einstein metric, as above we may assume that $(M,g_0)$ is Einstein and not conformally equivalent to $\Sn$. Then we use 
the existence result of Hang-Yang \cite{HYIMRN2015} to obtain
an optimizer $g$ of $\Theta_4(M,[g_0])$, which
has a constant $Q$ curvature
and then use the Obata theorem for $Q$ and for $R$ to show that 
$g$ is Einstein and is also a Yamabe metric. Then the equality follows.
\end{proof}

Inequality \eqref{eq:HY} was   used in \cite{HY} for the solution of $Q$ curvature Yamabe problem and
 was proved by using a finer expansion of the Green function.
 
 \subsection{\texorpdfstring{$\Theta_4$}{} vs 
 \texorpdfstring{$\widetilde Y_4$}{}}

There are other Yamabe type constants defined by Hang-Yang \cite{HY}
\eq{\nonumber 
Y_4 (M, [g_0]) = \inf _{u\in H^2(M)\backslash\{0\}}\frac{
\int_M u P_{g_0} u \, dv(g_0)}
{\|u\|^2 _{L^{\frac {2n}{n-4}}}}, \quad
Y_4 ^+ (M, [g_0]) = \inf _{0<u\in H^2(M)\backslash\{0\}}\frac{
\int_M u P_{g_0} u dv(g_0)}
{\|u\|^2 _{L^{\frac {2n}{n-4}}}}.
}
It is clear that 
$Y_4 \le Y^+_4$. 
$Y_4$ is the smallest constant related to 
the $Q$-curvature. 
Moreover, Hang-Yang proved that
\eq{
Y^+_4(M,[g_0]) \Theta _4(M,[g_0]) \le 1,
}
if $Y(M,[g_0]) >0$ and $g_0$ has semi-positive $Q$ curvature. It follows that
$\Theta_4 (M, [g_0])> \Theta_4 (\Sn)$ implies that $Y^+_4 (M, [g_0] )  < Y^+_4(\Sn)$ and $Y_4 (M, [g_0] )  < Y_4(\Sn)$.

For the completeness of the picture about these Yamabe constants, we   introduce one more 
\begin{equation}
    \widetilde Y_4(M,[g_0])= \inf_{g\in [g_0], Q_g > 0} \frac {{\frac{n-4}{2}}\int Q_g dv_g}{vol(g)^{\frac {n-4} n}}=\inf_{Pu> 0, u>0}
    \frac{\int uP_{g_0}u dv_{g_0}}{\|u\|^2_{\frac{2n}{n-4}}}.
\end{equation}
It is clear by definition that  $Y^+_4(M, [g_0]) \le \widetilde Y_4 (M, [g_0])\le Y ^{++}_4(M, [g_0])$.

Now due to duality, one can show that 
\begin{lem}  Under the same assumptions as in Theorem \ref{Thm4.3}, we have
\begin{equation}
    \tilde Y_4(M,[g_0])\cdot \Theta _4(M,[g_0])=1.
\end{equation}
\end{lem}
\begin{proof}
    It follows from the dualities between $L^{\frac {2n}{n-4}}$ and $L^{\frac {2n}{n+4}}$, and between $H^{2}$ and $H^{-2}$.
For a general theory see \cite{Carlen17}. For the convenience of the reader we provide the proof.

By the definition of $\Theta_4$ we have
\eq{\label{eq4.16} 2\int_M uv - \int_M u P^{-1}_{g_{0}} u \ge 
2 \int_M uv - \Theta_4(M, [g_0]) \|u\|^2_{L^{\frac {2n}{n+4}}}.
}

We claim  
\eq{
\int vP_{g_0}v \ge \Theta_4(M, [g_0])^{-1} \|v\|^2_{L^{\frac {2n}{n-4}}},
}
which implies
\eq{
\widetilde Y_4 \ge \Theta_4^{-1}.}
In order to prove the claim we only need to show that 
the supremum of the left hand side 
\eqref{eq4.16} is $\int_M v P_{g_0} v$. Since 
\eq{
\int_M uP^{-1}_{g_0} u +\int_M v P_{g_0} v -2\int uv 
= \int_M(P_{g_0}^{-\frac 12}u -P_{g_0}^{\frac 12 } v)^2,
}
we have
\eq{2\int_M uv - \int_M u P^{-1}_{g_{0}} u \le 
\int_M vPv,
}
with equality if and only if
$P_{g_0}^{-\frac 12} u=P_{g_0}^{\frac 12} v$. Thus, 
$$
\int_M vPv \ge 
2 \int_M uv -  \Theta_4(M, [g_0])\|u\|_{L^{\frac {2n}{n+4}}}^2
$$
which implies by duality
$$
\int_M vPv \ge 
2 \|u\|_{L^{\frac {2n}{n+4}}}\|v\|_{L^{\frac {2n}{n-4}}} -  \Theta_4(M, [g_0]) \|u\|_{L^{\frac {2n}{n+4}}}^2
$$
Taking $ \Theta_4(M, [g_0])\| u\|_{L^{\frac {2n}{n+4}}}=\|v\|_{L^{\frac {2n}{n-4}}}$, the desired claim yields.

For the other inequality, we use 
H\"older's inequality 
$$
\tilde Y_4(M,[g_0]) \frac{\ds\int_M u P_{g_0} u}{\|P_{g_0}u\|^2_{\frac{2n}{n+4}}}\le \frac{\left(\ds\int_M u P_{g_0} u\right)^2}{\|u\|^2_{\frac{2n}{n-4}}\|P_{g_0}u\|^2_{\frac{2n}{n+4}}}\le 1.
$$
Therefore, by taking the supremum, we prove that $\Theta_4 \le \widetilde  Y_4^{-1}$ and obtain the conclusion.
\end{proof}

\section{Existence for dimension \texorpdfstring{$n=3$}{}}\label{section:3-dim}

We start from $n=3$
and will prove  in this section the existence of metrics with constant quotient $Q_{g}/R_{g}$ curvature by a priori estimates and the degree theory, following closely the method of  Hang-Yang in \cite{HYCPAM3-dim}.

In 3-dimensional manifolds, we have 
\[
P_{g_0}\varphi=\Delta^{2}_{g_0}\varphi+4{\rm div}\left[Ric_{g_0}\left(\nabla\varphi,e_{i}\right)e_{i}\right]-\frac{5}{4}{\rm div}(R_{g_0}\nabla\varphi)-\frac{1}{2}Q_{g_0}\varphi,
\]
and hence
\begin{equation}
\begin{aligned}
\int_{M}P_{g_0}u\cdot vdv_{g_{0}} =\int_{M}\left[\Delta u\Delta v-4Ric_{g_0}(\nabla u,\nabla v)+\frac{5}{4}R_{g_0}\nabla u\cdot\nabla v-\frac{1}{2}Q_{g_0}uv\right]dv_{g_{0}}.
\end{aligned}
\label{eq:inner product for 3-dim}
\end{equation}
{In this section, if there is no confusion, the connection $\nabla$ and the Laplacian operator $\triangle $ are associated with the metric $g_0$.} 
For $u\in C^{\infty}(M)$ and $u>0,$
\[
-\frac{1}{2}Q_{g_{u}}=u^{7}P_{g_0}u.
\]
The Green function $G_{p}$ is defined by
\[
P_{g_0}G_{p}=\delta_{p}.
\]
By Corollary 2.2, Proposition 2.3 and Proposition 2.4 in \cite{HYCPAM3-dim},
we know the following theorem. 
\begin{thm}[Hang-Yang]\label{thm:HY3-dim}Assume that $(M,g_{0})$ is not
conformally equivalent to the standard sphere and $R_{g_0}>0$, $Q_{g_0}\ge0$
and $Q_{g_0}$ is not identically $0$, then $\ker P_{g_0}=0$ and the Green function
$G(p,q)<0$ for all $p,q\in M$.
\end{thm}
Since $G$ is countinous, $G$ has a
negative upper bound. 
We recall  also Proposition 2.1 in \cite{HYCPAM3-dim}:
\begin{prop}[Hang-Yang]\label{prop:Hang-Yang}Let $(M,g_{0})$ be a smooth, compact, Riemannian
3-manifold with $R_{g_0}>0,Q_{g_0}\geq0$. If $u\in C^{\infty}(M),u\neq$ constant,
and $P_{g_{0}}u\leq0$, then $u>0$ and $R_{u^{-4}g_0}>0$.
\end{prop}

Now if $Pu\leq0$, then $Pu=-\nu$ for some nonnegative measure $\nu$.
It follows that
\[
u(p)=-\int_{M}G_{p}(q)d\nu(q).
\]
For convenience denote $K(p,q)=-G(p,q)$. The solution $u$ to $Q_{g_{u}}/R_{g_{u}}=1$
is equivalent to 
\begin{equation}
P_{g_{0}}u=-\frac{1}{2}u^{-7}R_{g_{u}}.\label{eq:3-dim equation}
\end{equation}

\begin{proof}[Proof of Theorem \ref{mainthm} for $n=3$]
We only consider the case that  $(M, g_0)$ is not conformally equivalent to the unit sphere, otherwise we have done.
Let \[
\Omega=\{u\in C^{4,\alpha}(M,g_{0})|u>0,\,R_{g_{u}}>0,\,|u|_{C^{4,\alpha}(M,g_{0})}\le C\},
\]
for a suitable constant $C>0$.
We use the degree theory and
consider solutions in $\Omega$ to the following equation  for $0\le t\le1$, 
\begin{align}
P_{g_0}u & =-(1-t)u^{-3}-\frac{1}{2}tu^{-7}R_{g_{u}}\label{eq:path equation}\\
 & =-(1-t)u^{-3}-4tu^{-4}(\Delta u-2u^{-1}|\nabla u|^{2}+\frac{1}{8}R_{g_{0}}u),\label{eq:total expression of path equation}
\end{align}
where $u^{-7}R_{g_{u}}=8u^{-4}(\Delta u-2u^{-1}|\nabla u|^{2}+\frac{1}{8}R_{g_{0}}u).$
From  $R_{g_{u}}>0$, we have
\begin{equation}
\Delta u+\frac{1}{8}R_{g_{0}}u>0.\label{eq:Delta u+Ru positive}
\end{equation}
With the help of the Green function, we rewrite \eqref{eq:path equation} as 
\[
u(x)=\int_{M}K(x,y)((1-t)u^{-3}+\frac{1}{2}tu^{-7}R_{g_{u}})(y)dv_{g_{0}}.
\]
By Theorem \ref{thm:HY3-dim}, we know that $0<C_0\le K(x,y)\le C_1$ and hence
\[
C_{0}\int_{M}((1-t)u^{-3}+\frac{1}{2}tu^{-7}R_{g_{u}})dv_{g_{0}}\le u(x)\le C_{1}\int_{M}((1-t)u^{-3}+\frac{1}{2}tu^{-7}R_{g_{u}})dv_{g_{0}}.
\]
It follows 
\begin{equation}
C_{0}F\le u(x)\le C_{1}F,\label{eq:equvalence of u and F}
\end{equation}
where we denote 
\begin{align*}
F &\colon =\int_{M}((1-t)u^{-3}+\frac{1}{2}tu^{-7}R_{g_{u}})dv_{g_{0}}>0.
\end{align*}
Hence, we only need to estimate $F$, which can also be written as 
\begin{equation}
F=\int_{M}\left((1-t)u^{-3}+4tu^{-4}(\Delta u-2u^{-1}|\nabla u|^{2}+\frac{1}{8}R_{g_{0}}u)\right)dv_{g_{0}}.\label{eq:expression of F 1}
\end{equation}
On one side, by the Sobolev inequality $\int_{M}R_{g_{u}}dv_{g_{u}}\ge C_{s}(\int_{M}dv_{g_{u}})^{\frac{1}{3}}$ and \eqref{eq:equvalence of u and F}, we have
\begin{align*}
F & =\int_{M}(1-t)u^{-3}dv_{g_{0}}+\int_{M}\frac{1}{2}tu^{-1}R_{g_{u}}dv_{g_{u}}\\
 & \ge(1-t)Vol(g_{0})C_{1}^{-3}F^{-3}+\frac{t}{2}C_{1}^{-1}F^{-1}\int_{M}R_{g_{u}}dv_{g_{u}}\\
 & \ge(1-t)Vol(g_{0})C_{1}^{-3}F^{-3}+\frac{t}{2}C_{1}^{-1}F^{-1}C_{s}(\int_{M}dv_{g_{u}})^{\frac{1}{3}}\\
 & =(1-t)Vol(g_{0})C_{1}^{-3}F^{-3}+\frac{t}{2}C_{1}^{-1}F^{-1}C_{s}(\int_{M}u^{-6}dv_{g_{0}})^{\frac{1}{3}}\\
 & \ge(1-t)Vol(g_{0})C_{1}^{-3}F^{-3}+\frac{t}{2}C_{1}^{-3}F^{-3}C_{s}Vol(g_{0})^{\frac{1}{3}},
\end{align*}
from which 
\begin{equation}
F\ge C(g_{0},C_{1}).\label{eq:lower bound of F and u}
\end{equation}
On the other side, by (\ref{eq:Delta u+Ru positive}) and (\ref{eq:expression of F 1}),
we obtain 
\begin{align*}
F & \le\int_{M}((1-t)u^{-3}+4tu^{-4}(\Delta u+\frac{1}{8}R_{g_{0}}u))dv_{g_{0}}\\
 & \le(1-t)(C_{0}F)^{-3}Vol(g_{0})+4t(C_{0}F)^{-4}\int_{M}\Delta u+\frac{1}{8}R_{g_{0}}udv_{g_{0}}\\
 &  =(1-t)(C_{0}F)^{-3}Vol(g_{0})+4t(C_{0}F)^{-4}\int_{M}\frac{1}{8}R_{g_{0}}udv_{g_{0}}\\
 & \le(1-t)(C_{0}F)^{-3}Vol(g_{0})+\frac{1}{2}t(C_{0}F)^{-4}C_{1}F\max R_{g_{0}}Vol(g_{0}),
\end{align*}
and then 
\begin{equation}
F\le C(g_{0},C_{0},C_{1}).\label{eq:upper bound of F and u}
\end{equation}
Therefore, by (\ref{eq:equvalence of u and F})(\ref{eq:lower bound of F and u})(\ref{eq:upper bound of F and u}),
we obtain 
\begin{equation}\label{upperlowerbound of u}
 0<c_{0}\le u\le C_{0}.   
\end{equation}

From (\ref{eq:inner product for 3-dim}) and (\ref{eq:path equation}) we have
\begin{align*}
 & \int_{M}\left[\Delta u\Delta u-4Ric_{g_0}(\nabla u,\nabla u)+\frac{5}{4}R_{g_0}|\nabla u|^{2}-\frac{1}{2}Q_{g_0}u^{2}\right]dv_{g_{0}}\\
= & \int_{M}u(-(1-t)u^{-3}-\frac{1}{2}tu^{-7}R_{g_{u}})dv_{g_{0}}\\
= & -(1-t)\int_{M}u^{-2}dv_{g_{0}}-\frac{1}{2}t\int_{M}8u^{-3}(\Delta u-2u^{-1}|\nabla u|^{2}+\frac{1}{8}R_{g_{0}}u)dv_{g_{0}},
\end{align*}
 which, together with (\ref{eq:Delta u+Ru positive}) and \eqref{upperlowerbound of u}, implies  
\begin{align}
  \int_{M}(\Delta u)^{2}-4Ric_{g_0}(\nabla u,\nabla u)+\frac{5}{4}R_{g_0}|\nabla u|^{2}\nonumber
\le & \int_{M}\frac{1}{2}Q_{g_0}u^{2}+8t\int_{M}u^{-4}|\nabla u|^{2}dv_{g_{0}}\nonumber\\
\le & \frac{C_{0}^{2}}{2}\max Q_{g_{0}}Vol(g_{0})+8tc_{0}^{-4}\int_{M}|\nabla u|^{2}dv_{g_{0}}\label{upperofDEltau}.
\end{align}
Since $u$ is bounded, it is easy to see that
\[
\int_M |\nabla u|^2=\int_M (-\triangle u) u\le C\left(\int_M(\triangle u)^2\right)^{\frac{1}{2}}.
\]
Together with (\ref{upperofDEltau}) and H\"older's inequality, we obtain
\[
\int_{M}(\Delta u)^{2}\le C,
\]
from which, since $n=3$, by the Sobolev inequality we get 
\begin{equation}
\int_{M}|\nabla u|^{6}\le C.\label{eq:initial nabla^p1}
\end{equation}

By (\ref{eq:total expression of path equation}), we have $P_{g_{0}}u\in L^{2}(M,g_{0})$
and then by $W^{4,p}$ theory, $u\in W^{4,2}(M,g_{0})$ and $u\in C^{2,\alpha}(M,g_{0})$ for any $0<\alpha<\frac{1}{2}$
by the Sobolev embedding. By regularity theory, we obtain $|u|_{C^{4,\alpha}(M,g_{0})}\le C_{2}$,
where $C_{2}$ depends only on $g_{0}.$ 

Recall 
\[
\Omega=\{u\in C^{4,\alpha}(M,g_{0})|u>0,\,R_{g_{u}}>0,\,|u|_{C^{4,\alpha}(M,g_{0})}\le C_{2}\}.
\]
We have 
\[
\deg\left(I-T_{1},\Omega,0\right)=\deg\left(I-T_{0},\Omega,0\right)=1,
\]
where 
$$T_t(u)=\int_{M}K(x,y)((1-t)u^{-3}+\frac{1}{2}tu^{-7}R_{g_{u}})(y)dv_{g_{0}}.$$
The $\deg\left(I-T_{0},\Omega,0\right)=1$ can be obtained as Hang-Yang \cite{HYCPAM3-dim} by replacing  $u^{-7}$ there with $u^{-3}$, see the proof of Theorem 1.2 in \cite{HYCPAM3-dim}. 
Therefore, we obtain the existence. 
\end{proof}

From the proof it is clear that the existence in this case is based crucially  on the work of Hang-Yang in \cite{HYCPAM3-dim}.
The main difference to \cite{HYCPAM3-dim} is the choice of 
the power $u^{-1} $ in \eqref{eq:path equation}.  The main reason is 
 to keep the power of $u$  in the two terms in the right-hand side same, which is convenient to obtain the positive lower and upper bound of the solution.

\section{Existence for dimension \texorpdfstring{$ n=4$}{}}\label{section:dim4}
In this section we will show the existence of a metric $g$ with positive scalar curvature satisfying  $Q_{g}=cR_{g}$ in four-dimensional manifolds. 
 Regard to   $Q$-curvature flows in four-dimensional manifolds, there exist at least two types, one is introduced by Brendle \cite{Brendle-Annofmath2003} and another is a non-local flow proposed by Baird-Fardoun-Regbaoui \cite{BairdFR}. Under both flows,  we are unable to show that the positivity of the  scalar curvature is preserved.
  Inspired by the higher-dimensional non-local flow introduced by  Gursky-Malchiodi \cite{GM}, we introduce a non-local flow, which preserves the positivity of the scalar curvature.

On $(M^{4},g_{0})$ with $R_{g_0}\ge 0$, $g=e^{2u}g_{0}$, \[P_{g_0} \varphi=\Delta_{g_0}^2 \varphi-\div\left(\frac{2}{3} R_{g_0} I-2 {Ric_{g_0}}\right) d \varphi\]
and 
\begin{equation}\label{eq4.1_a}
Q_{g}=e^{-4u}(P_{g_{0}}u+Q_{g_{0}}). \end{equation}
The scalar curvature of $R_g$ is given by
\[
\frac{1}{6}R_{g}=e^{-2u}(L_{g_0}u+\frac{1}{6}R_{g_{0}}),
\]
where
\[
L_{g_0} \varphi = -\Delta \varphi -|\n \varphi|^2\] and $\Delta$ and $\nabla$ are the Laplacian and the covariant derivative w.r.t. to $g_0$.

We introduce the following functional
\[
F[u]=\langle P_{g_0} u, u\rangle+2\int Q_{g_0}udv_{g_{0}}-\int Q_{g_{0}}dv_{g_{0}}\log\frac{\int R_{g_{u}}dv_{g_{u}}}{\int R_{g_{0}}dv_{g_{0}}},
\]
where 
\[\langle P_{g_0} u, u\rangle\colon=\int(\triangle u)^2 d v_{g_{0}}+\frac{2}{3} \int R_{g_0}|\nabla u|^2 d v_{g_{0}}-2 \int R_{i j} u_i u_j d v_{g_{0}}\]
and 
\begin{align*}
\int R_{g_{u}}dv_{g_{u}} & =6\int(-\Delta u-|\nabla u|^{2}+\frac{1}{6}R_{g_{0}})e^{2u}dv_{g_{0}}\\
 & =6\int|\nabla u|^{2}e^{2u}dv_{g_{0}}+\int R_{g_{0}}e^{2u}dv_{g_{0}}.
\end{align*}
The critical point of $F[u]$ satisfies
\begin{equation}
P_{g_{0}}u+Q_{g_{0}}-\lambda R_{g_{u}}e^{4u}=0,\label{eq:critical point}
\end{equation}
with
\[ R_{g_u} = e^{-2u} (6L_{g_0} u +R_{g_0}), \qquad\lambda=\frac{\int Q_{g_{0}}dv_{g_{0}}}{\int R_{g_{u}}dv_{g_{u}}}.\]
It is clear that \eqref{eq:critical point} is our equation in the dimension 4 case.
Now we  rewrite the functional $F$ 
\begin{align*}
F[u]
= & \left\langle P_{g_{0}}u,u\right\rangle+2\int Q_{g_{0}}(u-\bar{u})dv_{g_{0}}+\int Q_{g_{0}}dv_{g_{0}}\log\int R_{g_{0}}dv_{g_{0}}\\
 & -\int Q_{g_{0}}dv_{g_{0}}\log(6\int|\nabla u|^{2}e^{2(u-\bar{u})}dv_{g_{0}}+\int R_{g_{0}}e^{2(u-\bar{u})}dv_{g_{0}}).
\end{align*}
where $\bar u$ is the average of $u$  given by
\[
\bar u:=\frac{\int u dv_{g_0}}{Vol(g_0)}.
\]

Gursky \cite{Gursky1999} proved that if $\int Q_{g_0}dv_{g_0}\ge 0$  and $Y(M,[g_0])\ge 0$, then $P_{g_0}$ is nonnegative operator and 
$\operatorname{Ker} P_{g_0}=\{\text{Constants}\}.$ Moreover if $Y(M,[g_0])\ge 0$, then $\int Q_{g_0}dv_{g_0}\le  16\pi^2$ and equality holds if and only if $(M, g_0)$ is conformally equivalent to the round sphere. 
{In view of $\int Q_{g_0}dv_{g_0}=\int \sigma_2(g_0)dv_{g_0}$ in any four-dimensional manifold,
Corollary 4.5 in \cite{CGYAnn} implies  that if  $\int Q_{g_0}dv_{g_0}> 0$  and $Y(M,[g_0])> 0$, then  there exists a metric $g\in [g_0]$ with positive scalar curvature and positive $Q$-curvature. To see this,  one can take $\delta=\frac 2 3$ in \cite[Corollary 4.5]{CGYAnn}, which corresponds to the $Q$ curvature.} 
We remark that $\int Q_{g} dv_g$ is a conformal invariant in $4$-dimensions.

Thanks  the result of Gursky, to  prove Theorem \ref{mainthm} in the case $n=4$, we only need to prove the following theorem.

\begin{thm}\label{4-dim}
    Let $(M^4, g_0)$ be a closed manifold such that  $R_{g_{0}}$ is positive and $Q_{g_0}$ is non-negative. Suppose $0< \int Q_{g_0}dv_{g_0}<16\pi^2$.
Then there exists a conformal metric  
$g\in [g_0]$
 with
a constant quotient $Q_{g}/R_{g}$ and $R_{g}$ is positive.
\end{thm}

Now we modify the flow   introduced by Gursky-Malchiodi for higher dimensional manifolds $n\ge 5$. 
Since $\operatorname{Ker} P_{g_0}=\{\text{Constants}\}$, we can define the inverse operator of $P_{g_{0}}$ in ${X}:=\{u\in C^{\infty}(M,g_0)|\int udv_{g_{0}}=0\}$
as 
\[
P_{g_{0}}^{-1}:{X}\rightarrow{X}.
\]

We introduce the following flow with the smooth initial date $u(x,0)=u_0 \in X$ 
\begin{equation}
u_{t}=-u+P_{g_{0}}^{-1}(r(t)|R_{g_{u}}|e^{4u}-Q_{g_0}), 
\label{eq:4-dim flow}
\end{equation}
where $r(t)=\frac{\int Q_{g_{0}}dv_{g_{0}}}{\int|R_{g_{u}}|e^{4u}dv_{g_{0}}}$ 
and $g_{u_{0}}=e^{2u_{0}}g_{0}$  has positive scalar curvature. It is easy to see that $\int udv_{g_0}=0 $ is preserved. 
Let $g_{u}=e^{2u(x,t)}g_{0}$ and recall 
\[
\frac{1}{6}R_{g_{u}}=e^{-2u}(-\Delta u-|\nabla u|^{2}+\frac{1}{6}R_{g_{0}}).
\]

\begin{lem} \label{Lemma6.2}Flow \eqref{eq:4-dim flow} preserves the positivity of the scalar curvature
along the flow, i.e.,
\[
R_{g_{u}}>0.
\]
\end{lem}

\begin{proof}
Denoting $f=r(t)|R_{g_{u}}|e^{4u}-Q_{g_0}$ and $v=P_{g_{0}}^{-1}f$,
we have 
\[
P_{g_{0}}v+Q_{g_0} =f+Q_{g_0}=r(t)|R_{g_{u}}|e^{4u} \ge 0.
\]
In view of \eqref{eq4.1_a}, we have $Q_{g_v} \ge 0$. Hence $Q_{g_v}$ is semi-positive, since $\int Q_{g_v}>0$.  Thus Theorem \ref{thm:positive scalar curvature} implies that  the scalar
curvature of $g_{v}$ is positive, which in turn implies 
\begin{equation}
-\Delta v-|\nabla v|^{2}+\frac{1}{6}R_{g_{0}}>0.\label{eq:positive scalar along flow}
\end{equation}
Together with the Cauchy inequality and the flow equation $u_t=-u+v$, it follows
\begin{align*}
 \frac 16 \frac d {dt}(R_{g_u}e^{2u})= \frac{d}{dt}(-\Delta u-|\nabla u|^{2}+\frac{1}{6}R_{g_{0}})
= & -\Delta(u_{t})-2\langle\nabla u,\nabla u_{t}\rangle\\
= & -\Delta(-u+v)-2\langle\nabla u,\nabla(-u+v)\rangle\\
= & -(-\Delta u-|\nabla u|^{2}+\frac{1}{6}R_{g_{0}})+|\nabla u|^{2}-2\langle\nabla u,\nabla v\rangle\ +|\n v|^2\\
 & +(-\Delta v-|\nabla v|^{2}+\frac{1}{6}R_{g_{0}})\\
\ge & -(-\Delta u-|\nabla u|^{2}+\frac{1}{6}R_{g_{0}})=-\frac{1}{6}R_{g_{u}}e^{2u}.
\end{align*}
Therefore we have 
\[
e^{t}R_{g_{u}}e^{2u}\ge R_{g_{u_{0}}}e^{2u_{0}}>0.
\]
\end{proof}
We remark that in the above proof, the non-negativity of $Q$-curvature of $e^{2u_0}g_0$ is not necessary.  The following fact will be used later
\begin{equation}\label{positivity}
P_{g_0}u+Q_{g_0}\ge e^{-t}(P_{g_0}u_0+Q_{g_0}).
\end{equation}
In fact, it follows from  
\[\frac{d }{dt}(P_{g_0}u+Q_{g_0})=P_{g_0}(\frac{\partial u}{\partial t})=-P_{g_0}u+r(t)|R_{g_{u}}|e^{4u}-Q_{g_0}> -P_{g_0}u-Q_{g_0}.\]

Since the scalar curvature is positive along the flow,  flow (\ref{eq:4-dim flow}) becomes 
\[
u_{t}=-u+P_{g_{0}}^{-1}(r(t)R_{g_{u}}e^{4u}-Q_{g_0}),
\]
where $r(t)=\frac{\int Q_{g_{0}}dv_{g_{0}}}{\int R_{g_{u}}dv_{g_{u}}}$. 

\begin{lem}
Along the flow, the functional is non-increasing and 
\begin{equation}\label{upper bound of Functional}
    F[u(x,t)]\le F[u_{0}].
\end{equation}

\end{lem}

\begin{proof}
We have
\begin{align}
\frac{d}{dt}F[u] & =2\int u_{t}P_{g_{0}}udv_{g_{0}}+2\int Q_{g_{0}}u_{t}dv_{g_{0}}-\int Q_{g_{0}}dv_{g_{0}}\frac{2\int R_{g_{u}}e^{4u}u_{t}dv_{g_{0}}}{\int R_{g_{u}}dv_{g_{u}}}\nonumber\\
 & =2\int u_{t}(P_{g_{0}}u+Q_{g_{0}}-\frac{\int Q_{g_{0}}dv_{g_{0}}}{\int R_{g_{u}}dv_{g_{u}}}R_{g_{u}}e^{4u})\nonumber\\
 & =-2\int u_{t}P_{g_{0}}u_{t}\le0,\label{dF}
\end{align}
where  we have used $\frac{d}{dt}\int R_{g_u}dv_{g_u}
=2\int R_{g_u}u_tdv_{g_u}.$
\end{proof}

Now we carry out some basic computations for later use. 
From the definition of $P_{g_0}$, we have
\begin{align}
\langle P_{g_0} w, w\rangle &\ge\int(\triangle w)^2 d v_{g_{0}}-C_{g_0} \int |\nabla w|^2 d v_{g_{0}}\nonumber\\
&\ge (1-C_{g_0}\varepsilon)\int(\triangle w)^2 d v_{g_{0}}-C_{\varepsilon}\int w^2\nonumber\\
&\ge (1-C_{g_0}\varepsilon)C_0\left(\int|\nabla  w|^4 d v_{g_{0}}\right)^{\frac1 2}-C_{\varepsilon}\int w^2,\label{lower bound of <Puu>1}
\end{align}
where we have used $\int |\nabla w|^2 d v_{g_{0}}=-\int w\Delta w\le \varepsilon \int (\Delta w)^2+C_{\varepsilon}\int w^2$ and $C_0$ is a positive constant.
 Since the first non-trivial eigenvalue of $P_{g_0}$ is positive,  we have 
\begin{equation}\label{lower bound of <Puu>2}
    \langle P_{g_0} w, w\rangle \ge \lambda_1(P_{g_0})\int (w-\bar w)^2=\lambda_1(P_{g_0})\int w^2 , \qquad \forall w \hbox{ with } \bar w=0.
\end{equation}
Therefore, there exists a positive constant $C_1$ and a fixed small constant $\varepsilon$ such that for $\bar w=0$, 
\begin{equation}\label{lower bound of <Puu>3}
 C_1\langle P_{g_0} w, w\rangle{\ge
  (1-C_{g_0}\varepsilon)\int|\triangle  w|^2 d v_{g_{0}}
  }
 \ge (1-C_{g_0}\varepsilon)C_0\left(\int|\nabla  w|^4 d v_{g_{0}}\right)^{\frac1 2}. 
\end{equation}

\begin{lem}\label{decay inequality for convergence}
Along the flow, 
   \begin{equation*}
       \frac{d}{dt}\int u_tP_{g_0}u_t dv_{g_0}\le  Cr(t) \left((\|u\|_{W^{2,2}}+1)(\int e^{8u}dv_{g_0})^{\frac 1 4} +(\int e^{4u}dv_{g_0})^{\frac 1 2}\right)\int u_tP_{g_0}u_t dv_{g_0},
   \end{equation*}
   where $C$ depends only on $g_0$.
\end{lem}
\begin{proof}
Recalling $r(t)=\frac{\int Q_{g_{0}}dv_{g_{0}}}{\int|R_{g_{u}}|e^{4u}dv_{g_{0}}}
$, we have 
\begin{align}\label{derivatives of r}
    r'(t)
    =-2r(t)\frac{\int R_{g_u}u_tdv_{g_u}}{\int R_{g_u}dv_{g_u}}.  \qquad 
\end{align}
Now 
\begin{align}
\frac{d}{dt}\int u_tP_{g_0}u_t dv_{g_0}
=&2\int \left(-u_t+P_{g_{0}}^{-1}\left(\frac{d}{dt}(r(t)R_{g_{u}}e^{4u})\right)\right)P_{g_0}u_t\nonumber\\
=&-2\int  u_tP_{g_0}u_t+2\int \frac{d}{dt}(r(t)R_{g_{u}}e^{4u})u_t.
\end{align}
Since \[\frac{d}{dt}(R_{g_u}e^{4u})=(2R_{g_u}u_t-6\Delta_{g_u} u_t)e^{4u},\] we get
 \begin{align}
 \frac{d}{dt}(r(t)R_{g_{u}}e^{4u})=&r'(t)R_{g_{u}}e^{4u}+r(t)\frac{d}{dt}(R_{g_{u}}e^{4u})\nonumber\\
 =&-2r(t)\frac{\int R_{g_u}u_tdv_{g_u}}{\int R_{g_u}dv_{g_u}}R_{g_{u}}e^{4u}+r(t)(2R_{g_u}u_t-6\Delta_{g_u} u_t)e^{4u}
\end{align}
and 
{
\begin{align}
&\frac{d}{dt}\int u_tP_{g_0}u_t\\
=&-2\int  u_tP_{g_0}u_t-\frac{4r(t)}{\int R_{g_u}dv_{g_u}}(\int u_t R_{g_u}dv_{g_u})^2+4r(t)\int R_{g_u}u_t^2dv_{g_u}+12r(t)\int |\nabla u_t|_{g_0}^2 e^{2u}dv_{g_0}\nonumber\\
\le &4r(t)\int R_{g_u}u_t^2dv_{g_u}+12r(t)\int |\nabla_{g_0} u_t|_{g_0}^2 e^{2u}dv_{g_0}\nonumber\\
\le & 24r(t)\int (-\Delta u-|\nabla u|^{2}+\frac{1}{6}R_{g_{0}})e^{2u}u_t^2 dv_{g_0}+12r(t)\int |\nabla_{g_0} u_t|_{g_0}^2 e^{2u}dv_{g_0}\nonumber\\
\le &24r(t)\left(\int (-\Delta u-|\nabla u|^{2}+\frac{1}{6}R_{g_{0}})^2 dv_{g_0}\right)^{\frac 12}\left(\int e^{8u}dv_{g_0}\right)^{\frac 14}\left(\int u_t^8 dv_{g_0}\right)^{\frac 14} \\& +12r(t) \left(\int|\nabla_{g_0} u_t|^4 dv_{g_0}\right)^{\frac 12}(\int e^{4u}dv_{g_0})^{\frac 1 2}\nonumber\\
\le& Cr(t) \left( (\|u\|_{W^{2,2}}+1)(\int e^{8u})^{\frac 1 4}+ (\int e^{4u}dv_{g_0})^{\frac 1 2}\right)\int u_tP_{g_0}u_t dv_{g_0}\nonumber
\end{align}
}
where the last inequality holds due to the fact $\int u_t dv_{g_0}=0$ and \eqref{lower bound of <Puu>3}. 

\end{proof}

We recall an inequality which is a variant of Adams' inequality, established by Chang-Yang \cite{ChangYang1995} and  Fontana \cite{Fontana1993}. 

\begin{lem}[Chang-Yang, Fontana]\label{Adams'type inequality}
On $(M, g_0)$, assume $P_{g_0}$ is a nonnegative operator with $\operatorname{Ker} P_{g_0}=\{\text{Constants}\}$. Then there exists a constant $C_2=C_2(M, g_0)$ so that
$$
\int_M e^{32 \pi^2 \frac{(w-\bar{w})^2}{\langle P_{g_0} w, w\rangle}} d v_{g_0} \leq C_2 \operatorname{Vol}(g_0).
$$
\end{lem}

Now we are going to prove Theorem \ref{4-dim}.

\begin{proof}[Proof of Theorem \ref{4-dim}] We prove it in several steps.

\

\noindent{\bf Step 1.} In this step, we aim to get some uniform estimates independent of time. 
    
    We first get a lower bound of the functional $F[u]$. Note that along the flow, $\bar u=0$.
By Lemma \ref{Adams'type inequality}, we have 
\begin{align*}
\int R_{g_{0}}e^{2u}dv_{g_{0}} & \le\max R_{g_{0}}\int e^{2u}dv_{g_{0}}\\
 & \le\max R_{g_{0}}\int e^{\frac{32\pi^{2}u^{2}}{\langle P_{g_{0}}u,u\rangle}+\frac{\langle P_{g_{0}}u,u\rangle}{32\pi^{2}}}dv_{g_{0}} \le CVol(g_{0})\max R_{g_{0}}\cdot e^{\frac{\langle P_{g_{0}}u,u\rangle}{32\pi^{2}}},
\end{align*}
and similarly,
\[
\int e^{4u}dv_{g_{0}}\le CVol(g_{0})\cdot e^{\frac{\langle P_{g_{0}}u,u\rangle}{8\pi^{2}}}, \quad 
\int e^{\alpha u}\le CVol(g_{0})\cdot e^{\frac{\alpha^2\langle P_{g_{0}}u,u\rangle}{128\pi^{2}}}, \quad \forall \alpha\in \R.\]
By the  H\"older inequality and \eqref{lower bound of <Puu>3}, there exists a positive constant $C$ such that
\begin{align*}
\int|\nabla u|^{2}e^{2u}dv_{g_{0}} & \le(\int|\nabla u|^{4}dv_{g_{0}})^{1/2}(\int e^{4u}dv_{g_{0}})^{1/2}\\
 & \le C\langle P_{g_{0}}u,u\rangle e^{\frac{\langle P_{g_{0}}u,u\rangle}{16\pi^{2}}}.
\end{align*}
Hence
\begin{equation}
F[u]\ge(1-\frac{\int Q_{g_{0}}dv_{g_{0}}}{16\pi^{2}})\left\langle P_{g_{0}}u,u\right\rangle+2\int Q_{g_0}udv_{g_{0}}-\int Q_{g_{0}}dv_{g_{0}}\log(\left\langle P_{g_{0}}u,u\right\rangle+1)-C.\label{eq:lower bound of F}
\end{equation}
By \eqref{upper bound of Functional},\eqref{lower bound of <Puu>2} and $\int Q_{g_{0}}dv_{g_{0}}<16\pi^{2}$, there exists a positive constant $C$ depending only on $g_0$ such that  
\[
\langle P_{g_0} u, u\rangle \le C,
\]
and hence by \eqref{lower bound of <Puu>3}
\begin{equation}\label{unform estimateforuW22} \int |\Delta u|^2+|\nabla u|^4\le C.\end{equation}
Therefore, for any $\alpha\in \mathbb R$, {it follows from Lemma \ref{Adams'type inequality} that }
\begin{equation}\label{unform estimatefore}
\int_M e^{\alpha u} d v_{g_0} \leq C(\alpha). \end{equation}
By H\"older's inequality, we have 
\begin{equation}\label{upperbound of integralR}
\int R_{g_u}dv_{g_u}\le C,
\end{equation}
where the constant $C$  is independent of time $t$.

\

\noindent{\bf Step 2:} In this step, we prove the long-time existence.

Now by \eqref{dF}, \eqref{unform estimateforuW22} and \eqref{unform estimatefore},  we have {
\eq{\label{bound}2\int_0^T\int u_{t}P_{g_{0}}u_{t}dv_{g_0}dt=F[u_0]-F[u(x,T)]\le C}
and \[\langle P_{g_0} u, u\rangle+2\int Q_{g_0}udv_{g_{0}}-\int Q_{g_{0}}dv_{g_{0}}\log\frac{\int R_{g_{u}}dv_{g_{u}}}{\int R_{g_{0}}dv_{g_{0}}}\le F[u_0],\]
from which, we obtain $\int R_{g_u}dv_{g_u}
\ge C>0$}. It,
together with \eqref{upperbound of integralR}, implies 
\begin{equation}\label{upper bound of r}
0<C_0(g_0, u_0)\le r(t)\le C(g_0, u_0).
\end{equation}
Due to the positivity of  $P_{g_0}$ in $X$, we have $\int_M\left(P_{g_0} u\right)^2 \ge C\|u\|^2_{W^{4,2}}$ for some positive constant $C$. With \eqref{upper bound of r}, we compute
\begin{align}
\frac{d}{d t}\left(\int_M\left(P_{g_0} u\right)^2 d v_{g_0}\right)&=2 \int_MP_{g_0} u P_{g_0} u_t d v_{g_0}\nonumber\\
&=2 \int_M P_{g_0} u P_{g_0}(-u+P_{g_{0}}^{-1}(r(t)R_{g_{u}}e^{4u}-Q_{g_0})) d v_{g_0}\nonumber\\
&=-2 \int_M (P_{g_0} u)^2 d v_{g_0}+2\int_M P_{g_0} u\cdot( r(t)R_{g_{u}}e^{4u}-Q_{g_0}) d v_{g_0}\nonumber\\
&\le  C_2\int_M (P_{g_0} u)^2 d v_{g_0}+C(g_0,u_0),
\end{align}
from which, we obtain $$e^{-C_2t} \int_M (P_{g_0} u)^2 d v_{g_0}\le Ct+C_1(g_0,u_0)$$
for some positive constant $C$. Now {it follows from Sobolev embedding theorem for any $\alpha\in (0,1)$
\begin{equation}\label{|u|dependent on time-fourdim}
    \|u\|_{C^{1,\alpha}}\le C(t).
\end{equation}
}

 By \eqref{|u|dependent on time-fourdim} \eqref{positivity} and  H\"older's inequality, for any $s\ge 2$,  we have {
 \begin{align*}
    &\frac{d}{d t}\int_M\left(P_{g_0} u+Q_{g_0}-e^{-t}(P_{g_0}u_0+Q_{g_0})\right)^s d v_{g_0}\nonumber\\
= &-s\int (P_{g_0} u+Q_{g_0}-e^{-t}(P_{g_0}u_0+Q_{g_0}))^sd v_{g_0}\nonumber\\
&+sr(t) \int (P_{g_0} u+Q_{g_0}-e^{-t}(P_{g_0}u_0+Q_{g_0}))^{s-1} R_{g_{u}}e^{4u}d v_{g_0}\nonumber \\
\le & Cs\int (P_{g_0} u+Q_{g_0}-e^{-t}(P_{g_0}u_0+Q_{g_0}))^sd v_{g_0}+C \left(\int R_{g_u}^s e^{4su}d v_{g_0}\right)\nonumber\\
\le & Cs\int (P_{g_0} u+Q_{g_0}-e^{-t}(P_{g_0}u_0+Q_{g_0}))^sd v_{g_0}+C(t) \left(\int (|\Delta u|^s+|\nabla u|^{2s}+1) d v_{g_0}\right)\nonumber\\
\le & C(t)s\int (P_{g_0} u+Q_{g_0}-e^{-t}(P_{g_0}u_0+Q_{g_0}))^sd v_{g_0}+C(t), 
\end{align*}
where we have used Young's inequality
\begin{align*}
&(P_{g_0} u+Q_{g_0}-e^{-t}(P_{g_0}u_0+Q_{g_0}))^{s-1} R_{g_{u}}e^{4u}\\
\le &\frac{s-1}{s}(P_{g_0} u+Q_{g_0}-e^{-t}(P_{g_0}u_0+Q_{g_0}))^{s}+ \frac{1}{s}R^s_{g_{u}}e^{4su}
\end{align*}
}
and
$$\int (P_{g_0} u+Q_{g_0}-e^{-t}(P_{g_0}u_0+Q_{g_0}))^sd v_{g_0}\ge \frac 1 2\int |P_{g_0} u|^s d v_{g_0}-C$$ and $\int |P_{g_0}u|^sdv_{g_0}\simeq \|u\|_{W^{4,s}}$ in $X$.

Therefore, we obtain $\|u(\cdot,t)\|_{W^{4,s}}\le C(t)$ for any $s\ge 2$ and then $\|u(\cdot,t)\|_{C^{3,\alpha}}\le C(t)$ for any $0<\alpha<1$, and high regularity follows from the equation. And we obtain  the long-time existence of  flow \eqref{eq:4-dim flow}.

 \ 
 
\noindent{\bf Step 3:} We begin to prove the convergence of the flow.

Let $y(t)=\int u_tP_{g_0}u_t dv_{g_0}$. 
By Lemma \ref{decay inequality for convergence}, together with \eqref{unform estimateforuW22}, \eqref{unform estimatefore} and \eqref{upper bound of r}, we obtain 
\begin{equation}\label{decay}
  y'(t)\le Cy(t).  
\end{equation}
Now we begin to show $y(t)\rightarrow 0$ as $t\rightarrow \infty$ by a similar argument as  in \cite{Brendle-Annofmath2003}. Let $\varepsilon$ be arbitrary small positive number. We can choose $t_0\ge 0$ large such that if $y(t_0)\le \varepsilon$, then $y(t)\le 3\varepsilon$ for all $t\ge t_0$ by a contradiction argument. Define $t_1=\inf\{t\ge t_0: y(t)\ge 3\varepsilon\}$. 
By \eqref{decay}, we have $$2\varepsilon\le y(t_1)-y(t_0)\le C\int_{t_0}^{t_1}y(t)dt<C\int_{t_0}^{\infty} y(t)dt.$$
{
It follows from Lemma \ref{Lemma6.2} and step 1 that $F(u(t))$ is non-increasing and bounded that $\int_{\bar t}^{+\infty} y(t) dt\to 0$ as $\bar t\to +\infty$}. Choosing  $t_0$ large enough such that $C\int_{t_0}^{\infty} y(t)dt\le \varepsilon$, then we obtain $2\varepsilon<\varepsilon$, a contradiction. By \eqref{bound} $\int_0^\infty y(t) dt \le C$, it is now easy to show
 $y(t)\rightarrow 0$ as $t\rightarrow \infty$.
It follows that $-u+P_{g_{0}}^{-1}(r(t)R_{g_{u}}e^{4u}-Q_{g_0})\rightarrow 0$ strongly in  $W^{2,2}$ as $t\rightarrow \infty$.
By \eqref{unform estimateforuW22}, there exists a subsequence of $t_j$ such that $r(t_j)\rightarrow r_{\infty}$, where $r_{\infty}$ is  a positive constant, such that  $u(x, t_j)$ converges to $u_{\infty}\in X$ weakly in $W^{2,2}$-sense {and $e^{4u(x, t_j)}$ converges to $e^{4u_{\infty}}$ strongly in $L^p$-sense for all $p>1$}, where $u_{\infty}\in W^{2,2}$ is a weak solution to $$P_{g_{0}}u_{\infty}+Q_{g_0}=r_{\infty}R_{g_{u_{\infty}}}e^{4u_{\infty }}.$$ 
Now by elliptic regularity theory, $u_{\infty}$ is smooth and $R_{g_{u_\infty}}\ge 0.$  
By Theorem \ref{thm:positive scalar curvature} and $Q_{g_{u_\infty}}\ge 0$, we know  $R_{g_{u_\infty}}> 0.$ We finish the proof.

\end{proof}

\section{Existence for dimension \texorpdfstring{$n\ge 5$}{}}\label{section:5+dim}
In this section, we study the general case, $n\ge 5$, under the assumption $Y_{4,2} (M, [g_0])<Y_{4,2} (\mathbb{S}^n, [g_{\mathbb{S}^n}])$, which was proved in Section \ref{section:various yamabe problems}, provided that $(M,g)$ is not conformally equivalent to the standard sphere.

As in the ordinary Yamabe problem, we divide the proof of the existence into two steps. First we obtain the existence of a metric satisfying the subcritical $Q/R$ equation \eqref{critical equation of subcrtical functional} below by  using a non-local flow, see Theorem \ref{subcritical flow convergence}  in Subsection \ref{subsection:Subcritical existence}.  Then we argue by contradiction to prove the existence of metrics with constant $Q/R$ curvature, under the assumption $Y_{4,2} (M, [g_0])<Y_{4,2} (\mathbb{S}^n, [g_{\mathbb{S}^n}])$,
see Subsection \ref{subsection:Critical Existence}.
It would be interesting to see if there are several  methods,  as in the ordinary Yamabe problem, to prove the main Theorem for $n\ge 5$.  Until now for us this is the only way.

\subsection{Subcritical 
case}\label{subsection:Subcritical existence}

On $(M,g_{0})$ with $Q_{g_{0}}\ge0$ and $R_{g_{0}}>0,$ let $g_{u}=u^{\frac{4}{n-4}}g_{0}$. 
Since $g_{u}=\big(u^{\frac{n-2}{n-4}}\big)^{\frac{4}{n-2}}g_{0}$, we have
\begin{align}
\frac{n-2}{4(n-1)}R_{g_{u}} & =\left(u^{\frac{n-2}{n-4}}\right)^{-\frac{n+2}{n-2}}\big(-{\Delta_{g_0}} u^{\frac{n-2}{n-4}}+\frac{n-2}{4(n-1)}R_{g_{0}}u^{\frac{n-2}{n-4}}\big)\label{eq:R transformation}\\
 & =\left(u^{\frac{n-2}{n-4}}\right)^{-\frac{n+2}{n-2}}L_{g_{0}}u^{\frac{n-2}{n-4}}.\nonumber 
\end{align}
Define  a $W^{2,2}$ inner product as in \cite{GM}: For any $\phi,\psi\in W^{2,2}(M)$,
we have 
\[
\begin{aligned} & \langle\phi,\psi\rangle_{W^{2,2}(g_{0})}=\int_{M}\left(P_{g_{0}}\phi\right)\psi dv_{g_{0}}\\
= & \int_{M}\left[\left(\Delta_{g_{0}}\phi\right)\left(\Delta_{g_{0}}\psi\right)-4A_{g_{0}}(\nabla_{g_0}\phi,\nabla_{g_0}\psi)+(n-2)\sigma_{1}\left(A_{g_{0}}\right)g_{0}(\nabla_{g_0}\phi,\nabla_{g_0}\psi)+\frac{n-4}{2}Q_{g_{0}}\phi\psi\right]dv_{g_{0}}.
\end{aligned}
\]
Since $Q_{g_{0}}$ is semi-positive and $R_{g_{0}}$ is positive,
the self-adjoint operator $P_{g_{0}}$ is positive, see Proposition 2.3 in \cite{GM}. By the positivity of $P_{g_{0}}$, see also \cite[Proposition 3.2]{Robert} or \cite[Proposition 2]{Frederic2011} for example,  we know that 
\[
\|\phi\|_{W^{2,2}}\simeq\int_{M}\phi P_{g_{0}}\phi dv_{g_{0}},
\]
and also 
\begin{equation}
\|\phi\|_{W^{4+k,p}}\simeq\|P_{g_{0}}\phi\|_{W^{k,p}}.\label{eq:normal equalven}
\end{equation}
 The Sobolev embedding theorem gives
\[
(\int_{M}\phi^{\frac{2n}{n-4}})^{\frac{n-4}{2n}}\lesssim\|\phi\|_{W^{2,2}},\quad(\int_{M}|\nabla_{g_0} \phi|^{\frac{2n}{n-2}})^{\frac{n-2}{2n}}\lesssim\|\phi\|_{W^{2,2}}.
\]

For small $0<\varepsilon<1$, we consider the following perturbed functional
\begin{align*}
I_{\varepsilon} (u)=\colon\frac{\int uP_{g_{0}}udv_{g_{0}}}{\left(\int R_{g_{u}}u^{-\varepsilon\frac{2(n-2)}{n-4}}dv_{g_{u}}\right)^{\frac{n-4}{n-2}\frac{1}{1-\varepsilon}}}=\frac{E(u)}{J_{\varepsilon}(u)^{\frac{n-4}{n-2}\frac{1}{1-\varepsilon}}}
\end{align*}
and define 
\[
Y_{\varepsilon}(M,[g_{0}])=\inf_{u\in W^{2,2}(M,g_{0})\cap \mathcal C_{1}[g_{0}]}I_{\varepsilon}(u).
\]
Here $J_\varepsilon$ is a perturbation of  the total scalar curvature  $\int R d v_g$, 
\begin{align*}
J_{\varepsilon}(u) & :=\int R_{g_{u}}u^{-\varepsilon\frac{2(n-2)}{n-4}}dv_{g_{u}}\\
 & =\frac{4(n-1)}{n-2}\int(1-2\varepsilon)u^{-\frac{2(n-2)}{n-4}\varepsilon}|\nabla_{g_{0}}u^{\frac{n-2}{n-4}}|^{2}+\int R_{g_{0}}u^{\frac{2(n-2)}{n-4}(1-\varepsilon)}\\
 & =\frac{4(n-1)(n-2)}{(n-4)^{2}}\int(1-2\varepsilon)|\nabla u|^{2}u^{\frac{2}{n-4}(2-(n-2)\varepsilon)}+\int R_{g_{0}}u^{\frac{2(n-2)}{n-4}(1-\varepsilon)}.
\end{align*}
 The Euler-Lagrange equation of $I_\ve$ is 
\begin{equation}\label{critical equation of subcrtical functional}
P_{g_{0}}u=r_{\varepsilon}u^{\frac{n+4}{n-4}-\frac{2(n-2)}{n-4}\varepsilon}R^{\varepsilon}_{g_{u}},
\end{equation}
where $g=g_u=u^{\frac4{n-4}}g_0$,
\begin{equation}\label{definition of R}
    R^{\varepsilon}_{g_u}:=(1-2\varepsilon)R_{g_u}+\frac{2(n-1)a_{\varepsilon}}{n-2}u^{\frac{-4}{n-4}-2}|\nabla_{g_0} u|^{2}+\frac{2(n-1)b_{\varepsilon}}{n-2}R_{g_{0}}u^{\frac{-4}{n-4}},
\end{equation}
{\begin{eqnarray}
    \label{ab}
    a_{\varepsilon}=\frac{2(n-2)^2\varepsilon}{(n-4)^2}(1-2\varepsilon),
    \quad b_{\varepsilon}=\frac{(n-2)}{2(n-1)}\varepsilon>0 
    \end{eqnarray}} and 
\[
r_{\varepsilon}=\frac{n-4}{2}\frac{\int_{M}Q_{g}dv_{g}}{\int_{M}u^{-\frac{2(n-2)}{n-4}\varepsilon}R^{\varepsilon}_{g}dv_{g}}.
\]

\subsubsection{A subcritical flow}
In order to find a solution to \eqref{critical equation of subcrtical functional} we consider a subcritical flow on $(M,g_0)$: Given an initial smooth data $u(x,0)=u_0$ with a property that  the metric $g_{u_0}=u_0^{\frac{4}{n-4}}g_0$ has positive scalar curvature, we introduce 
\begin{equation}
\frac{4}{n-4}u_{t}=-u+\frac{n-4}{2}rP_{g_{0}}^{-1}\left(|u^{\frac{n+4}{n-4}-\frac{2(n-2)}{n-4}\varepsilon}R^{\varepsilon}_{g_u}|\right),
\label{eq:real subcritcial flow}
\end{equation}
where {\[
r(t):= r_\varepsilon(t)\frac{2}{n-4}=\frac{\int_{M}Q_{g_u}dv_{g_u}}{\int_{M}|u^{-\frac{2(n-2)}{n-4}\varepsilon}R^{\varepsilon}_{g_u}|dv_{g_u}},
\]}
and the operator $P_{g_{0}}^{-1}$
is the inverse of $P_{g_0}$.
Note that \begin{equation}\label{integral of modified R}
\int R^{\varepsilon}_{g_u}u^{-\varepsilon\frac{2(n-2)}{n-4}}dv_{g_{u}}=h_{\varepsilon}\int|\nabla u|^{2}u^{\frac{2}{n-4}(2-(n-2)\varepsilon)}+d_{\varepsilon}\int R_{g_{0}}u^{\frac{2(n-2)}{n-4}(1-\varepsilon)}=(1-\varepsilon)\int R_{g_u}u^{-\varepsilon\frac{2(n-2)}{n-4}} dv_{g_u},
\end{equation}
where {$h_{\varepsilon}=(1-2\varepsilon)(1-\varepsilon)\frac{4(n-1)(n-2)}{(n-4)^{2}}$ and $d_{\varepsilon}=1-\varepsilon.$} 
It is easy to check that $$h_{\varepsilon}\rightarrow \frac{4(n-1)(n-2)}{(n-4)^{2}}, \quad d_{\varepsilon}\rightarrow 1, \hbox{ as } \varepsilon\to 0.$$
Moreover, for {all $\varepsilon\in (0,\frac{1}{2})$, we have
\begin{equation}\label{equvalence of R and tilde R}
2\int R_{g_{u}}u^{-\varepsilon\frac{2(n-2)}{n-4}}dv_{g_{u}}\ge\int R^{\varepsilon}_{g_u}u^{-\varepsilon\frac{2(n-2)}{n-4}}dv_{g_{u}}=(1-\varepsilon)\int R_{g_u}u^{-\varepsilon\frac{2(n-2)}{n-4}}dv_{g_{u}}\ge \frac{1}{2}\int R_{g_{u}}u^{-\varepsilon\frac{2(n-2)}{n-4}}dv_{g_{u}}.
\end{equation}
}
Note that $\int_M Q_{g_{u_0}} dv_{g_{u_0}}>0$ holds naturally due to the positivity of $P_{g_0}$.

\begin{thm}\label{subcritical flow convergence}
Let  $(M,g_{0})$ be a closed manifold with semi-positive $Q_{g_{0}}$ and positive  $R_{g_{0}}$. Assume $\varepsilon\in (0,\frac{2}{n-2})$ with $n\ge 5$. For  any initial smooth positive data $u_0$ such that the scalar curvature of  metric $g_{u_0}=u_0^{\frac{4}{n-4}}g_0$ is positive,  
 flow  (\ref{eq:real subcritcial flow}) has long-time existence and  there exists a subsequence of time $t_j$ going to infinity such that $u(\cdot,t_j)$  converges weakly  
 in  the $W^{2,2}$-sense to a non-trivial smooth positive solution $u_{\varepsilon}$ satisfying the following perturbed equation 
\[
P_{g_{0}}u=\frac{n-4}{2}\bar{r}u^{\frac{n+4}{n-4}-\frac{2(n-2)}{n-4}\varepsilon}R^{\varepsilon}_{g_u},
\]
where $\bar{r}$ is a positive constant.
Moreover, 
\begin{equation}\label{enerydecreasing}
    I_{\varepsilon}(u_{\varepsilon})\le  I_{\varepsilon}(u_{0}).
\end{equation}
 
\end{thm}

Modifying an argument of Gursky-Malchiodi in \cite{GM}, we first show that the scalar curvature along  flow (\ref{eq:real subcritcial flow}) remains positive, and therefore the flow reduces to (\ref{eq:subcritical flow}) below.  
We then prove the long-time existence and the convergence of flow (\ref{eq:real subcritcial flow}), yielding Theorem \ref{subcritical flow convergence}, and  
a nontrivial solution of equation \eqref{subcritical Yamabe equation} (see Theorem \ref{Exsistencesubcritical}).

We first show the short-time existence of (\ref{eq:real subcritcial flow}). 
\begin{lem}
There exists a  $T>0$ such that flow (\ref{eq:real subcritcial flow})
has smooth solution on $M\times[0,T)$.
\end{lem}

\begin{proof} First of all, we have $r(0)>0$ by assumption.
Consider the following normalized flow:
\eq{ \label{eq:subcritical flow0}
\begin{cases}
\frac{\partial v}{\partial t}=\frac{n-4}{4}(-v+\frac{n-4}{2}P_{g_{0}}^{-1}(|R^{\varepsilon}_{g_v}v^{\frac{n+4}{n-4}-\frac{2(n-2)}{n-4}\varepsilon}|))\\
v(\cdot,0)=u_0.
\end{cases}
}
Let $u(x,t)=e^{\frac{n-4}{4}(s(t)-t)}v(x,s(t)),$ where 
\[
\mu(t)=\frac{\int_{M}Q_{g_{v}}dv_{g}}{\int_{M}|v^{-\frac{2(n-2)}{n-4}\varepsilon}R^{\varepsilon}_{g_v}|dv_{g_{v}}},\quad s(t)=\int_{0}^{t}\mu(\tau)d\tau.
\]
Then $u$ satisfies equation (\ref{eq:real subcritcial flow}). The converse is also true. Hence,  both flows are equivalent.

Denote $X_{\epsilon}=C^{4,\alpha}\left(M^{n}\times[0,\epsilon]\right)$. 
 One can show that the mapping
\[
v\mapsto\Psi(v)(x,t)=u_0-\frac{n-4}{4}\int_{0}^{t}v(x,\tau)d\tau+\frac{(n-4)^{2}}{8}\int_{0}^{t}P_{g_{0}}^{-1}\left(|R^{\varepsilon}_{g_v}v^{\frac{n+4}{n-4}-\frac{2(n-2)}{n-4}\varepsilon}|\right)(x,\tau)d\tau
\]
is a contraction on a small neighborhood of $v_{0}\equiv u_0$ in $X_{\epsilon}$
for $\epsilon>0$ small. The short-time existence follows then from the fixed
point theorem.
\end{proof}
By  Theorem
\ref{thm:maximum principle}, 
 we have the following lemma. 
\begin{lem}
\label{lem:lower bound of u}As long as flow \eqref{eq:real subcritcial flow} exists,  we have
\[
u(x,t)>e^{-\frac{n-4}{4}t}u_0(x)>0,\,\,R_{g}>0,
\]
and \[P_{g_{0}}(u-e^{-\frac{n-4}{4}t}u_0)\ge 0.\]
\end{lem}

\begin{proof}
Denote $v:=P_{g_0}^{-1}(\frac{n-4}{2}r|u^{\frac{n+4}{n-4}-\frac{2(n-2)}{n-4}\varepsilon}R^{\varepsilon}_{g_u}|)$. By Theorem \ref{thm:maximum principle}, we know that $v>0$ and the scalar curvature of $g_v=v^{\frac4{n-4}}g_0$ is positive. 
From Lemma \ref{convex} below, we know that $L_{g_0}(u^{\frac 2{n-4}}v)\ge 0$ due to $L_{g_0}(u^{\frac{n-2}{n-4}})\ge 0$ and $L_{g_0}(v^{\frac{n-2}{n-4}})\ge 0$.
We compute the evolution equation for $L_{g_0} u^{\frac {n-2}{n-4}}$
\begin{align}
\frac{\partial}{\partial t}L_{g_0}u^{\frac{n-2}{n-4}}& =\frac{n-2}{n-4}L_{g_{0}}\left(u^{\frac{2}{n-4}}\frac{\partial}{\partial t}u\right)\\
&=\frac{n-2}{4}L_{g_0}\left(u^{\frac{2}{n-4}}\big(-u+P_{g_0}^{-1}(\frac{n-4}{2}r|u^{\frac{n+4}{n-4}-\frac{2(n-2)}{n-4}\varepsilon}R^{\varepsilon}_{g_u}|)\big)\right)\label{eq:partial t P_g u_a}\\
 &= -\frac{n-2}{4}L_{g_0}u^{\frac{n-2}{n-4}} + \frac {n-2} 4 L_{g_0} u^{\frac 2 {n-4}} v
  \ge-\frac{n-2}{4}L_{g_0}u^{\frac{n-2}{n-4}}.\nonumber 
\end{align}
It follows
\[
L_{g_0}u^{\frac{n-2}{n-4}}\ge e^{-\frac{n-2}{4}t}L_{g_0}u^{\frac{n-2}{n-4}}(x,0)=L_{g_{0}}(e^{-\frac{n-2}{4}t}u^{\frac{n-2}{n-4}}(x,0))>0.
\]
By the strong maximum principle, we know that $u(x,t)>e^{-\frac{n-4}{4}t}u_0(x)>0$ and $R_{g}>0$, whenever the flow exists.
Moreover we have
\begin{align}
\frac{\partial}{\partial t}P_{g_{0}}u & =P_{g_{0}}\left(\frac{\partial}{\partial t}u\right)=\frac{n-4}{4}\big(-P_{g_{0}}u+\frac{n-4}{2}r|u^{\frac{n+4}{n-4}-\frac{2(n-2)}{n-4}\varepsilon}R^{\varepsilon}_{g_u}|\big)\label{eq:partial t P_g u}\\
 & \ge-\frac{n-4}{4}P_{g_{0}}u,\nonumber 
\end{align}
and thus 
\[
P_{g_{0}}u\ge e^{-\frac{n-4}{4}t}P_{g_{0}}u_0=P_{g_{0}}(e^{-\frac{n-4}{4}t}u_0),
\]
that is $P_{g_{0}}(u-e^{-\frac{n-4}{4}t}u_0)\ge 0.$
\end{proof}

Along the flow the previous Lemma implies that $R_g>0$. Hence by the   definition  of  $R_{g_u}^{\ve}$ 
\eqref{definition of R}, it holds that
$R_{g_u}^{\ve} >0$,  and 
flow \eqref{eq:subcritical flow0} becomes
\begin{equation}\label{eq:subcritical flow}
\begin{cases}
\frac{4}{n-4}u_{t}=-u+\frac{n-4}{2}rP_{g_{0}}^{-1}\left(u^{\frac{n+4}{n-4}-\frac{2(n-2)}{n-4}\varepsilon}R^{\varepsilon}_{g_u}\right)\\
u(\cdot,0)=u_0,
\end{cases}
\end{equation}
where $$r(t)=\frac{\int_{M}Q_{g}dv_{g}}{\int_{M}u^{-\frac{2(n-2)}{n-4}\varepsilon}R^{\varepsilon}_{g_u}dv_{g}}>0.$$

For the subcritical flow, we have the following monotonicity.
\begin{lem}
\label{thm: functional evolution of t}As long as  flow \eqref{eq:subcritical flow} exits, the flow has the following properties:
\begin{equation}
\frac{\partial}{\partial t}\int R_{g}\cdot u^{-\frac{2(n-2)}{n-4}\varepsilon}dv_{g}=\frac{16(n-2)}{(n-4)^{3}r}\int\frac{\partial u}{\partial t}P_{g_{0}}\frac{\partial u}{\partial t}dv_{g_{0}}\label{eq:W^2, 2 norm of u_t} \ge  0
\end{equation}
and 
\begin{equation}
\frac{d}{dt}\int Q_{g}dv_{g}=0,\,\,\text{and}\,\,\int Q_{g}dv_{g}=\int Q_{g_{u_0}}dv_{g_{u_0}}.\label{eq: intergral of Q constant under the flow}
\end{equation}
Moreover,
\begin{equation}
\frac{d}{dt}\frac{\int uP_{g_{0}}u}{\left(\int R_{g}u^{-\varepsilon\frac{2(n-2)}{n-4}}dv_{g}\right)^{\frac{n-4}{n-2}\frac{1}{1-\varepsilon}}}\le0.\label{eq:functional decreasing}
\end{equation}
\end{lem}

\begin{proof}
A direct computation gives 
\begin{align} \label{derivative of integral Q=0}
\frac{d}{dt}\int Q_{g}dv_{g}
 &=\frac{4}{n-4}\int u_tP_{g_0}u dv_{g_0}\nonumber\\
 &=\int \left(-u+\frac{n-4}{2}rP_{g_{0}}^{-1}\left(u^{\frac{n+4}{n-4}-\frac{2(n-2)}{n-4}\varepsilon}R^{\varepsilon}_{g_u}\right)\right)P_{g_0}u dv_{g_0}\nonumber\\
 & =-\frac{n-4}{2}\int Q_{g}dv_{g}+\frac{n-4}{2}\int1\cdot ru^{-\frac{2(n-2)}{n-4}\varepsilon}R^{\varepsilon}_{g_u}dv_{g}
 =0,
\end{align}
due to the choice of $r=\frac{\int_{M}Q_{g}dv_{g}}{\int_{M}u^{-\frac{2(n-2)}{n-4}\varepsilon}R^{\varepsilon}_{g_u}dv_{g}}$
.  Therefore, $\int Q_{g}dv_{g}=\int Q_{g_{u_0}}dv_{g_{u_0}}$ and  $r>0.$ 
By $\frac{4}{n-4}\frac{u_{t}}{u}=g^{-1}\frac{dg}{dt}$, we have 
\begin{align}
&\frac{d}{dt}\int R_{g_{u}}u^{-\frac{2(n-2)\varepsilon}{n-4}}dv_{g_{u}}\\
= & \frac{(n-2)(1-\varepsilon)}{2}\int R_{g}u^{-\frac{2(n-2)\varepsilon}{n-4}}g^{-1}\frac{dg}{dt}dv_{g_{u}}\label{eq:derivative of integral of R}
 -(n-1)\int\Delta_{g}(g^{-1}\frac{dg}{dt})\cdot u^{-\frac{2(n-2)\varepsilon}{n-4}}dv_{g_{u}}\nonumber \\
= & \frac{2(n-2)(1-\varepsilon)}{(n-4)}\int R_{g}u^{-\frac{2(n-2)\varepsilon}{n-4}}\frac{u_{t}}{u}dv_{g_{u}}
 -\frac{4(n-1)}{n-4}\int\frac{u_{t}}{u}\cdot\Delta_{g}(u^{-\frac{2(n-2)\varepsilon}{n-4}})dv_{g_{u}}.\nonumber 
\end{align}
By \eqref{conformal laplacian transform}, we get
\begin{align*}
L_{g}(u^{-\frac{2(n-2)\varepsilon}{n-4}}) & =-\Delta_{g}u^{-\frac{2(n-2)\varepsilon}{n-4}}+\frac{n-2}{4(n-1)}R_{g}u^{-\frac{2(n-2)\varepsilon}{n-4}}\\
 & =(u^{\frac{n-2}{n-4}})^{-\frac{n+2}{n-2}}L_{g_{0}}(u^{\frac{n-2}{n-4}}u^{-\frac{2(n-2)\varepsilon}{n-4}})\\
 & =u^{-\frac{n+2}{n-4}}(-\Delta_{g_{0}}(u^{\frac{n-2}{n-4}}u^{-\frac{2(n-2)\varepsilon}{n-4}})+\frac{n-2}{4(n-1)}R_{g_{0}}u^{\frac{n-2}{n-4}}u^{-\frac{2(n-2)\varepsilon}{n-4}}).
\end{align*}
Moreover, from the expression of $R_g$ \eqref{eq:R transformation}, we have
\begin{align*}
&-\Delta_{g}(u^{-\frac{2(n-2)\varepsilon}{n-4}})\\
= & -\frac{n-2}{4(n-1)}R_{g}u^{-\frac{2(n-2)\varepsilon}{n-4}}
 +u^{-\frac{n+2}{n-4}}(-\Delta_{g_{0}}(u^{\frac{n-2}{n-4}}u^{-\frac{2(n-2)\varepsilon}{n-4}})+\frac{n-2}{4(n-1)}R_{g_{0}}u^{\frac{n-2}{n-4}}u^{-\frac{2(n-2)\varepsilon}{n-4}})\\
= & u^{-\frac{n+2}{n-4}}\left(-u^{\frac{n-2}{n-4}}\Delta_{g_{0}}u^{-\frac{2(n-2)\varepsilon}{n-4}}-2\langle\nabla u^{\frac{n-2}{n-4}},\nabla u^{-\frac{2(n-2)\varepsilon}{n-4}}\rangle_{g_0}\right)\\
= & u^{-\frac{n+2}{n-4}}\left(\frac{2(n-2)\varepsilon}{n-4}u^{\frac{2}{n-4}-\frac{2(n-2)\varepsilon}{n-4}}\Delta_{g_{0}}u
  -u^{\frac{n-2}{n-4}-\frac{2(n-2)\varepsilon}{n-4}-2}|\nabla u|_{g_0}^{2}\left((\varepsilon^{2}-\varepsilon)(\frac{2(n-2)}{n-4})^{2}+\frac{2(n-2)}{n-4}\varepsilon\right)\right)\\
= & -\frac{\varepsilon(n-2)}{2(n-1)}R_{g}u^{-\frac{2(n-2)\varepsilon}{n-4}} +u^{-\frac{n+2}{n-4}}(u^{\frac{n-2}{n-4}-\frac{2(n-2)\varepsilon}{n-4}-2}|\nabla u|_{g_{0}}^{2}a_{\varepsilon}+b_{\varepsilon}R_{g_{0}}u^{\frac{n-2}{n-4}-\frac{2(n-2)\varepsilon}{n-4}}),
\end{align*}
where $a_{\varepsilon}$ and $b_\varepsilon$ are defined in \eqref{ab}. 
 Hence, we have 
\begin{align*}
  \frac{d}{dt}\int R_{g_{u}}u^{-\frac{2(n-2)\varepsilon}{n-4}}dv_{g_{u}}
=& \frac{2(n-2)}{n-4}\int R^{\varepsilon}_{g_u}u^{\frac{n+4}{n-4}}u^{-\frac{2(n-2)\varepsilon}{n-4}}u_{t}dv_{g_{0}}\\
 =&\frac{4(n-2)}{(n-4)^{2}r}\int P_{g_{0}}(\frac{4}{n-4}u_{t}+u)u_{t}dv_{g_{0}}\\
 =& \frac{16(n-2)}{(n-4)^{3}r}\int u_{t}P_{g_{0}}u_{t}dv_{g_{0}},
\end{align*}
due to \eqref{derivative of integral Q=0}. 
This concludes the proof of the lemma.
\end{proof}

\begin{cor}\label{uniform est}
For any $0<T<\infty$ and $t\in (0,T)$, we have 
\textup{ 
\begin{equation}
r(t)\le C,\quad\int_{M}uP_{g_{0}}udv_{g_{0}}=\frac{n-4}{2}\int_{M}Q_{g_{u_0}}dv_{g_{u_0}},\label{eq:upper bound of r and W2,2 norm}
\end{equation}
\begin{equation}\label{lowerandupperbound of integral R}
\int R_{g_{u_0}}u_0^{-\frac{2(n-2)}{n-4}\varepsilon}dv_{g_{u_0}}\le\int_{M}R_{g}\cdot u^{-\frac{2(n-2)}{n-4}\varepsilon}dv_{g}<\bigg(\frac{\frac{n-4}{2}\int Q_{g_{u_0}}dv_{g_{u_0}}}{Y_{\varepsilon}(M,[g_{0}])}\bigg)^{\frac{(n-2)(1-\varepsilon)}{n-4}},
\end{equation}
\[
\int_{0}^{T}\|f\|_{W^{2,2}}^2dt\le C,\quad\int_{0}^{T}(\int_{M}|f|^{\frac{2n}{n-4}})^{\frac{n-4}{n}}dt\le C,
\]
}where $f\colon=u_{t}=\frac{n-4}{4}\bigg(-u+\frac{n-4}{2}rP_{g_{0}}^{-1}(u^{\frac{n+4}{n-4}-\frac{2(n-2)}{n-4}\varepsilon}R^{\varepsilon}_{g_u})\bigg)$
and $C$ depends only  on $g_{0}$, $\|u_0\|_{W^{2,2}}$, $J_{\varepsilon}(u_0)$ and $n.$
\end{cor}

\begin{proof}
By Lemma \ref{thm: functional evolution of t}, 
we know that 
$$\int R_{g}\cdot u^{-\frac{2(n-2)}{n-4}\varepsilon}dv_{g}\ge \int R_{g_{u_0}}\cdot u_0^{-\frac{2(n-2)}{n-4}\varepsilon}dv_{g_{u_0}}$$
and hence, {it follows from (\ref{equvalence of R and tilde R}) }
$$r\le \frac{1}{1-\varepsilon}\frac{\int Q_{g_{u_0}dv_{g_{u_0}}}}{\int R_{g_{u_0}}\cdot u_0^{-\frac{2(n-2)}{n-4}\varepsilon}dv_{g_{u_0}}}.$$
Since $P_{g_0}$ is positive,  
$Y_{\varepsilon}(M,[g_{0}])>0$, we have 
\begin{align*}
0  < Y_{\varepsilon}(M,[g_{0}])&\le\big(\int R_{g}\cdot u^{-\frac{2(n-2)}{n-4}\varepsilon}dv_{g}\big)^{-\frac{n-4}{n-2}\frac{1}{1-\varepsilon}}\int uP_{g_{0}}udv_{g_{0}}\\
 & =\frac{n-4}{2}\big(\int R_{g}\cdot u^{-\frac{2(n-2)}{n-4}\varepsilon}dv_{g}\big)^{-\frac{n-4}{n-2}\frac{1}{1-\varepsilon}}\int Q_{g}dv_{g},
\end{align*}
and 
\begin{equation}
Y_{\varepsilon}(M,[g_{0}])\big(\int R_{g}\cdot u^{-\frac{2(n-2)}{n-4}\varepsilon}dv_{g}\big)^{\frac{n-4}{n-2}\frac{1}{1-\varepsilon}}\le\frac{n-4}{2}\int Q_{g}dv_{g}=\frac{n-4}{2}\int Q_{g_{u_0}}dv_{g_{u_0}},\label{eq:upper bound of integral R}
\end{equation}
yielding \eqref{lowerandupperbound of integral R}.
Integrating \eqref{eq:W^2, 2 norm of u_t},  we get 
\[
\int_{M}R_{g}\cdot u^{-\frac{2(n-2)}{n-4}\varepsilon}dv_{g}(T)-\int R_{g_{u_0}}u_0^{-\frac{2(n-2)\varepsilon}{n-4}}dv_{g_{u_0}}=\frac{16(n-2)}{(n-4)^{3}}\int_{0}^{T}\frac{1}{r}\int_{M}\frac{\partial u}{\partial t}(P_{g_{0}}\frac{\partial u}{\partial t})dv_{g_{0}}dt.
\]
Together with (\ref{eq:upper bound of integral R}) and $r\le C$, it follows, 
\[
\int_{0}^{T}\|f\|_{W^{2,2}}^2dt<C(g_{0},n, u_0),
\]
{which implies by the Sobolev inequality
\[
\int_{0}^{T}(\int_{M}|f|^{\frac{2n}{n-4}})^{\frac{n-4}{n}}dt\le C(g_{0},n, u_0).
\]
}
\end{proof}
Now we show the long-time existence of conformal subcritical flow
(\ref{eq:subcritical flow}). 
\begin{thm}\label{long-time existence}
{Under the same assumptions as in Theorem \ref{subcritical flow convergence}}, flow  (\ref{eq:subcritical flow}) has  long-time existence. Moreover,  for  any $0<T<\infty $,  the solution $u$ to flow (\ref{eq:subcritical flow})
satisfies 
\[
\|u\|_{C^{\infty}(M)}\le C(T) 
\]
for any $t\in[0,T).$
\end{thm}

\begin{proof}
Denote $v=u-e^{-\frac{n-4}{4}t}u_0>0$.
Lemma \ref{lem:lower bound of u} implies
$P_{g_{0}}v>0$.  By $|P_{g_{0}}v|-C\le |P_{g_{0}}u|\le|P_{g_{0}}v|+C$, we have for any $s>0$
\begin{equation}\label{u-vnormal}
\left(\int |P_{g_{0}}u|^s\right)^{1/s}-C\le \left(\int (P_{g_{0}}v)^s\right)^{1/s} \le \left(\int |P_{g_{0}}u|^s\right)^{1/s}+C.
\end{equation}
A direct computation yields 
\begin{equation}
\begin{aligned}  \frac{d}{dt}\int\left(P_{g_{0}}v\right)^{s}dv_{g_0} &=s\int\left(P_{g_{0}}v\right)^{s-1}\frac{\partial}{\partial t}\left(P_{g_{0}}v\right)dv_{g_0}\\
 & =\frac{s}{4}(n-4)\int\left(P_{g_{0}}v\right)^{s-1}\left\{ -P_{g_{0}}v+\frac{n-4}{2}ru^{\frac{n+4}{n-4}-\frac{2(n-2)}{n-4}\varepsilon}R^{\varepsilon}_{g_u}\right\} dv_{g_0}\\
 & =-\frac{n-4}{4}s\int\left(P_{g_{0}}v\right)^{s}+sr\frac{(n-4)^{2}}{8}\int\left(P_{g_{0}}v\right)^{s-1}(u^{\frac{n+4}{n-4}-\frac{2(n-2)}{n-4}\varepsilon}R^{\varepsilon}_{g_u})dv_{g_0},
\end{aligned}
\label{eq:first derivative}
\end{equation}
where $s>1$ will be determined later. 
H\"older's inequality gives
\begin{equation}
\int(P_{g_{0}}v)^{s-1}u^{\frac{n+4}{n-4}-\frac{2(n-2)}{n-4}\varepsilon}R^{\varepsilon}_{g_u}dv_{g_0}\le  \left(\int(P_{g_{0}}v)^{s}dv_{g_0}\right)^{\frac{s-1}{s}}\left(\int(u^{\frac{n+4}{n-4}-\frac{2(n-2)}{n-4}\varepsilon}R^{\varepsilon}_{g_u})^{s}dv_{g_0}\right)^{\frac{1}{s}}.\label{eq:holder inequality}
\end{equation}
Along the flow, we have
\begin{align}
0 & <u^{\frac{n+4}{n-4}-\frac{2(n-2)}{n-4}\varepsilon}R_{g}\nonumber \\
 & =\frac{4(n-1)}{n-2}u^{\frac{2}{n-4}-\frac{2(n-2)}{n-4}\varepsilon}(-\frac{n-2}{n-4}u^{\frac{2}{n-4}}\Delta u-\frac{2(n-2)}{(n-4)^{2}}u^{\frac{n-2}{n-4}-2}|\nabla u|^{2}+\frac{n-2}{4(n-1)}R_{g_{0}}u^{\frac{n-2}{n-4}}),\nonumber 
 \end{align}
 which implies
 \begin{equation}\label{eq:upper bound of R}
    \frac{8(n-1)}{(n-4)^2}u^{\frac{2}{n-4}-\frac{2(n-2)}{n-4}\varepsilon}u^{\frac{n-2}{n-4}-2}|\nabla u|^{2} \le-\frac{4(n-1)}{n-4}u^{\frac{4}{n-4}-\frac{2(n-2)}{n-4}\varepsilon}\Delta u+R_{g_{0}}u^{\frac{n}{n-4}-\frac{2(n-2)}{n-4}\varepsilon}.
 \end{equation} 
Together with the definition of $R_{g_u}^{\varepsilon}$, \eqref{definition of R},  we have 
\begin{align}
\int(u^{\frac{n+4}{n-4}-\frac{2(n-2)}{n-4}\varepsilon}R^{\varepsilon}_{g_u})^{s}dv_{g_0} & \le C\int(u^{\frac{4}{n-4}-\frac{2(n-2)}{n-4}\varepsilon}|\Delta u|)^{s}+C\int u^{(\frac{n}{n-4}-\frac{2(n-2)}{n-4}\varepsilon)s}.\label{eq:fundamental caluc}
\end{align}
From (\ref{eq:upper bound of r and W2,2 norm}) and the Sobolev inequality,
we know that $\int u^{\frac{4}{n-4}}|\nabla_{g_{0}}u|^{2}+\int u^{\frac{2(n-2)}{n-4}}dv_{g_{0}}\le C$
and $\int u^{\frac{2n}{n-4}}\le C.$
Thus, from (\ref{eq:normal equalven}),  H\"older's inequality and  the Sobolev inequality, we find that for
any $s<\frac{n}{2},$
\begin{align*}
\int(u^{\frac{4}{n-4}-\frac{2(n-2)}{n-4}\varepsilon}|\Delta u|)^{s}dv_{g_0} & \le\big(\int u^{\frac{2n}{n-4}}\big)^{\frac{(2-(n-2)\varepsilon)s}{n}}\big(\int|\Delta u|^{\frac{sn}{n-(2-(n-2)\varepsilon)s}}\big)^{\frac{n-(2-(n-2)\varepsilon)s}{n}}\\
 & \le C\big(\int|\Delta u|^{\frac{sn}{n-2s}}\big)^{\frac{n-2s}{n}}\big(\int u^{\frac{2n}{n-4}}\big)^{\frac{(2-(n-2)\varepsilon)s}{n}}\\
 & \le C|P_{g_{0}}u|_{L^{s}}^{s}.
\end{align*}
{Here we have used the fact $0<\varepsilon<\frac{2}{n-2}$.} For $1<s<\frac{n}{4},$
\[
\int u^{(\frac{n}{n-4}-\frac{2(n-2)}{n-4}\varepsilon)s}dv_{g_0}\le C(\int u^{\frac{sn}{n-4s}})^{\frac{n-4s}{n}}(\int u^{\frac{n(4-2(n-2)\varepsilon)}{4(n-4)}})^{\frac{4s}{n}}\le C|P_{g_{0}}u|_{L^{s}}^{s}.
\]
Therefore, 
\[
\big(\int(u^{\frac{n+4}{n-4}-\frac{2(n-2)}{n-4}\varepsilon}R^{\varepsilon}_{g_u})^{s}dv_{g_0}\big)^{1/s}\le C|P_{g_{0}}u|_{L^{s}}\le C|P_{g_{0}}v|_{L^{s}}+C,
\]
which, together with  \eqref{eq:first derivative} and \eqref{eq:holder inequality},
implies \begin{align*}
\frac{d}{dt}\int(P_{g_{0}}v)^{s}dv_{g_0} & \le-s\frac{n-4}{4}\int(P_{g_{0}}v)^{s}+Cs\int(P_{g_{0}}v)^{s}+Cs\left(\int(P_{g_{0}}v)^{s}\right)^{\frac{s-1}{s}}\\
 & \le Cs\int(P_{g_{0}}v)^{s}+Cs.
\end{align*}
{In the sequel,  positive constants $C,c, C_s, c_s, C_p, c_p$ are independent of $t$. However, their values may vary from  line to line.} It follows that for any $1<s<\frac{n}{4}$, there exists a positive constant $C$ depending on $s$ such that $\int(P_{g_{0}}v)^{s}dv_{g_0}\le Ce^{Cst}.$
Then, thanks to \eqref{u-vnormal} and  the  Sobolev inequality, for any $p>1$, there exist positive constants $C_p$ and $c_p$  such that
\begin{equation}\label{uLpestimate}
    \|u\|_{L^{p}}\le C_pe^{c_{p}t}.
\end{equation}
For any $\frac{n}{4}<s<\frac{n}{2}$, we can choose $\frac{n}{6}\le s_{0}<\frac{n}{4}$ such that
$\frac{ns_{0}}{n-2s_{0}}>s$. 
With the help of \eqref{uLpestimate}, for such $s_0$, by  H\"older's inequality,  there exists $t_{0}<+\infty$ such that 
\[
\int(u^{\frac{4}{n-4}-\frac{2(n-2)}{n-4}\varepsilon}|\Delta u|)^{s}\le\big(\int|\Delta u|^{\frac{ns_{0}}{n-2s_{0}}}\big)^{\frac{s}{\frac{ns_{0}}{n-2s_{0}}}}\big(\int u^{t_{0}}\big)^{1-\frac{s}{\frac{ns_{0}}{n-2s_{0}}}}\le Ce^{ct}|P_{g_{0}}u|_{L^{s_{0}}}^{s}\le Ce^{ct}.
\]
Together with \eqref{eq:fundamental caluc} and \eqref{uLpestimate}, we have
\[
\int(u^{\frac{n+4}{n-4}-\frac{2(n-2)}{n-4}\varepsilon}R^{\varepsilon}_{g_u})^{s}dv_{g_0}\le Ce^{ct}.
\]
Now from (\ref{eq:first derivative}), (\ref{eq:holder inequality}) and (\ref{eq:fundamental caluc}),
we have 
\begin{align*}
\frac{d}{dt}\int(P_{g_{0}}v)^{s}dv_{g_0} & \le-s\frac{n-4}{4}\int\left(P_{g_{0}}v\right)^{s}+Ce^{c_{1}t}sr\big(\int(P_{g_{0}}v)^{s}dv_{g_0}\big)^{\frac{s-1}{s}}\\
 & \le Ce^{Ct}.
\end{align*}
Thus, $\frac{d}{dt}\int(P_{g_{0}}v)^{s}dv_{g_0}\le Ce^{Ct}$ for any
$s<\frac{n}{2}$. This yields $u\in W^{4,s}$ for $s<\frac{n}{2}$
and $\|u\|_{C^{1,\alpha}}\le Ce^{ct}$ for any $0\le\alpha<1$.  

Since  $u\in W^{4,p}$ for any $p<\frac{n}{2}$, we get that
 for any $s>0$, there exists a positive constant $p<\frac{n}{2}$ such that
\[
\big(\int|\Delta u|^{s}\big)^{\frac{1}{s}}\le C|P_{g_{0}}u|_{L^{p}}<C_s e^{c_st}.
\]
Now due to $|u|_{C^{1,\alpha}}\le C$, \eqref{eq:first derivative} and   \eqref{eq:upper bound of R}, for any $s>1$, we have
{
\begin{align*}
\frac{d}{dt}\int\left(P_{g_{0}}v\right)^{s}dv_{g_0} & \le-\frac{n-4}{4}s\int\left(P_{g_{0}}v\right)^{s}+srC\frac{(n-4)^{2}}{8}e^{ct}\int\left(P_{g_{0}}v\right)^{s-1}(|\Delta u|+1)dv_{g_0}\\
 & \le-\frac{n-4}{4}s\int\left(P_{g_{0}}v\right)^{s}+Ce^{ct}\big(\int(P_{g_{0}}v)^{s}dv_{g_0}\big)^{\frac{s-1}{s}}\left(\int|\Delta u|^{s}dv_{g_0}+1\right)^{\frac{1}{s}}\\
 & \le-\frac{n-4}{4}s\int\left(P_{g_{0}}v\right)^{s}+Ce^{ct}\big(\int(P_{g_{0}}v)^{s}dv_{g_0}+1\big)^{\frac{s-1}{s}}\\
 &\le Ce^{ct}\big(\int(P_{g_{0}}v)^{s}dv_{g_0}+1\big)^{\frac{s-1}{s}}
\end{align*}
}
Thus, $\int|P_{g_{0}}u|^{s}dv_{g_0}\le Ce^{Ct}$ for any $s>1$. The
solution $u$ is at least $C^{3,\alpha}$ for any $0<\alpha<1$,  and due to \eqref{eq:partial t P_g u},  we know that $u$ is at least $C^{5}$  and higher regularity follows.
\end{proof}

Recall the definition of a weak solution in Section 1.3 of \cite{Robert} as follows:
 $u \in W^{2,2}(M)$ is a weak solution of $P_g u=f$, with $f \in L^1(M)$, if
$$
\langle\varphi,u\rangle_{W^{2,2}(g_{0})}=\int_M f \varphi d v_{g_0}
$$
for all $\varphi \in C^{\infty}(M)$. We also refer the reader to the regularity theory in Section 1.3 of \cite{Robert} or \cite{robert2024localizationbubblinghighorder}. 

\begin{proof}[Proof of Theorem \ref{subcritical flow convergence}]
From  Corollary \ref{uniform est} and the Rellich-Kondrachov Theorem,
there exists a subsequence
of $t_{j}\uparrow \infty$  and a function $u_{\varepsilon}\in W^{2,2}(M^n)$ such that $u_{j}=u(t_{j},\cdot)$, $r_{j}=r(t_{j})$
satisfies
{
\[
\begin{array}{cl}
r_{j}\rightarrow\bar{r},\\
u_{j}\rightharpoonup u_{\varepsilon} & \text{weakly in }W^{2,2}(M^{n}),\\
 u_{j}\rightarrow u_{\varepsilon} & \text{strongly in }W^{1,p}(M^{n})\quad \forall p<\frac{2n}{n-2}, \text{ and strongly in }L^{q}(M^{n})\quad \forall q<\frac{2n}{n-4}\\
f_{j}\rightarrow0 & \text{strongly in }W^{2,2}(M^{n}),
\end{array}
\]
}
where $$f_{j}\colon=\frac{n-4}{4}\bigg(-u_{j}+\frac{n-4}{2}r_{j}P_{g_{0}}^{-1}|u_{j}^{\frac{n+4}{n-4}-\frac{2(n-2)}{n-4}\varepsilon}R^{\varepsilon}_{g_{u_j}}|\bigg).$$
By \eqref{equvalence of R and tilde R}, \eqref{eq:upper bound of r and W2,2 norm} and  \eqref{lowerandupperbound of integral R}, we have
\[
\bar r=\lim_{j\rightarrow\infty}r_{j}=\lim_{j\rightarrow\infty}\frac{\int Q_{g_{j}}dv_{g_{j}}}{\int_{M}R^{\varepsilon}_{g_{j}}u_{j}^{-\frac{2(n-2)}{n-4}\varepsilon}dv_{g_{j}}}\ge C>0.
\]
It holds that $P_{g_{0}}^{-1}(u_{j}^{\frac{n+4}{n-4}-\frac{2(n-2)}{n-4}\varepsilon}R^{\varepsilon}_{g_{u_{j}}})=\frac{8}{(n-4)^{2}}(f_{j}+\frac{n-4}{4}u_{j})/r_{j}$, that is, $$u_{j}^{\frac{n+4}{n-4}-\frac{2(n-2)}{n-4}\varepsilon}R^{\varepsilon}_{g_{u_{j}}}=\frac{8}{(n-4)^{2}}P_{g_{0}}(f_{j}+\frac{n-4}{4}u_{j})/r_{j}.$$ 
Note that 
$u_{\varepsilon}^{\frac{n-2}{n-4}}\in W^{1,2}$
 is a weak non-negative supersolution to $L_{g_0} u_{\varepsilon}^{\frac{n-2}{n-4}}\ge 0 $. 
 By the Harnack inequality \cite[Theorem 8.18]{GT}, either $u_{\varepsilon}>0$ or $u_{\varepsilon}\equiv 0$. Since we have 
\begin{align*}
\int_{M}R_{g_{j}}u_{j}^{-\frac{2(n-2)}{n-4}\varepsilon}dv_{g_{j}}\rightarrow&\int_{M}R_{g_{u_{\varepsilon}}}u_{\varepsilon}^{-\frac{2(n-2)}{n-4}\varepsilon}dv_{g_{u_{\varepsilon}}}\\
&=\frac{4(n-1)(n-2)}{(n-4)^{2}}\int(1-2\varepsilon)|\nabla u_{\varepsilon}|^{2}u_{\varepsilon}^{\frac{2}{n-4}(2-(n-2)\varepsilon)}+\int R_{g_{0}}u_{\varepsilon}^{\frac{2(n-2)}{n-4}(1-\varepsilon)},
\end{align*}
by Lemma \ref{thm: functional evolution of t}
we conclude that the function
$u_{\varepsilon}$ is non-zero  and positive.

By the lower semi-continuous, we have  $$\liminf_j\int (\Delta u_j)^2 dv_{g_0}\ge \int (\Delta u_{\varepsilon})^2 dv_{g_0}$$
and then 
\begin{equation}\label{energy decrease}
I_{\varepsilon}(u_{\varepsilon})\le \liminf_j
I_{\varepsilon}(u_{j})\le I_{\varepsilon}(u_{0}) .
\end{equation}
From the boundness of $\|u_{j}\|_{W^{2,2}}$, there exists a $q_0$ such that $\frac{2n}{n-4}> q_{0}=\frac{2n}{n+4-2(n-2)\varepsilon}>1$
such that $u_{j}^{\frac{n+4}{n-4}-\frac{2(n-2)}{n-4}\varepsilon}R^{\varepsilon}_{g_{u_{j}}}$ 
is   $L^{q_{0}}$ bounded  by (\ref{eq:upper bound of R}). Hence 
we get that $(f_{j}+\frac{n-4}{2}u_{j})/r_{j}$ is bounded in $W^{4,q_{0}}$
norm. 
It is clear that  $u_{\varepsilon}$ is a weak solution to the following equation
\begin{equation}
P_{g_{0}}u_{\varepsilon}=\frac{n-4}{2}\bar{r}u_{\varepsilon}^{\frac{n+4}{n-4}-\frac{2(n-2)}{n-4}\varepsilon}R^{\varepsilon}_{g_{u_{\varepsilon}}}. \nonumber
\end{equation}

Starting from $q_0$ we find a sequence $\{q_i\}$ by induction as follows.
Once $u_{\varepsilon}^{\frac{n+4}{n-4}-\frac{2(n-2)}{n-4}\varepsilon}R^{\varepsilon}_{g_{u_{\varepsilon}}}$
is $L^{q_{i}}$, we have $u_{\varepsilon}\in W^{4,q_{i}}$ and $\Delta u_{\varepsilon}\in W^{2,q_{i}}$. Then $u_{\varepsilon}^{\frac{n+4}{n-4}-\frac{2(n-2)}{n-4}\varepsilon}R^{\varepsilon}_{g_{u_{\varepsilon}}}$
belongs to $L^{q_{i+1}}$.  To see this, 
we take $q_{i+1}$ such that \begin{align*}
(\frac{4}{n-4}-\frac{2(n-2)}{n-4}\varepsilon)\frac{q_{i+1}}{1-\frac{q_{i+1}(n-2q_{i})}{nq_{i}}} 
 & =\frac{nq_{i}}{n-4q_{i}}\quad\text{for}\quad i\ge 0,
\end{align*}
that is, 
\[
q_{i+1}=\frac{nq_{i}}{(4-2(n-2)\varepsilon)\frac{n-4q_{i}}{n-4}+n-2q_{i}}.
\]
Note that $q_{i}>\frac{(2-(n-2)\varepsilon)n}{n+4-4(n-2)\varepsilon}$ and $q_{i+1}>q_i$, {provided $q_i<\frac{n}{4}$.}
We find that 
\begin{align}
&\int(u_{\varepsilon}^{\frac{4}{n-4}-\frac{2(n-2)}{n-4}\varepsilon}|\Delta u_{\varepsilon}|)^{q_{i+1}} \label{inductionimprovement}\\
&\le\left(\int|\Delta u_{\varepsilon}|^{\frac{nq_{i}}{n-2q_{i}}}\right)^{\frac{q_{i+1}(n-2q_{i})}{nq_{i}}}
  \times\bigg(\int u_{\varepsilon}^{(\frac{4}{n-4}-\frac{2(n-2)}{n-4}\varepsilon)q_{i+1}\frac{1}{1-\frac{q_{i+1}(n-2q_{i})}{nq_{i}}}}\bigg)^{1-\frac{q_{i+1}(n-2q_{i})}{nq_{i}}}.\nonumber
\end{align}
{When $\frac{n}{2}>q_i>\frac{n}{4}$, $u_\varepsilon\in C^{0,\alpha}$ for some $\alpha\in(0,1)$ and $u_{\varepsilon}^{\frac{n+4}{n-4}-\frac{2(n-2)}{n-4}\varepsilon}R^{\varepsilon}_{g_{u_{\varepsilon}}}\in L^{q_{i+1}}$, where $q_{i+1}:=\frac{nq_i}{n-2q_i}>q_i$. }By iterating this argument,
we have $u_{\varepsilon}\in W^{4,q}$ for any $q>0$ and  the smoothness of $u_{\varepsilon}$  follows from the elliptic regularity theory. We conclude  Theorem \ref{subcritical flow convergence}.

\end{proof}

\subsubsection{The achievement of $Y_{\varepsilon}(M,[g_0])$}
Now we begin to prove that $Y_{\varepsilon}(M,[g_0])$ can be achieved by a smooth positive function belonging to $\mathcal C_{Q}.$
\begin{thm}\label{Exsistencesubcritical}
Let $(M, g_0)$ be a closed manifold such that $Q_{g_{0}}$ is semi-positive and $R_{g_{0}}$ is non-negative.
There exists a smooth positive function $\bar{u}_{\varepsilon}$ achieving
the infimum of the functional $I_{\varepsilon}(u)$ in $\mathcal C_{1}[g_{0}]$, i.e., 
\[
I_{\varepsilon}(\bar{u}_{\varepsilon})=Y_{\varepsilon}(M,[g_{0}]),
\]
with the normalization $\int_{M}R_{g_{\bar u_{\varepsilon}}}\bar u_{\varepsilon}^{-\frac{2(n-2)}{n-4}\varepsilon}dv_{g_{\bar u_{\varepsilon}}}=1,$
and satisfying
\begin{equation}\label{subcritical Yamabe equation}
P_{g_{0}}u=r_{\varepsilon}u^{\frac{n+4}{n-4}-\frac{2(n-2)}{n-4}\varepsilon}R^{\varepsilon}_{g_{u}},
\end{equation} 
for some positive constant $r_\varepsilon$.
Moreover, $Q_{g_{\bar{u}_{\varepsilon}}}>0$ and $r_{\varepsilon}=  \frac{1}{1-\varepsilon} Y_{\varepsilon}(M,[g_{0}]).$
\end{thm}

\begin{proof}
By the definition of $Y_{\varepsilon}(M,[g_{0}])$, we know that there exists a sequence
of positive smooth functions $u_{i}$ with the properties that $R_{g_{u_{i}}}>0$, $\lim_{i\rightarrow\infty}I_{\varepsilon}(u_{i})=Y_{\varepsilon}(M,[g_{0}])$, $J_{\varepsilon}(u_i)=1$ and $\|u_i\|_{W^{2,2}}\le C.$
By Theorem
\ref{subcritical flow convergence} and Lemma \ref{thm: functional evolution of t} and taking $u_{i}$ as  initial data of flow \eqref{eq:real subcritcial flow}, one can find a subsequence of  $t_j\rightarrow\infty$ such that $u_i(\cdot, t_j)$ of  flow (\ref{eq:real subcritcial flow})
converges to smooth positive function $u_{i,\infty}$ with $g_{u_{i,\infty}}=u_{i,\infty}^{\frac 4{n-4}} g_0 \in \mathcal C_Q$  satisfying
\begin{equation}
P_{g_{0}}u_{i,\infty}=\frac{n-4}{2}r_{i}u_{i,\infty}^{\frac{n+4}{n-4}-\frac{2(n-2)}{n-4}\varepsilon}R^{\varepsilon}_{g_{u_{i,\infty}}},\label{eq:paroblic limit process}
\end{equation}
 $\|u_{i,\infty}\|_{W^{2,2}}\le C \|u_i\|_{W^{2,2}}$ with a uniform $C$ and $\int_{M}R_{g_{u_{i,\infty}}}u_{i,\infty}^{-\frac{2(n-2)}{n-4}\varepsilon}dv_{g_{u_{i,\infty}}}\ge1$.
By the elliptic regularity theory, we obtain $\|u_{i,\infty}\|_{C^{4, \alpha}}\le C.$
Therefore, there exists a positive smooth function $\bar{u}$
such that $u_{i,\infty}\rightarrow\bar{u}$ in  $C^{4,\alpha}$-sense. Now 
by \eqref{enerydecreasing},  we have 
\[
Y_{\varepsilon}(M,[g_{0}])=\lim_{i\rightarrow\infty}I_{\varepsilon}(u_{i})\ge\lim_{i\rightarrow\infty}I_{\varepsilon}(u_{i,\infty})=I_{\varepsilon}(\bar{u})\ge Y_{\varepsilon}(M,[g_{0}]).
\]
Hence $I_{\varepsilon}(\bar u)=Y_{\varepsilon}(M,[g_{0}]).$
With (\ref{eq:paroblic limit process}), we conclude the theorem.
\end{proof}

\begin{cor}
Let $(M, g_0)$ be a closed manifold such that $Q_{g_{0}}$ is semi-positive and $R_{g_{0}}$ is non-negative. We have
    $$Y_{\varepsilon}(M, [g_0])=\inf_{u\in W^{2,2}(M,g_{0})\cap \mathcal C_{Q}}I_{\varepsilon}(u).$$
\end{cor}
\subsection{Critical case}\label{subsection:Critical Existence}

Now we consider the critical case.
First of all we prove a lemma, which will be used in the blow-up argument.
\begin{lem}
\textup{\label{lem:Sobolev inequality in R^n}Let $v$ be the positive
smooth solution to 
\eq{
\label{entire solution}
\Delta^{2}_{g_{\R^n}}v(x)=-\lambda\frac{4(n-1)}{n-2}v^{\frac{2}{n-4}}\Delta_{g_{\R^n}} v^{\frac{n-2}{n-4}} \quad \hbox{ 
in } \mathbb{R}^{n}} for non-negative constant $\lambda$, and $R_{g_{v}}\ge 0$,  where $g_{v}=v^{\frac 4{n-4}}g_{\mathbb E}$
in $\mathbb{R}^{n}$. Assume that $v$ satisfies $\int_{\mathbb{R}^{n}}|\nabla_{g_{\R^n}} v^{\frac{n-2}{n-4}}|^2dx\le C$, $\int_{\mathbb{R}^{n}} v^{\frac{2n}{n-4}}dx\le C$
and $\int_{\mathbb{R}^{n}}(\Delta_{g_{\R^n}} v)^{2}dx\le C$. Then we have
\[
\int_{\mathbb{R}^{n}}(\Delta_{g_{\R^n}} v)^{2}dx\ge Y_{4,2}(\mathbb{S}^{n})\big(\frac{4(n-1)}{n-2}\int_{\mathbb{R}^{n}}|\nabla_{g_{\R^n}} v^{\frac{n-2}{n-4}}|^{2}dx\big)^{\frac{n-4}{n-2}}.
\]
}
\end{lem}

\begin{proof}
Let $\Phi:\mathbb{S}^{n}\backslash\{N\}\rightarrow\mathbb{R}^{n}$ the stereographic projection,
where $N$ is the north pole. It is well-known that $\Phi$ is a conformal transformation with  $(\Phi^{-1})^{*}(g_{\mathbb{S}^{n}})=\psi^{\frac{4}{n-4}}g_{\mathbb E}$,
where $\psi=(\frac{2}{1+|x|^{2}})^{\frac{n-4}{2}}$.
For $x\in\mathbb{S}^{n}\backslash\{N\}$, denote $\bar{v}(x)=v(\Phi(x))$
and $\bar{\psi}^{-1}(x)=\psi^{-1}(\Phi(x))$.
It holds that
\[
\Phi^{*}(g_{v})=\Phi^{*}(v^{\frac{4}{n-4}}g_{\mathbb E})=(v\psi^{-1}\circ\Phi)^{\frac{4}{n-4}}g_{\mathbb{S}^{n}}
\]
 and 
\begin{equation}
P_{g_{\mathbb{S}^{n}}}(\bar{v}\bar{\psi}^{-1})=\lambda(\bar{v}\bar{\psi}^{-1})^{\frac{n+4}{n-4}}R_{g_{\bar{v}\bar{\psi}^{-1}}}\quad\text{on}\quad\mathbb{S}^{n}\backslash\{N\}.\label{eq:solution on punctured sphere}
\end{equation}
Since $R_{g_{v}}\ge0$ in $\mathbb{R}^{n}$ is equivalent to
$-\Delta_{g_{\R^n}} v^{\frac{n-2}{n-4}}\ge0$
in $\mathbb{R}^{n}$, {we have from fact $v>0$ that}
\[
v^{\frac{n-2}{n-4}}(x)\ge(\min_{\partial B_{1}}v^{\frac{n-2}{n-4}})|x|^{2-n}
\]
for any $x\in\mathbb{R}^{n}\backslash B_{1}.$ That is $v\ge C_{0}|x|^{4-n}$
for some positive constant $C_{0}$. Therefore, $v\psi^{-1}\ge C_{0}|x|^{4-n}(\frac{1+|x|^{2}}{2})^{\frac{n-4}{2}} \ge C_0$
for any $x\in\mathbb{R}^{n}\backslash B_{1}$, and then $v\psi^{-1}\ge C_1$ in $\mathbb{R}^{n}$ for some positive constant $C_1$.

For simplicity, denote $u=\bar{v}\bar{\psi}^{-1}=(v\psi^{-1})\circ\Phi$ and then $u\ge C_{1}$.
It is easy to check  by the conformal invariance 
that on $\mathbb{S}^{n}\backslash\{N\}$,
\begin{equation}
P_{g_{\mathbb{S}^{n}}}u=\lambda\frac{4(n-1)}{n-2}u^{\frac{2}{n-4}}(-\Delta_{\mathbb{S}^{n}}u^{\frac{n-2}{n-4}}+\frac{n-2}{4(n-1)}R_{g_{\mathbb{S}^{n}}}u^{\frac{n-2}{n-4}})=\lambda u^{\frac{n+4}{n-4}}R_{g_{u}},\label{eq:expression on punctured sphere}
\end{equation}
with the following properties: 
\begin{equation}
\int_{\mathbb{R}^{n}}(\Delta_{g_{\R^n}} v)^{2}=\int_{\mathbb{S}^{n}\backslash\{N\}}(\Delta_{g_{\Sn}} u)^{2}+\frac{n^2-2n-4}{2}|\nabla_{g_{\Sn}} u|^2+\frac{n-4}{2}Q_{g_{\mathbb{S}^{n}}}u^{2}dv_{g_{\mathbb{S}^{n}}}<\infty\label{eq:identity relationship 1}
\end{equation}
and 
\begin{equation}
\int_{\mathbb{R}^{n}}|\nabla_{g_{\R^n}} v^{\frac{n-2}{n-4}}|^{2}=\int_{\mathbb{S}^{n}\backslash\{N\}}\langle\nabla_{g_{\Sn}} u^{\frac{n-2}{n-4}},\nabla_{g_{\Sn}} u^{\frac{n-2}{n-4}}\rangle+\frac{n-2}{4(n-1)}R_{g_{\mathbb{S}^{n}}}u^{2\frac{n-2}{n-4}}dv_{g_{\mathbb{S}^{n}}}\le C.\label{eq:identity relationship 2}
\end{equation}
We now prove that $N$ is a removable singularity of $u$ to equation
(\ref{eq:expression on punctured sphere}).

Recall 
\[
P_{g_{\mathbb{S}^{n}}}u=\Delta^{2}_{g_{\Sn}}u-\frac{n^2-2n-4}{2}\Delta_{g_{\Sn}} u+\frac{n-4}{2}Q_{g_{\mathbb{S}^{n}}}u.
\]
We will prove that for any $\varphi\in C^{\infty}(\mathbb{S}^{n})$
with $\|\varphi\|_{W^{2,2}(\mathbb{S}^{n})}\le C,$
\begin{align}
 & \int_{\mathbb{S}^{n}}\Delta_{g_{\Sn}} u\Delta\varphi+\frac{n^2-2n-4}{2}\langle\nabla_{g_{\Sn}} u,\nabla_{g_{\Sn}}\varphi\rangle+\frac{n-4}{2}Q_{g_{\mathbb{S}^{n}}}u\varphi dv_{g_{\mathbb{S}^{n}}}\nonumber \\
= & \lambda\frac{4(n-1)}{n-2}\int_{\mathbb{S}^{n}}\langle\nabla_{g_{\Sn}}\varphi,\nabla_{g_{\Sn}} u^{\frac{n-2}{n-4}}\rangle u^{\frac{2}{n-4}}+\varphi\langle\nabla_{g_{\Sn}} u^{\frac{2}{n-4}},\nabla_{g_{\Sn}} u^{\frac{n-2}{n-4}}\rangle+\frac{n-2}{4(n-1)}R_{g_{\mathbb{S}^{n}}}u^{\frac{n}{n-4}}\varphi,\label{eq:integral sense}
\end{align}
i.e, $u$ is a weak solution on the whole sphere.
Take a cut-off function $\eta_{\varepsilon}$ such that $\eta_{\varepsilon}=0$
in $B_{\varepsilon}(N)$ and $\eta_{\varepsilon}=1$ on $\mathbb{S}^{n}\backslash B_{2\varepsilon}(N)$
with $|\nabla_{g_{\Sn}}\eta_{\varepsilon}|^{2}+|\nabla^{2}_{g_{\Sn}}\eta_{\varepsilon}|\le\frac{C}{\varepsilon^{2}}$
in $B_{2\varepsilon}(N).$ 
Then by \eqref{eq:expression on punctured sphere} we have
\begin{align}
 & \int_{\mathbb{S}^{n}}\Delta_{g_{\Sn}} u\Delta_{g_{\Sn}}(\varphi\eta_{\varepsilon})+\frac{n^2-2n-4}{2}(\nabla_{g_{\Sn}} u,\nabla_{g_{\Sn}}\langle\varphi\eta_{\varepsilon})\rangle+\frac{n-4}{2}Q_{g_{\mathbb{S}^{n}}}u\eta_{\varepsilon}\varphi dv_{g_{\mathbb{S}^{n}}}\nonumber \\
= & \lambda\frac{4(n-1)}{n-2}\int_{\mathbb{S}^{n}}\langle\nabla_{g_{\Sn}}(\varphi\eta_{\varepsilon}u^{\frac{2}{n-4}}),\nabla_{g_{\Sn}} u^{\frac{n-2}{n-4}}\rangle+\frac{n-2}{4(n-1)}R_{g_{\mathbb{S}^{n}}}\varphi\eta_{\varepsilon}u^{\frac{n}{n-4}}.\label{eq:solution entie}
\end{align}
We can show that
\begin{align}
\int_{\mathbb{S}^{n}}\Delta_{g_{\Sn}} u\Delta_{g_{\Sn}}(\varphi\eta_{\varepsilon}) & =\int_{\mathbb{S}^{n}}\eta_{\varepsilon}\Delta_{g_{\Sn}} u\Delta_{g_{\Sn}}\varphi+\int_{\mathbb{S}^{n}}\varphi\Delta_{g_{\Sn}} u\Delta_{g_{\Sn}}\eta_{\varepsilon}+2\int_{\mathbb{S}^{n}}\Delta_{g_{\Sn}} u\langle\nabla_{g_{\Sn}}\varphi,\nabla_{g_{\Sn}}\eta_{\varepsilon}\rangle\label{eq:(1)}\\
 & \rightarrow\int_{\mathbb{S}^{n}}\Delta_{g_{\Sn}} u\Delta_{g_{\Sn}}\varphi.\nonumber 
\end{align}
In fact, since 
\[
\int_{B_{2\varepsilon}}\varphi^{2}\le(\int_{B_{2\varepsilon}}\varphi^{\frac{2n}{n-4}})^{\frac{n-4}{n}}(\int_{B_{2\varepsilon}}1)^{\frac{4}{n}},
\]
 we have 
\begin{align*}
|\int_{\mathbb{S}^{n}}\varphi\Delta_{g_{\Sn}} u\Delta_{g_{\Sn}}\eta_{\varepsilon}| & \le\frac{1}{\varepsilon^{2}}\int_{B_{2\varepsilon}(N)\backslash B_{\varepsilon}(N)}|\varphi\Delta_{g_{\Sn}} u|\\
 & \le\frac{1}{\varepsilon^{2}}(\int_{B_{2\varepsilon}(N)\backslash B_{\varepsilon}(N)}|\Delta_{g_{\Sn}} u|^{2})^{\frac{1}{2}}(\int_{B_{2\varepsilon}\backslash B_{\varepsilon}(N)}\varphi^{2})^{\frac{1}{2}}\\
 & \le C(\int_{B_{2\varepsilon}(N)\backslash B_{\varepsilon}(N)}|\Delta_{g_{\Sn}} u|^{2})^{\frac{1}{2}}(\int_{B_{2\varepsilon}\backslash B_{\varepsilon}(N)}\varphi^{\frac{2n}{n-4}})^{\frac{n-4}{2n}}\rightarrow0.
\end{align*}
Similarly, we have 
\[
|\int_{\mathbb{S}^{n}}\Delta_{g_{\Sn}} u\langle\nabla_{g_{\Sn}}\varphi,\nabla_{g_{\Sn}}\eta_{\varepsilon}\rangle|\le C(\int_{B_{2\varepsilon}(N)\backslash B_{\varepsilon}(N)}|\Delta_{g_{\Sn}} u|^{2})^{\frac{1}{2}}(\int_{B_{2\varepsilon}\backslash B_{\varepsilon}(N)}|\nabla_{g_{\Sn}}\varphi|^{\frac{2n}{n-2}})^{\frac{n-2}{2n}}\rightarrow0.
\]
Moreover, we can prove 
\begin{align}\label{eq:part2}
\int_{\mathbb{S}^{n}}\langle \nabla_{g_{\Sn}} u,\nabla_{g_{\Sn}}(\varphi\eta_{\varepsilon})\rangle 
\rightarrow\int_{\mathbb{S}^{n}}\langle \nabla_{g_{\Sn}} u,\nabla_{g_{\Sn}}\varphi\rangle,\nonumber 
\end{align}
due to 
\[
|\int_{\mathbb{S}^{n}}\langle\nabla_{g_{\Sn}} u,\nabla_{g_{\Sn}}\eta_{\varepsilon}
\rangle \varphi|\le C\varepsilon\rightarrow0,
\]
and 
\begin{equation}
\int_{\mathbb{S}^{n}}\langle\nabla_{g_{\Sn}}(\varphi\eta_{\varepsilon}),\nabla_{g_{\Sn}} u^{\frac{n-2}{n-4}}\rangle u^{\frac{2}{n-4}}\rightarrow\int_{\mathbb{S}^{n}}\langle\nabla_{g_{\Sn}}\varphi,\nabla_{g_{\Sn}} u^{\frac{n-2}{n-4}}\rangle u^{\frac{2}{n-4}},\label{eq:part 3}
\end{equation}
since
\begin{align*}
 & |\int_{\mathbb{S}^{n}}\varphi\langle\nabla_{g_{\Sn}}\eta_{\varepsilon},\nabla_{g_{\Sn}} u^{\frac{n-2}{n-4}}\rangle u^{\frac{2}{n-4}}|\\
 & \le\frac{1}{\varepsilon}(\int_{B_{2\varepsilon}(N)\backslash B_{\varepsilon}(N)}|\nabla_{g_{\Sn}} u^{\frac{n-2}{n-4}}|^{2})^{\frac{1}{2}}(\int_{B_{2\varepsilon}(N)\backslash B_{\varepsilon}(N)}\varphi^{2}u^{\frac{4}{n-4}})^{\frac{1}{2}}\\
 & \le\frac{1}{\varepsilon}(\int_{B_{2\varepsilon}(N)\backslash B_{\varepsilon}(N)}|\nabla_{g_{\Sn}} u^{\frac{n-2}{n-4}}|^{2})^{\frac{1}{2}}(\int_{B_{2\varepsilon}(N)\backslash B_{\varepsilon}(N)}u^{\frac{2n}{n-4}})^{\frac{1}{n}}\big(\int_{B_{2\varepsilon}(N)\backslash B_{\varepsilon}(N)}\varphi^{\frac{2n}{n-4}}\big)^{\frac{n-4}{2n}}\varepsilon\\
 & \le(\int_{B_{2\varepsilon}(N)\backslash B_{\varepsilon}(N)}|\nabla_{g_{\Sn}} u^{\frac{n-2}{n-4}}|^{2})^{\frac{1}{2}}(\int_{B_{2\varepsilon}(N)\backslash B_{\varepsilon}(N)}u^{\frac{2n}{n-4}})^{\frac{1}{n}}\big(\int_{B_{2\varepsilon}(N)\backslash B_{\varepsilon}(N)}\varphi^{\frac{2n}{n-4}}\big)^{\frac{n-4}{2n}}\rightarrow0.
\end{align*}
Recalling that $u$ has a positive lower bound, we have 
\begin{align}
\int_{\mathbb{S}^{n}}\eta_{\varepsilon}\varphi\langle\nabla_{g_{\Sn}} u^{\frac{2}{n-4}},\nabla_{g_{\Sn}} u^{\frac{n-2}{n-4}}\rangle & =\int_{\mathbb{S}^{n}}\eta_{\varepsilon}\varphi\frac{n-2}{2}u\langle\nabla_{g_{\Sn}} u^{\frac{2}{n-4}},\nabla_{g_{\Sn}} u^{\frac{2}{n-4}}\rangle\label{eq:part 4}\\
 & \rightarrow\int_{\mathbb{S}^{n}}\varphi\frac{n-2}{2}u\langle\nabla_{g_{\Sn}} u^{\frac{2}{n-4}},\nabla_{g_{\Sn}} u^{\frac{2}{n-4}}\rangle.\nonumber 
\end{align}
Therefore, by (\ref{eq:(1)}), (\ref{eq:part 3}), (\ref{eq:part 4})
and (\ref{eq:solution entie}), we obtain (\ref{eq:integral sense}). 
Namely $u$ is a weak solution on $\Sn$ and hence it is a smooth solution by a regularity result proved in Appendix, Theorem \ref{thm8.1}.  By Theorem \ref{thm1.1} we have inequality \eqref{inq1.11} for $u$, i.e.,
$$\int_{\mathbb{S}^{n}}u P_{g_{\mathbb{S}^{n}}}udv_{g_{\mathbb{S}^{n}}}\ge Y_{4,2}(\mathbb S^n) \left(\int R_{g_{u}}dv_{g_{u}}\right)^{\frac{n-4}{n-2}} .$$

Now we transform this inequality from $\Sn$ back to $\Rn$ to finish the proof of the Lemma. In fact,
by (\ref{eq:identity relationship 1})
and (\ref{eq:identity relationship 2}), it holds that 
\[
\int_{\mathbb{R}^{n}}(\Delta v)^{2}=\int_{\mathbb{S}^{n}}uP_{g_{\mathbb{S}^{n}}}udv_{g_{\mathbb{S}^{n}}}
\]
and 
\[
\int_{\mathbb{R}^{n}}|\nabla_{g_{\Rn}} v^{\frac{n-2}{n-4}}|^{2}=\int_{\mathbb{S}^{n}}\langle\nabla_{g_{\Sn}} u^{\frac{n-2}{n-4}},\nabla_{g_{\Sn}} u^{\frac{n-2}{n-4}}\rangle+\frac{n-2}{4(n-1)}R_{g_{\mathbb{S}^{n}}}u^{2\frac{n-2}{n-4}}dv_{g_{\mathbb{S}^{n}}}.
\]
Therefore, we complete
the proof of this lemma.
\end{proof}

It is an open question to  classify all positive solutions of
\eqref{entire solution} as mentioned  in Section 3.

Now
 we recall the definition of $Y_{4,2}$ and have the following equivalent definition 
\[
Y_{4,2}(M,[g_{0}])=\inf_{u\in W^{2,2}(M,g_{0})\cap \mathcal C_Q[g_{0}]}\frac{\int uP_{g_{0}}udv_{g_{0}}}{\left(\int R_{g_{u}}dv_{g_{u}}\right)^{\frac{n-4}{n-2}}}>0.
\]
Let $w_{i}\in W^{2,2}\cap \mathcal C_Q[g_{0}]$ such that 
\[
I(w_{i})\rightarrow Y_{4,2}(M,[g_{0}])
\]
and 
\[
\lim_\varepsilon Y_{\varepsilon}(M,[g_{0}])\le \lim_\varepsilon I_{\varepsilon}(w_{i})= I(w_{i}).
\]
Taking $i\to\infty$, we have
\begin{equation}
\lim_{\varepsilon\rightarrow0}Y_{\varepsilon}(M,[g_{0}])\le Y_{4,2}(M,[g_{0}]).\label{eq:Yamabe constant relationship}
\end{equation}

Note that by the Sobolev inequality and the positivity of $P_{g_0}$,  there exists a positive constant $C_0$ such that  for any small $\varepsilon>0$
\begin{equation}
Y_{\varepsilon}(M,[g_{0}])\ge C_{0}>0.\label{eq:positive lower bound of Yamabe constant}
\end{equation}
{
Indeed, we have
\begin{align*}
& \int R_{g_{u}}^\varepsilon u^{-\varepsilon\frac{2(n-2)}{n-4}}dv_{g_{u}}
=(1-\varepsilon)\int R_{g_{u}}^\varepsilon u^{-\varepsilon\frac{2(n-2)}{n-4}}dv_{g_{u}}\\
= & (1-\varepsilon)\frac{4(n-1)(n-2)}{(n-4)^{2}}\int(1-2\varepsilon)|\nabla u|^{2}u^{\frac{2}{n-4}(2-(n-2)\varepsilon)}+(1-\varepsilon)\int R_{g_{0}}u^{\frac{2(n-2)}{n-4}(1-\varepsilon)}\\
\le & C\|u\|_{W^{2,2}}^{\frac{2(n-2)}{n-4}(1-\varepsilon)}\le C_1 \left(\int uP_{g_0}u\right)^{\frac{(n-2)}{n-4}(1-\varepsilon)}
\end{align*}
where $C, C_1$ are positive constants independent of $\varepsilon$. Thus, the above claim is proved. Now, we want to prove the existence result.
}

\begin{thm}
\label{thm:existence under Yamabe restriction}Let $\left(M^{n},g_{0}\right)$
be a closed Riemannian manifold of dimension $n\geq5$, with properties that $Q_{g_{0}}$ is semi-positive and $R_{g_{0}}\geq0$.
If  \[Y_{4,2}(M,[g_{0}])<Y_{4,2}(\mathbb{S}^{n},[g_{\mathbb{S}^{n}}])\]
then there exists a positive smooth function $u$ and positive constant
$\lambda_{0}$ such that $g_{u}=u^{\frac{4}{n-4}}g_{0}\in \mathcal C_Q[g_{0}]$
satisfies
\[
P_{g_{0}}u=\lambda_{0}u^{\frac{n+4}{n-4}}R_{g_{u}}.
\]
Moreover, the solution $u$ achieves $Y_{4,2}(M,[g_{0}])$.
\end{thm}

\begin{proof}
By Theorem \ref{Exsistencesubcritical}, {for any $\varepsilon\in (0,\frac{2}{n-2})$}, there exists a positive smooth solution $u_{\varepsilon}$ and $g_{u_\ve}=u_{\ve}^{\frac 4{n-4}}g_0\in \mathcal C_{Q}$
satisfying 
\begin{equation}\label{approximation}
P_{g_{0}}u_{\varepsilon}=r_{\varepsilon}u_{\varepsilon}^{\frac{n+4}{n-4}-\frac{2(n-2)}{n-4}\varepsilon}R^{\varepsilon}_{g_{u_{\varepsilon}}}
\end{equation}
with 
\begin{equation}\label{normaliztion}
    \int_{M}R_{g_{u_{\varepsilon}}}u_{\varepsilon}^{-\frac{2(n-2)}{n-4}\varepsilon}dv_{g_{u_{\varepsilon}}}=1
\end{equation}
and $r_{\varepsilon}=  \frac{1}{1-\varepsilon} Y_{\varepsilon}(M,[g_{0}]).$ 
By \eqref{eq:Yamabe constant relationship}, \eqref{equvalence of R and tilde R} and \eqref{eq:positive lower bound of Yamabe constant}, we know that 
\begin{equation}\label{limit of r}
   \frac{1}{2}C_0 \le \lim_{\varepsilon\rightarrow0}r_{\varepsilon}\le Y_{4,2}(M,[g_{0}]).
\end{equation} 
Denote $u_{i}\colon=u_{\varepsilon_{i}}$
and   
\begin{align}N_i:=&\max_M{\left(\left(u_{{i}}^{\frac{n+4}{n-4}-\frac{2(n-2)}{n-4}\varepsilon_i}R^{\varepsilon_i}_{g_{u_{{i}}}}\right)^\frac{1}{\frac{n+4}{n-4}-\frac{4(n-2)}{n-4}\varepsilon_i}+u_i\right)}\nonumber\\
=&\left(\left(u_{{i}}^{\frac{n+4}{n-4}-\frac{2(n-2)}{n-4}\varepsilon_i}R^{\varepsilon_i}_{g_{u_{{i}}}}\right)^\frac{1}{\frac{n+4}{n-4}-\frac{4(n-2)}{n-4}\varepsilon_i}+u_i\right)(x_i).
\end{align}
If $N_{i}$ is bounded, then by $L^{p}$ theory, we obtain $\|u_{i}\|_{W^{4,p}}<C$ for any $1<p<\infty$,
and there exists a positive function $u_{\infty}\in C^{3,\alpha}$
such that $u_{i}\rightarrow u_{\infty}$ in $C^{3,\alpha}$-sense
and the positivity of $u_{\infty}$ follows from the Harnack inequality and $L_{g_0}u_{\infty}^{\frac{n-2}{n-4}}\ge 0$.
 Moreover, by Schauder theory, $\|u_{i}\|_{C^{4,\alpha}}<C$. For the  regularity theory, we refer to a lecture note by  Robert \cite{Robert} or \cite{robert2024localizationbubblinghighorder}.  Thus,
$u_{i}$ smoothly converges to $u_{\infty}$  in $M$ and $u_{\infty}$ satisfies
\[
P_{g_{0}}u=\lambda_{0}u^{\frac{n+4}{n-4}}R_{g_{u}}.
\]
{From the positivity of the operator $P_{g_{0}}$, we infer that the scalar curvature and the $Q$-curvature of $g_{u_\infty}$ are positive. Hence,
$$
\lim_i Y_{\varepsilon}(M,[g_{0}])=\lim_i I_{\varepsilon_i}(u_i)=I(u_\infty)\ge  Y_{4,2}(M,[g_{0}]).
$$
Together with} \eqref{eq:Yamabe constant relationship}, it holds that $$\lim_{\varepsilon\rightarrow0}Y_{\varepsilon}(M,[g_{0}])= Y_{4,2}(M,[g_{0}]).$$
Therefore, we remain to exclude the following case
\[
N_{i}\rightarrow\infty \quad \text{as}\quad \varepsilon_i\rightarrow 0.
\]

For this case, 
define the rescaled function 
\[
v_{i}=\frac{1}{N_{i}}u_{i}(\exp_{x_{i}}(\lambda_{i}^{-1}x))=\colon\frac{1}{N_{i}}u_{i}(\Psi_{i}(x))\quad\text{in }B_{\lambda_{i}\delta_{0}}(0)
\]
where $\delta_{0}$ is the injective radius of $M$, $\lambda_{i}\colon=N_{i}^{\frac{2}{n-4}-\frac{(n-2)}{n-4}\varepsilon_i}\rightarrow\infty$.
It is clear that $v_{i}\le1$ in $B_{\lambda_{i}\delta_{0}}(0)$. 
Note that
\[
\Psi_{i}(g_{u_{i}})=u_{i}(\Psi_{i}(x))^{\frac{4}{n-4}}\Psi_{i}(g_{0}):=u_{i}(\Psi_{i}(x))^{\frac{4}{n-4}}\lambda_{i}^{-2}g_{i},
\]
where $\Psi_{i}(g_{0})=\lambda_{i}^{-2}g_{i}$ and that $g_{i}$ converges
to $g_{\mathbb E}$ smoothly in any compact set as $\lambda_{i}\rightarrow\infty.$
It holds that \begin{align}\label{pullback of gradient}
\Psi_i\left(|\nabla u_i|_{g_0}^{2}\right)
=|\nabla_{\Psi_i(g_0)}u_{i}\circ\Psi_{i}|^2   =\lambda_i^2|\nabla_{g_i}(u_i\circ \Psi)|^2.
\end{align}
By (\ref{eq:conformal P transformation}),
we have
\begin{align}\label{pullback of Paneizte}
\Psi_{i}(P_{g_{0}}u_{i})  =P_{\Psi_{i}(g_{0})}u_{i}\circ\Psi_{i}=\lambda_{i}^{\frac{n+4}{2}}P_{g_{i}}(\lambda_{i}^{-\frac{n-4}{2}}u_{i}\circ\Psi_{i})=\lambda_{i}^{4}P_{g_{i}}(u_{i}\circ\Psi_{i}),
\end{align}
and by (\ref{conformal laplacian transform}),
\begin{align}
R_{\Psi_{i}(g_{u_{i}})} 
 & =\frac{4(n-1)}{n-2}\left((\lambda_{i}^{-\frac{n-4}{2}}u_{i}\circ\Psi_{i})^{\frac{n-2}{n-4}}\right)^{-\frac{n+2}{n-2}}L_{g_{i}}(\lambda_{i}^{-\frac{n-4}{2}}u_{i}\circ\Psi_{i})^{\frac{n-2}{n-4}}\nonumber \\
 & =\frac{4(n-1)}{n-2}\lambda_{i}^{2}(u_{i}\circ\Psi_{i})^{-\frac{n+2}{n-4}}L_{g_{i}}(u_{i}\circ\Psi_{i})^{\frac{n-2}{n-4}}.\label{eq:transformation} 
\end{align}
Moreover, together with \eqref{pullback of gradient},
we have  \begin{align}
 &R^{\varepsilon_i}_{\Psi_{i}(g_{u_{i}})}=R^{\varepsilon_i}_{(\lambda_i^{-\frac{n-4}{2}}u_{i}\circ\Psi_{i})^{\frac{4}{n-4}}g_i}\nonumber\\
 =&(1-2\varepsilon_i)R_{\Psi_{i}(g_{u_{i}})}+\frac{2(n-1)b_{\varepsilon_i}}{(n-2)}R_{g_i}u_i^{-\frac{4}{n-4}}\lambda_i^2+\lambda_i^2u_i^{\frac{-4}{n-4}-2}|\nabla_{g_i} u_i|^{2}\frac{2(n-1)a_{\varepsilon_i}}{(n-2)}\nonumber\\
 =&\lambda_i^2 R^{\varepsilon_i}_{(u_{i}\circ\Psi_{i})^{\frac{4}{n-4}}g_i}.
 \end{align}
 Denote  $g_{v_{i}} =v_{i}^{\frac{4}{n-4}}g_{i}$.  It is clear that
\begin{equation}\label{relationship of metric}
g_{v_{i}}=\frac{\lambda_{i}^{2}}{N_{i}^{\frac{4}{n-4}}}\Psi_{i}(g_{u_{i}}).
\end{equation}
One can check 
\begin{align}
 v_{i}{}^{\frac{n+4}{n-4}-\frac{2(n-2)}{n-4}\varepsilon_i}R_{g_{v_{i}}}
= & \frac{4(n-1)}{n-2}v_{i}{}^{\frac{2}{n-4}-\frac{2(n-2)}{n-4}\varepsilon_i}L_{g_{i}}v_{i}{}^{\frac{n-2}{n-4}}\nonumber \\
= & \frac{4(n-1)}{n-2}N_{i}^{-(\frac{2}{n-4}-\frac{2(n-2)}{n-4}\varepsilon_i)-\frac{n-2}{n-4}}u_{i}{}^{\frac{2}{n-4}-\frac{2(n-2)}{n-4}\varepsilon_i}L_{g_{i}}u_{i}{}^{\frac{n-2}{n-4}}\nonumber \\
= & (u_{i}\circ\Psi_{i})^{\frac{n+4}{n-4}-\frac{2(n-2)}{n-4}\varepsilon_i}R_{\Psi_{i}(g_{u_{i}})}\lambda_{i}^{-2}N_{i}^{-(\frac{2}{n-4}-\frac{2(n-2)}{n-4}\varepsilon_i)-\frac{n-2}{n-4}}\nonumber \\
= & N_{i}^{-\frac{n+4}{n-4}+\frac{4(n-2)}{n-4}\varepsilon_i}(u_{i}\circ\Psi_{i})^{\frac{n+4}{n-4}-\frac{2(n-2)}{n-4}\varepsilon_i}R_{\Psi_{i}(g_{u_{i}})}
\end{align}
where the third equality holds due to (\ref{eq:transformation}) and
the fourth equality  due to the choice of $\lambda_{i}$.

Similarly, we have 
\begin{equation}\label{relationshipwithviandui}
v_{i}{}^{\frac{n+4}{n-4}-\frac{2(n-2)}{n-4}\varepsilon_i}R_{g_{v_{i}}}^{\varepsilon_i}=N_{i}^{-\frac{n+4}{n-4}+\frac{4(n-2)}{n-4}\varepsilon_i}(u_{i}\circ\Psi_{i})^{\frac{n+4}{n-4}-\frac{2(n-2)}{n-4}\varepsilon_i}R^{\varepsilon_i}_{\Psi_{i}(g_{u_{i}})}.
\end{equation}
By  \eqref{pullback of Paneizte} and \eqref{approximation}, we have
\begin{align*}
\lambda_{i}^{4}P_{g_{i}}(u_{i}\circ\Psi_{i}) & =\Psi_{i}(P_{g_{0}}u_{i})=r_{\varepsilon_i}\Psi_{i}(u_{i}^{\frac{n+4}{n-4}-\frac{2(n-2)}{n-4}\varepsilon}R^{\varepsilon_i}_{g_{u_{i}}})\\
 &=r_{\varepsilon_i}(u_{i}\circ\Psi_{i})^{\frac{n+4}{n-4}-\frac{2(n-2)}{n-4}\varepsilon}R^{\varepsilon_i}_{\Psi_{i}(g_{u_{i}})},
\end{align*}
and hence,
\[
P_{g_{i}}(u_{i}\circ\Psi_{i})=\lambda_{i}^{-4}r_{\varepsilon_i}(u_{i}\circ\Psi_{i})^{\frac{n+4}{n-4}-\frac{2(n-2)}{n-4}\varepsilon}R^{\varepsilon_i}_{\Psi_{i}(g_{u_{i}})}.
\]
Together with \eqref{relationshipwithviandui}, due to the choice of $v_{i}$, $N_i$ and $\lambda_{i}$, we find
\begin{align}
P_{g_{i}}v_{i} & =r_{\varepsilon_i}v_i^{\frac{n+4}{n-4}-\frac{2(n-2)}{n-4}\varepsilon}R^{\varepsilon_i}_{v_i^{\frac{4}{n-4}}g_i}.
 \label{eq:rescaled equation}
\end{align}
By \eqref{relationshipwithviandui}, we know
\begin{align}
   v_{i}{}^{\frac{n+4}{n-4}-\frac{2(n-2)}{n-4}\varepsilon_i}R^{\varepsilon_i}_{g_{v_{i}}}
   &\le \left(\max_{M} u_{\varepsilon_{i}}^{\frac{n+4}{n-4}-\frac{2(n-2)}{n-4}\varepsilon_i}R^{\varepsilon_i}_{g_{u_{\varepsilon_{i}}}}\right)^{-1}(u_{i}\circ\Psi_{i})^{\frac{n+4}{n-4}-\frac{2(n-2)}{n-4}\varepsilon_i}R^{\varepsilon_i}_{\Psi_{i}(g_{u_{i}})}\le1.\label{eq:upper bound}
\end{align}
where the
last inequality is due to the choice of $N_{i}.$
By the definition of $N_i$ and \eqref{relationshipwithviandui} again,
\begin{equation}\label{normalize at a point}
\left(v_i^{\frac{n+4}{n-4}-\frac{2(n-2)}{n-4}\varepsilon_i}R^{\varepsilon_i}_{g_{v_i}}\right)^\frac{1}{\frac{n+4}{n-4}-\frac{4(n-2)}{n-4}\varepsilon_i}+v_i(0)=1.
\end{equation}
By (\ref{eq:rescaled equation}) and (\ref{eq:upper bound}), with
$0\le v_{i}\le1$, in any compact set $K\subset\mathbb{R}^{n}$, we
know that for any $1<p<\infty$
\[
\|v_{i}\|_{W^{4,p}(K)}\le C_{K,p},
\]
and thus,  there exists a non-negative function $v_{\infty}\in C^{3,1}_{loc}(\mathbb{R}^{n})$
 such that $v_{i}$ converges to $v_{\infty}$
in $C_{loc}^{3,\alpha}$-sense for any $0<\alpha<1$. It holds that $v_{\infty}^{\frac{n-2}{n-4}}$
is at least $C^{1}$ globally in $\mathbb{R}^{n}$.  
Since $R_{g_{v_{\infty}}}(x)\ge0$,
we have $-\Delta v_{\infty}^{\frac{n-2}{n-4}}\ge0$ in weak sense. 
By the Harnack inequality, we know $v_{\infty}^{\frac{n-2}{n-4}}>0$
 or $v_{\infty}=0$ in $\mathbb{R}^{n}$. The latter   is impossible. 
In fact, 
in view of  \eqref{normalize at a point} and \eqref{definition of R},
we have {
\begin{align*}
1=&\left(v_i^{\frac{n+4}{n-4}-\frac{2(n-2)}{n-4}\varepsilon_i}R^{\varepsilon_i}_{g_{v_i}}\right)^\frac{1}{\frac{n+4}{n-4}-\frac{4(n-2)}{n-4}\varepsilon_i}+v_i(0)\\
\le & \left(-\frac{4(n-1)}{n-4}(1-2\varepsilon_i)v_i^{\frac{4}{n-4}-\frac{2(n-2)}{n-4}\varepsilon_i}\Delta v_i+(\frac{2(n-1)b_{\varepsilon_i}}{n-2}+1-2\varepsilon_i)R_{g_{i}}v_i^{\frac{n}{n-4}-\frac{2(n-2)}{n-4}\varepsilon_i}\right)^\frac{1}{\frac{n+4}{n-4}-\frac{4(n-2)}{n-4}\varepsilon_i}+v_{i}(0) \\
\rightarrow &\left(-\frac{4(n-1)}{n-4}v_{\infty}^{\frac{4}{n-4}}\Delta v_{\infty}\right)^{\frac{n-4}{n+4}}+v_{\infty}(0).
\end{align*}
Here we have used (\ref{eq_scalar}), (\ref{ab}) and the fact that
$$
\frac{8(n-1)(1-2\varepsilon_i)}{(n-4)^2}-\frac{2(n-1)a_{\varepsilon_i}}{n-2}>0
$$
provided $\varepsilon_i$ is sufficiently small. It is known that $v_\infty\ge 0$ and $R_{g_{v_{\infty}}}\ge 0$. By the strong maximum principle,  either $v_\infty>0$ or $v_\infty\equiv 0$. By the previous property,  we know that $v_\infty$ is a positive smooth function. On the other hand,   since $\frac{1}{2(n-1)} \Delta_{g_{v_\infty}} R_{g_{v_\infty}} \leq d(n)R_{g_{v_\infty}}^2$ with $d(n)=\frac{n^2-4}{8n(n-1)^2}$,   by the strong maximum principle, we infer that either $R_{g_{v_\infty}} >0$ or $
R_{g_{v_\infty}} \equiv 0$. In the later case, we have $v_\infty^{\frac{n-2}{n-4}}$ is bounded harmonic function, and then, it is a positive constant function on $\Rn$, which contradicts the following fact which will be proved below
$$
\int_{\Rn} v_\infty^{\frac{2n}{n-4}}\le C.
$$
Hence, the second case does not occur. }

By \eqref{limit of r}, denoting $\lambda=\lim_{i\rightarrow \infty}r_{\varepsilon_{i}}$, we have
$$ \lambda\le Y_{4,2}(M,[g_0]).$$
Hence $v_\infty$ satisfies   
\begin{equation}
\Delta^{2}_{g_{\Rn}}v_{\infty}(x)=-\lambda\frac{4(n-1)}{n-2}v_{\infty}^{\frac{2}{n-4}}\Delta_{g_{\Rn}} v_{\infty}^{\frac{n-2}{n-4}}, \quad \hbox{ in }\Rn\label{eq:limit equation}
\end{equation}
 with $\left(-\frac{4(n-1)}{n-2}v_{\infty}^{\frac{2}{n-4}}\Delta v_{\infty}^{\frac{n-2}{n-4}}\right)^{\frac{n-4}{n+4}}+v_\infty(0)=1$
and {$R_{g_{v_{\infty}}}> 0.$}
By (\ref{relationshipwithviandui}) and \eqref{relationship of metric}, we have 
\begin{align*}
 & \int_{B_{\frac{\delta_{0}}{2}\lambda_{i}}(0)}R^{\varepsilon_i}_{g_{v_{i}}}v_{i}^{-\frac{2(n-2)}{n-4}\varepsilon_{i}}dv_{g_{v_{i}}}\\
= & \int_{B_{\frac{\delta_{0}}{2}\lambda_{i}}}N_{i}^{-\frac{n+4}{n-4}+\frac{4(n-2)}{n-4}\varepsilon_{i}}(u_{i}\circ\Psi_{i})^{\frac{n+4}{n-4}-\frac{2(n-2)}{n-4}\varepsilon_i}R^{\varepsilon_i}_{\Psi_{i}(g_{u_{i}})}v_{i}^{-\frac{n+4}{n-4}}\frac{\lambda_{i}^{n}}{N_{i}^{\frac{2n}{n-4}}}\Psi_{i}(dv_{g_{u_{i}}})\\
= & N_{i}^{-(n-2)\varepsilon_{i}}\int_{B_{\frac{\delta_{0}}{2}\lambda_{i}}}(u_{i}\circ\Psi_{i})^{\frac{n+4}{n-4}-\frac{2(n-2)}{n-4}\varepsilon_i}R^{\varepsilon_i}_{\Psi_{i}(g_{u_{i}})}u_{i}^{-\frac{n+4}{n-4}}\Psi_{i}(dv_{g_{u_{i}}})\\
= & N_{i}^{-(n-2)\varepsilon_{i}}\int_{\Psi_{i}(B_{\frac{\delta_{0}}{2}\lambda_{i}})}u_{i}{}^{\frac{n+4}{n-4}-\frac{2(n-2)}{n-4}\varepsilon_i}R^{\varepsilon_i}_{g_{u_{i}}}u_{i}^{-\frac{n+4}{n-4}}dv_{g_{u_{i}}}\\
\le & N_{i}^{-(n-2)\varepsilon_{i}}\int_{M}u_{i}{}^{-\frac{2(n-2)}{n-4}\varepsilon_i}R^{\varepsilon_i}_{g_{u_{i}}}dv_{g_{u_{i}}}.
\end{align*}

Combining  Fatou's lemma, \eqref{equvalence of R and tilde R} and (\ref{normaliztion}), we obtain that 
\[
\int_{\mathbb{R}^{n}}R_{g_{v_{\infty}}}dv_{g_{v_{\infty}}}\le1.
\]

Note that

\begin{align*}
 & \frac{4(n-1)(n-2)(1-2\varepsilon_{i})}{(n-4)^{2}}\int_{B_{\frac{\delta_{0}\lambda_{i}}{2}}}|\nabla_{g_{i}}v_{i}|_{g_{i}}^{2}v_{i}^{\frac{2}{n-4}(2-(n-2)\varepsilon_{i})}dv_{g_{i}}+\int_{B_{\frac{\delta_{0}\lambda_{i}}{2}}}R_{g_{i}}v_{i}^{\frac{2(n-2)}{n-4}(1-\varepsilon_{i})}dv_{g_{i}}\\
= & \frac{4(n-1)(n-2)(1-2\varepsilon_{i})}{(n-4)^{2}}\int_{B_{\frac{\delta_0}{2}\lambda_{i}}(0)}|\nabla_{\Psi_{i}(g_{0})\lambda_{i}^{2}}\frac{1}{N_{i}}u_{i}(\Psi_{i}(x))|_{\Psi_{i}(g_{0})\lambda_{i}^{2}}^{2}N_{i}^{-\frac{2(2-(n-2)\varepsilon_{i})}{n-4}}u_{i}^{\frac{2(2-(n-2)\varepsilon_{i})}{n-4}}dv_{\Psi_{i}(g_{0})\lambda_{i}^{2}}\\
 & +\int R_{\Psi_{i}(g_{0})\lambda_{i}^{2}}u_{i}^{\frac{2(n-2)}{n-4}(1-\varepsilon_{i})}N_{i}^{-\frac{2(n-2)}{n-4}(1-\varepsilon_{i})}dv_{\Psi_{i}(g_{0})\lambda_{i}^{2}}\\
= & N_{i}^{-(n-2)\varepsilon_{i}}\frac{4(n-1)(n-2)(1-2\varepsilon_{i})}{(n-4)^{2}}\int_{B_{\frac{\delta_0}{2}r_{i}(0)}}|\nabla_{\Psi_{i}(g_{0})}u_{i}(\Psi_{i}(x))|^{2}u_{i}^{\frac{2(2-(n-2)\varepsilon_{i})}{n-4}}dv_{\Psi_{i}(g_{0})}\\
 & +N_{i}^{-(n-2)\varepsilon_{i}}\int_{B_{\frac{\delta_0}{2}r_{i}(0)}}R_{\Psi_{i}(g_{0})}u_{i}^{\frac{2(n-2)}{n-4}(1-\varepsilon_{i})}dv_{\Psi_{i}(g_{0})}\\
\le & N_{i}^{-(n-2)\varepsilon_{i}}\left(\frac{4(n-1)(n-2)(1-2\varepsilon_{i})}{(n-4)^{2}}\int_{M}|\nabla_{g_{0}}u_{i}|^{2}u_{i}^{\frac{2}{n-4}(2-(n-2)\varepsilon_{i})}dv_{g_{0}}+\int R_{g_{0}}u_{i}^{\frac{2(n-2)}{n-4}(1-\varepsilon_{i})}dv_{g_{0}}\right)\\
= & N_{i}^{-(n-2)\varepsilon_{i}}\int_{M}u_{i}{}^{-\frac{2(n-2)}{n-4}\varepsilon_i}R_{g_{u_{i}}}dv_{g_{u_{i}}}.
\end{align*}
Using Fatou's lemma again, we obtain that 
\begin{equation}
\frac{4(n-1)(n-2)}{(n-4)^{2}}\int_{\mathbb{R}^{n}}|\nabla v_{\infty}|^{2}v_{\infty}^{\frac{4}{n-4}}dx\le1.\label{eq:int R=00005Cle1}
\end{equation}

Similarly, one can find that
\begin{align*}
\int_{B_{\frac{\delta_{0}\lambda_{i}}{2}}}v_{i}P_{g_{i}}v_{i}(x)dv_{g_{i}} & =N_{i}^{-(n-2)\varepsilon_i}\int_{\Psi_{i}(B_{\frac{\delta_{0}\lambda_{i}}{2}}(0))}u_{i}(P_{g_{0}}u_{i})(y)dv_{g_{0}}\\
 & \le N_{i}^{-(n-2)\varepsilon_i}\int_{M}u_{i}(P_{g_{0}}u_{i})(y)dv_{g_{0}},
\end{align*}
where the last inequality holds due to $P_{g_{0}}u_{i}>0$ and $u_{i}>0$
in $M$. By the assumption of this theorem and (\ref{normaliztion}) and \eqref{equvalence of R and tilde R},
we have
\[
\int_{M}u_{i}(P_{g_{0}}u_{i})dv_{g_{0}}=\int_{M}r_{\varepsilon_i}u_{i}^{-\frac{2(n-2)}{n-4}\varepsilon_i}R^{\varepsilon_i}_{g_{u_{i}}}dv_{g_{u_{i}}}<\infty.
\]
And now Fatou's lemma implies
\begin{equation}
\int_{\mathbb{R}^{n}}v_{\infty}\Delta^{2}v_{\infty}\le C.\label{eq:integral}
\end{equation}
Using an argument similar to (\ref{eq:integral}) and (\ref{eq:int R=00005Cle1}),
we obtain 
\[
\int_{\mathbb{R}^{n}}(\Delta v_{\infty})^{2}\le C.
\]
Similarly, 
due to the positivity of $P_{g_0}$,
\begin{align*}
 & \int_{B_{\frac{\delta_{0}}{2}\lambda_{i}}(0)}1dv_{g_{v_{i}}}= \int_{B_{\frac{\delta_{0}}{2}\lambda_{i}}}\frac{\lambda_{i}^{n}}{N_{i}^{\frac{2n}{n-4}}}\Psi_{i}(dv_{g_{u_{i}}})= N_{i}^{\frac{-(n-2)n\varepsilon_{i}}{n-4}}\int_{B_{\frac{\delta_{0}}{2}\lambda_{i}}}\Psi_{i}(dv_{g_{u_{i}}})\\
\le  & \int_{M} 1dv_{g_{u_{i}}}\le C(\int_{M} u_iP_{g_0} u_idv_{g_{0}})^{\frac{n}{n-4}}\le C, 
\end{align*}
and it yields from Fatou's lemma again that 
\[
\int_{\mathbb{R}^{n}} v_{\infty}^{\frac{2n}{n-4}} dx\le C.
\]

Take the cut-off function $\eta\in C_{0}^{\infty}(\mathbb{R}^{n})$
such that $0\le\eta\le1$, $\eta=1$ on $B_{1}$ and $\eta=0$ on
$\mathbb{R}^{n}\backslash B_{2}$ and let $v_{\infty,R}(x)=\eta(\frac{x}{R})v_{\infty}(x)$.
Then we  have
\[
\int_{\mathbb{R}^{n}}|\Delta(v_{\infty}-v_{\infty,R})|^{2}+|\nabla(\eta (\frac{x}{R})v_{\infty}^{\frac{n-2}{n-4}}-v_{\infty}^{\frac{n-2}{n-4}})|^{2}\rightarrow0,\quad\text{as}\ R\rightarrow\infty.
\]
Now by \eqref{eq:limit equation} one can prove
\[
\int_{\mathbb{R}^{n}}(\Delta_{g_{\Rn}} v_{\infty})^{2}dx=\lambda\frac{4(n-1)}{n-2}\int_{\mathbb{R}^{n}}|\nabla_{g_{\Rn}}  v_{\infty}^{\frac{n-2}{n-4}}|^{2}dx.
\]

Since $0<\lambda\le Y_{4,2}(M,[g_{0}])$, we obtain
\begin{equation}\label{upper bound of delta v}
\int_{\mathbb{R}^{n}}(\Delta_{g_{\Rn}}  v_{\infty})^{2}dx\le Y_{4,2}(M,[g_{0}])\frac{4(n-1)}{n-2}\int_{\mathbb{R}^{n}}|\nabla_{g_{\Rn}}  v_{\infty}^{\frac{n-2}{n-4}}|^{2}dx.
\end{equation}
Since $v_{\infty}$ is the solution to (\ref{eq:limit equation})
satisfying $\int_{\mathbb{R}^{n}}R_{g_{v_{\infty}}}dv_{g_{\infty}}\le1$
and $\int_{\mathbb{R}^{n}}(\Delta v_{\infty})^{2}dv_{g}\le C$, by
Lemma \ref{lem:Sobolev inequality in R^n}, we have
\[
\int_{\mathbb{R}^{n}}(\Delta_{g_{\Rn}}  v_{\infty})^{2}dx\ge Y_{4,2}(\mathbb{S}^{n},[g_{\mathbb{S}^{n}}])\big(\frac{4(n-1)}{n-2}\int_{\mathbb{R}^{n}}|\nabla_{g_{\Rn}}  v_{\infty}^{\frac{n-2}{n-4}}|^{2}dx\big)^{\frac{n-4}{n-2}}.
\]
It, together with \eqref{upper bound of delta v}, implies 
\[
Y_{4,2}(M,[g_{0}])\frac{4(n-1)}{n-2}\int_{\mathbb{R}^{n}}|\nabla v_{\infty}^{\frac{n-2}{n-4}}|^{2}dx\ge Y_{4,2}(\mathbb{S}^{n},[g_{\mathbb{S}^{n}}])\big(\frac{4(n-1)}{n-2}\int_{\mathbb{R}^{n}}|\nabla_{g_{\Rn}}  v_{\infty}^{\frac{n-2}{n-4}}|^{2}dx\big)^{\frac{n-4}{n-2}}.
\]
Hence we obtain 
\[
Y_{4,2}(\mathbb{S}^{n},[g_{\mathbb{S}^{n}}])\le Y_{4,2}(M,[g_{0}])\big(\frac{4(n-1)}{n-2}\int_{\mathbb{R}^{n}}|\nabla v_{\infty}^{\frac{n-2}{n-4}}|^{2}dx\big)^{\frac{2}{n-2}}\le Y_{4,2}(M,[g_{0}]),
\]
where the last inequality holds due to (\ref{eq:int R=00005Cle1}). 
This contradicts  the assumption $Y_{4,2}(\mathbb{S}^{n},[g_{\mathbb{S}^{n}}])>Y_{4,2}(M,[g_{0}])$, and hence $N_{i}$ is bounded. We finish the proof.
\end{proof}

{We define  $\tilde Y_{4,2}$ as follows 
\[
\tilde Y_{4,2}(M,[g_{0}])=\inf_{\mathcal C_1[g_{0}]}\frac{\int uP_{g_{0}}udv_{g_{0}}}{\left(\int R_{g_{u}}dv_{g_{u}}\right)^{\frac{n-4}{n-2}}}.
\]
It is clear that $\tilde Y_{4,2}(M,[g_{0}])\le Y_{4,2}(M,[g_{0}])$.  We establish the following result.

\begin{thm}
\label{thm:existence in 1 cone}Let $\left(M^{n},g_{0}\right)$
be a closed Riemannian manifold of dimension $n\geq5$, with properties that $Q_{g_{0}}$ is semi-positive and $R_{g_{0}}\geq0$. Then
$$
\tilde Y_{4,2}(M,[g_{0}])=  Y_{4,2}(M,[g_{0}]).
$$
Moreover, $\tilde Y_{4,2}(M,[g_{0}])$ is achieved by some $g_{u}=u^{\frac{4}{n-4}}g_{0}\in \mathcal C_Q[g_{0}]$ with a positive smooth function $u$ which
satisfies
\[
P_{g_{0}}u=\lambda_{0}u^{\frac{n+4}{n-4}}R_{g_{u}}
\]
for some positive constant $\lambda_{0}$.
\end{thm}
\begin{proof}
The desired results follow from Corollary \ref{main_cor} and Theorem \ref{thm:existence under Yamabe restriction}. Finally, we finish the proof.

\end{proof}
}

\appendix
\section{Appendix}
\label{appendix}

\subsection{The regularity of $W^{2,2}$ weak solutions}

We adopt the argument of Uhlenbeck and Viaclovsky in \cite{UV} to prove
the smoothness of a weak solution. The key lemma is the following
$\varepsilon$-regularity.
\begin{lem}\label{e-regularity}
Let $u\in W^{2,2}$ be a positive weak solution to 
\[
P_{g_0}u=\lambda u^{\frac{2}{n-4}}L_{g_0}u^{\frac{n-2}{n-4}}.
\]
Let $x_{0}\in M$ and  $B_{1}$  a unit ball in a Riemannian
normal coordinate system $\{x\}$ around $x_{0}$. Then there exists
an $\epsilon>0$ such that

\[
\|\Delta u\|_{L^{2}\left(B_{1}\right)}+\|\nabla u\|_{L^{\frac{2n}{n-2}}\left(B_{1}\right)}+\|u\|_{L^{\frac{2n}{n-4}}(B_1)}<\epsilon
\]
and
\begin{equation}
\max_{i,j}\left\Vert g_0^{ij}-\delta^{ij}\right\Vert _{L^{\infty}\left(B_{1}\right)}<\epsilon\label{eq:almost Euc}
\end{equation}
imply that $u\in W^{2,q}\left(B_{1/2}\right)$ for some $2<q<\frac{2n}{n-2}$.
\end{lem}

\begin{proof}
Let $\phi\in C_{0}^{\infty}(B_{1})$, $\phi=1$ in $B_{\frac{1}{2}}$.
Denote $h=\phi u$.

We first \emph{claim:} For any $2<q<\frac{2n}{n-2}$,

\[
P_1 h\colon=\Delta_{E}^{2}h+L_1 h+\frac{n-2}{n-4}\lambda u^{\frac{4}{n-4}}\Delta h+\frac{2(n-2)}{(n-4)^{2}}\lambda\frac{n-4}{4}\langle\nabla h,\nabla u^{\frac{4}{n-4}}\rangle=f\in L_{-2}^{q}
\]
where 
$\Delta_{E}$ denotes the Euclidean Laplacian operator,
\[
L_1 h=\sum_{i,j,k,l}\partial_{i}\partial_{j}\left(a^{ij,kl}\partial_{k}\partial_{l}h\right),
\]
and
$\ensuremath{a^{ij,kl}(x)=g_0^{ij}(x)g_0^{kl}(x)-\delta^{ij}\delta^{kl}}$
and $f$ is the remaining term. 

By computations, we have
\begin{align}
\Delta^{2}(\phi u)= & \phi\Delta^{2}u+u\Delta^{2}\phi+2\langle\nabla u,\nabla(\Delta\phi)\rangle+2\Delta\langle\nabla\phi,\nabla u\rangle+2\div(\nabla\phi\cdot\Delta u)\nonumber \\
= & \phi(P_{g_0}u-\div(B\nabla u)-\frac{n-4}{2}Q_{g_0}u)\nonumber \\
 & +u\Delta^{2}\phi+2\langle\nabla u,\nabla(\Delta\phi)\rangle+2\Delta\langle\nabla\phi,\nabla u\rangle+2\div(\nabla\phi\cdot\Delta u)\nonumber \\
= & \phi\lambda u^{\frac{2}{n-4}}L_{g_0}u^{\frac{n-2}{n-4}}+u\Delta^{2}\phi+2\langle\nabla u,\nabla(\Delta\phi)\rangle+2\Delta\langle\nabla\phi,\nabla u\rangle+2\div(\nabla\phi\cdot\Delta u)\label{eq:fourth order equation}\\
&-\phi\div(B\nabla u)-\phi\frac{n-4}{2}Q_{g_0}u\nonumber
\end{align}
 where $\ensuremath{B:=4A_{g_{0}}-(n-2)\sigma_{1}\left(g_{0}\right)g_{0}}.$ Moreover, we have
\begin{align*}
\phi\lambda u^{\frac{2}{n-4}}L_{g_0}u^{\frac{n-2}{n-4}}= & \phi\lambda u^{\frac{2}{n-4}}(-\Delta u^{\frac{n-2}{n-4}}+\frac{n-2}{4(n-1)}R_{g_{0}}u^{\frac{n-2}{n-4}})\\
= & \phi\lambda u^{\frac{2}{n-4}}(-\frac{n-2}{n-4}u^{\frac{n-2}{n-4}-1}\Delta u \ensuremath{-\frac{2(n-2)}{(n-4)^{2}}u^{\frac{n-2}{n-4}-2}|\nabla u|^{2}+\frac{n-2}{4(n-1)}R_{g_{0}}u^{\frac{n-2}{n-4}})}\\
= & -\frac{n-2}{n-4}\lambda u^{\frac{4}{n-4}}\left(\Delta h-\Delta\phi\cdot u-2\langle\nabla u,\nabla\phi\rangle\right) 
\\&-\frac{2(n-2)}{(n-4)^{2}}\lambda(\frac{n-4}{4}\langle\nabla h,\nabla u^{\frac{4}{n-4}}\rangle-\langle\nabla u,\nabla\phi\rangle u^{\frac{4}{n-4}})\\
 & +\frac{n-2}{4(n-1)}R_{g_{0}}\lambda hu^{\frac{4}{n-4}}.
\end{align*}
Together with \eqref{eq:fourth order equation}, we obtain
\[
\Delta^{2}h+\frac{n-2}{n-4}\lambda u^{\frac{4}{n-4}}\Delta h+\frac{2(n-2)}{(n-4)^{2}}\lambda\frac{n-4}{4}\langle\nabla h,\nabla u^{\frac{4}{n-4}}\rangle=R,
\]
where we denote all the remaining terms as $R.$
As Uhlenbeck and Viaclovsky's argument, we have
\[
\Delta^{2}h=\Delta_{E}^{2}h+\sum_{i,j,k,l}\partial_{i}\partial_{j}\left(a^{ij,kl}\partial_{k}\partial_{l}w\right)+\text{lower order terms.}
\]
To prove the claim, we need to show that $R\in L_{-2}^{q}$ for
some $q>2$.

For any $q\le\frac{2n}{n-2}$, 
\begin{align*}
\|\Delta\langle\nabla\phi,\nabla u\rangle\|_{L_{-2}^{q}} & \le\|\langle\nabla\phi,\nabla u\rangle\|_{L^{q}}\\
 & \le(\int|\nabla\phi|^{q}|\nabla u|^{q})^{1/q} <\infty
\end{align*}
and for $q<\frac{2n}{n-2}$, 
\begin{align*}
\|\div(\nabla\phi\cdot\Delta u)\|_{L_{-2}^{q}} & \le\|\nabla\phi\cdot\Delta u\|_{L_{-1}^{q}}\\
 & \le(\int|\Delta u|^{\frac{nq}{n+q}})^{\frac{n+q}{nq}}<\infty,
\end{align*}
where in the last inequality we have used the fact that $q<\frac{2n}{n-2}\Leftrightarrow\frac{nq}{n+q}<2$.

Let $\frac{1}{q'}+\frac{1}{q}=1$. By the definition, we have
\begin{align*}
\|\langle\nabla u,\nabla\phi\rangle u^{\frac{4}{n-4}}\|_{L_{-2}^{q}} & =\sup_{\|\varphi\|_{W^{2,q'}}\le1}|\int\langle\nabla u,\nabla\phi\rangle u^{\frac{4}{n-4}}\varphi|\\
 & =\sup_{\|\varphi\|_{W^{2,q'}}\le1}\frac{n-4}{n}|\int\langle\nabla u^{\frac{n}{n-4}},\nabla\phi\rangle\varphi|\\
 & =\sup_{\|\varphi\|_{W^{2,q'}}\le1}\frac{n-4}{n}|\int u^{\frac{n}{n-4}}\langle\nabla\phi,\nabla\varphi\rangle+u^{\frac{n}{n-4}}\Delta\phi\cdot\varphi|\\
 & \le(\int u^{\frac{2n}{n-4}})^{\frac{1}{2}}(\int\langle\nabla\phi,\nabla\varphi\rangle^{2})^{\frac{1}{2}}+(\int u^{\frac{2n}{n-4}})^{\frac{1}{2}}(\int|\Delta\phi\cdot\varphi|^{2})^{\frac{1}{2}}\\
 & \le C(\int u^{\frac{2n}{n-4}})^{\frac{1}{2}}
\end{align*}
where in the last inequality we have used $q<\frac{2n}{n-2}$ and
$\int|\nabla\varphi|^{2}\le1$. 
Other lower order terms can be handled similarly. Thus, we obtain
the \emph{Claim}.

Now we begin to prove that 
\begin{equation}
\|L_1 h\|_{L_{-2}^{q}}+\frac{n-2}{n-4}\lambda\|u^{\frac{4}{n-4}}\Delta h\|_{L_{-2}^{q}}+\frac{2(n-2)}{(n-4)^{2}}\lambda\frac{n-4}{4}\|\langle\nabla h,\nabla u^{\frac{4}{n-4}}\rangle\|_{L_{-2}^{q}}\le C\epsilon\|h\|_{W^{2,q}}.\label{eq:small bound}
\end{equation}
The bound of $\|L_1 h\|_{L_{-2}^{q}}$ is obtained due to (\ref{eq:almost Euc}). For any $q>2$, $n-2q'>0$, 

\begin{align*}
\|u^{\frac{4}{n-4}}\Delta h\|_{L_{-2}^{q}} & =\sup_{\|\varphi\|_{W^{2,q'}}\le1}|\int u^{\frac{4}{n-4}}\Delta h\varphi|\\
 & \le(\int|\Delta h|^{q})^{\frac{1}{q}}(\int u^{\frac{4q'}{n-4}}\varphi^{q'})^{\frac{1}{q'}}\\
 & \le(\int|\Delta h|^{q})^{\frac{1}{q}}\left((\int u^{\frac{4q'}{n-4}\frac{2n}{4q'}})^{\frac{2q'}{n}}(\int\varphi^{q'\frac{n}{n-2q'}})^{\frac{n-2q'}{n}}\right)^{\frac{1}{q'}}\\
 & \le(\int u^{\frac{2n}{n-4}})^{\frac{2}{n}}(\int\varphi^{q'\frac{n}{n-2q'}})^{\frac{n-2q'}{nq'}}\|h\|_{W^{2,q}}\\
 & \le(\int u^{\frac{2n}{n-4}})^{\frac{2}{n}}\|h\|_{W^{2,q}},
\end{align*}
and 
\begin{align*}
 & \|\langle\nabla h,\nabla u^{\frac{4}{n-4}}\rangle\|_{L_{-2}^{q}}\\
  =&\sup_{\|\varphi\|_{W^{2,q'}}\le1}|\int\langle\nabla h,\nabla u^{\frac{4}{n-4}}\rangle\varphi|\\ =&\sup_{\|\varphi\|_{W^{2,q'}}\le1}|\int\Delta h\varphi u^{\frac{4}{n-4}}+\langle\nabla h,\nabla\varphi\rangle u^{\frac{4}{n-4}}|\\
  \le &(\int u^{\frac{2n}{n-4}})^{\frac{2}{n}}(\int|\Delta h|^{q})^{\frac{1}{q}}+(\int|\nabla h|^{\frac{nq}{n-q}})^{\frac{n-q}{nq}}\left(\int(u^{\frac{4}{n-4}}|\nabla\varphi|)^{\frac{nq}{nq-n+q}}\right)^{\frac{nq-n+q}{nq}}\\
\le &(\int u^{\frac{2n}{n-4}})^{\frac{2}{n}}(\int|\Delta h|^{q})^{\frac{1}{q}}\\
 & +(\int|\nabla h|^{\frac{nq}{n-q}})^{\frac{n-q}{nq}}\left((\int|\nabla\varphi|^{\frac{nq}{nq-n+q}\frac{nq-n+q}{nq-n-q}})^{\frac{nq-n-q}{nq-n+q}}(\int u^{\frac{4nq}{(n-4)(nq-n+q)}\frac{nq-n+q}{2q}})^{\frac{2q}{nq-n+q}}\right)^{\frac{nq-n+q}{nq}}\\
  \le &(\int u^{\frac{2n}{n-4}})^{\frac{2}{n}}\|h\|_{W^{2,q}}.
\end{align*}
By the assumptions, we get (\ref{eq:small bound}). The left argument
is the same as Uhlenbeck and Viaclovsky did in \cite{UV}.

\end{proof}

\begin{thm} \label{thm8.1} 
    Let $u\in W^{2,2}$ be a positive weak solution to 
\[
P_{g_0}u=\lambda u^{\frac{2}{n-4}}L_{g_0}u^{\frac{n-2}{n-4}}.
\]
Then $u\in C^{\infty}(M)$.
\end{thm}
\begin{proof}
Since the equation is scale invariant, the conditions in Lemma \ref{e-regularity}
can be obtained by rescaling. Now  applying  Lemma \ref{e-regularity}, we know that the rescaled solution belongs to $W^{2,q}$ for some $q>2$, and so does the original solution. Therefore, the higher regularity follows. 
\end{proof}

\subsection{Convexity}The following properties, used in Lemma \ref{lem:lower bound of u},  should be well known to experts. For completeness we provide the proof.

\begin{lem}\label{convex} {Let  $Lu=-\Delta u +au$, where $a$ is a function.}
If $Lu\ge 0$ and $Lv\ge 0$, then for $0\le s=1-t\le 1$, $L(u^tv^s)\ge 0$.
\end{lem}
 \begin{proof}

We compute the Laplacian of $u^tv^s$.
Using
\[\nabla (u^tv^s)=\{t \frac {\nabla u }u + s \frac {\nabla v} v \}u^tv^s,\]
and 
\begin{equation*}\begin{array}{rcl}\vspace{2mm}
	-\Delta (u^t v^s) &=& \ds \{ t \frac{ -\Delta u} u + t \frac {|\nabla u|^2} {u^2}
	 +s\frac{-\Delta v} v +s\frac{|\nabla v|^2 }{v^2}  \}	u^tv^s
     -|t\frac {\n u} u +s\frac {\n v} v |^2 u^t v^s,\\
     
     \end{array}
\end{equation*}
we have
\begin{equation*}\begin{array}{rcl}\vspace{2mm}
	L (u^t v^s) &=& \ds \{ t \frac{ L u} u + t \frac {|\nabla u|^2} {u^2}
	 +s\frac{L v} v +s\frac{|\nabla v|^2 }{v^2}  \}	u^tv^s 
     -|t\frac {\n u} u +s\frac {\n v} v |^2 u^t v^s\\
     \vspace{2mm}
	 &=& \ds \{ t \frac{ L u} u 
	 +s\frac{L v} v  \}	u^tv^s
     + \{t \frac {|\nabla u|^2} {u^2} + s \frac {|\nabla v|^2} {v^2}\} u^t v^s
     -|t\frac {\n u} u +s\frac {\n v} v |^2 u^t v^s\\
     \vspace{2mm}
     &\ge & \ds \{t(1-t) \frac {|\n u|^2} {u^2} +t(1-t) \frac{|\n v|^2} {v^2} -2t(1-t) \frac {\n u}u\cdot \frac{\n v} v  \}
u^tv^s
\\ &=& \ds t(1-t) \left|\frac{ \n u}u -\frac{\n v}{v} \right|^2
\ge 0.    \end{array}
\end{equation*}
 \end{proof}
 
\begin{cor}
     Let $g_u:=u^{\frac 4 {n-2}}g_0$. If
 $g_u$ and $g_v$ are two metrics of positive scalar curvature, then so is  $g_{u^sv^{1-s}}$ $(s\in (0,1))$.
\end{cor}

\medskip

\noindent{\bf Acknowledgement:} This work was carried out while
W. Wei was visiting University of Freiburg supported  by the Alexander von Humboldt research fellowship. She would like to thank Institute of Mathematics, University of Freiburg for its  hospitality. 
She was also
partly supported by NSFC (No.12571218, No. 12201288). Y. Ge is supported
in part by the grant ANR-23-CE40-0010-02 of the French National Research Agency (ANR): Einstein
constraints: past, present, and future (EINSTEIN-PPF). The second named author would like to thank Rupert Frank for helpful discussions.

\bibliography{bibYamabe}
\bibliographystyle{amsplain}

\end{document}